\documentclass[a4paper]{amsart}

\usepackage{amsmath,amssymb,amsthm,stmaryrd,mathrsfs}

\usepackage{color,graphicx}
\usepackage[dvipsnames]{xcolor}
\usepackage{accents}
\usepackage{enumerate}
\usepackage{varwidth}

\usepackage[textsize=footnotesize,color=yellow!40, bordercolor=white]{todonotes}

\newtheorem{theorem}{Theorem}[section]
\newtheorem{conjecture}[theorem]{Conjecture}
\newtheorem{lemma}[theorem]{Lemma}

\newcounter{cl}[theorem]
\newtheorem{claim}[cl]{Claim}

\newtheorem*{claim*}{Claim}
\newtheorem*{subclaim*}{Subclaim}
\newtheorem{corollary}[theorem]{Corollary}

\newtheorem{question}[theorem]{Question}

\newcommand{\thistheoremname}{}
\newtheorem*{genericthm}{\thistheoremname}
\newenvironment{namedthm}[1]
  {\renewcommand{\thistheoremname}{#1}%
   \begin{genericthm}}
  {\end{genericthm}}

\theoremstyle{definition}
\newtheorem{definition}[theorem]{Definition}
\newtheorem{terminology}[theorem]{Terminology}
\newtheorem{notation}[theorem]{Notation}
\newtheorem*{definition*}{Definition}
\newcounter{ca}[cl]

\newcounter{sca}[ca]

\newcounter{ssca}[sca]

\theoremstyle{remark}
\newtheorem{remark}[theorem]{Remark}
\newtheorem*{remark*}{Remark}

\usepackage{tikz}
\usetikzlibrary{decorations.pathmorphing,shapes}
\usetikzlibrary{positioning,arrows}

\usepackage[
    colorlinks=true, 
    citecolor=blue,
    linkcolor=blue,
    urlcolor=blue]{hyperref}
    
\newcommand{\ZFC}{\ensuremath{\operatorname{ZFC}} }

\newcommand{\Ord}{\sf{Ord}}

\newcommand{\rng}{\ensuremath{\operatorname{rng}} }
\newcommand{\cf}{\ensuremath{\operatorname{cf}} }

\newcommand{\AD}{\ensuremath{\operatorname{AD}} }

\newcommand{\Col}{\ensuremath{\operatorname{Col}} }
\newcommand{\lh}{\ensuremath{\operatorname{lh}} }
\newcommand{\str}{\ensuremath{\operatorname{str}} }
\newcommand{\Ult}{Ult}

\newcommand{\meas}{\ensuremath{\operatorname{meas}} }
\newcommand{\TW}{\ensuremath{\operatorname{TW}} }

\newcommand{\id}{\ensuremath{\operatorname{id}} }

\newcommand{\powerset}{{\wp}}
\newcommand{\cp}{{\rm cp }}

\newcommand{\crit}{\cp}
\newcommand{\rest}{\restriction}
    
\newcommand{\BS}{{}^\omega\omega}

\newcommand{\cT}{\mathcal{T}}

\newcommand{\bR}{\mathbb{R}}
\newcommand{\bP}{\mathbb{P}}

\newcommand{\comp}{{\sf{comp}}}
\newcommand{\sfb}{{\sf{b}}}
\newcommand{\sfa}{{\sf a}}
\newcommand{\sfd}{{\sf d}}
\newcommand{\sfe}{{\sf e}}

\newcommand{\gen}{{\sf gen}}

\newcommand{\card}[1]{{\vert #1 \vert} }

\newcommand{\Sealing}{\mathsf{Sealing}}
\newcommand{\GCSealing}{\mathsf{Generically}\ \allowbreak \mathsf{Correct}\ \allowbreak \mathsf{Sealing}}

\newcommand{\Hom}{\ensuremath{\operatorname{Hom}}}

\newcommand{\pwimg}{\, "}
\newcommand{\shortpwimg}{"}

\newcommand{\ThetaR}{\sf{\Theta reg}}

\def\k{\kappa}
\def\a{\alpha}
\def\b{\beta}
\def\d{\delta}

\def\l{\lambda}

\def\M{{\mathcal{M}}}

\def\T {{\mathcal{T}}}
\def\U{{\mathcal{U}}}

\def\X{{\mathcal{X}}}

\begin{document}
\title[Towards a generic absoluteness theorem for Chang models]{Towards a generic absoluteness theorem for Chang models}



\author{Sandra M\"uller} \address{Sandra M\"uller, Institut f\"ur Diskrete Mathematik und Geometrie, TU Wien, Wiedner Hauptstra{\ss}e 8-10/104, 1040 Wien, Austria.}
\email{sandra.mueller@tuwien.ac.at}
\author{Grigor Sargsyan} \address{Grigor Sargsyan, IMPAN, Antoniego Abrahama 18, 81-825 Sopot, Poland.}
\email{gsargsyan@impan.pl}

\thanks{The first author gratefully acknowledges that this research was funded in part by the Austrian Science Fund (FWF) [10.55776/Y1498, 10.55776/I6087, 10.55776/V844]. The second author's work is funded by the National Science Centre, Poland under the Weave-UNISONO call in the Weave program, registration number UMO-2021/03/Y/ST1/00281. For the purpose of open access, the authors have applied a CC BY public copyright license to any Author Accepted Manuscript version arising from this submission. Both authors would like to thank the referee for their helpful comments which improved the presentation of this paper.}




\begin{abstract}
  Let $\Gamma^\infty$ be the set of all universally Baire sets of reals. Inspired by the work done in \cite{SaTrConSealing} and \cite{SaTr21}, we introduce a new technique for establishing generic absoluteness results for models containing $\Gamma^\infty$. 
  
 Our main technical tool is an iteration that realizes $\Gamma^\infty$ as the sets of reals in a derived model of some iterate of $V$. We show, from a supercompact cardinal $\kappa$ and a proper class of Woodin cardinals, that whenever $g \subseteq \Col(\omega, 2^{2^\kappa})$ is $V$-generic and $h$ is $V[g]$-generic for some poset $\mathbb{P}\in V[g]$, there is an elementary embedding $j: V\rightarrow M$ such that $j(\kappa)=\omega_1^{V[g*h]}$ and $L(\Gamma^\infty, \bR)$ as computed in $V[g*h]$ is a derived model of $M$ at $j(\kappa)$. Here $j$ is obtained by iteratively taking ultrapowers of $V$ by extenders with critical point $\k$ and its images.
 
 As a corollary we obtain that $\Sealing$ holds in $V[g]$, which was previously demonstrated by Woodin using the stationary tower forcing. Also, using a theorem of Woodin, we conclude that the derived model of $V$ at $\kappa$ satisfies $\AD_{\mathbb{R}}+``\Theta$ is a regular cardinal".
 
  Inspired by core model induction, we introduce the definable powerset $\mathcal{A}^\infty$ of $\Gamma^\infty$ and use our derived model representation mentioned above to show that the theory of $L(\mathcal{A}^\infty)$ cannot be changed by forcing (see Theorem \ref{sealing for the ub powerset}). Working in a different direction, we also show that the theory of $L(\Gamma^\infty, \bR)[\mathcal{C}]$, where $\mathcal{C}$ is the club filter on $\powerset_{\omega_1}(\Gamma^\infty)$, cannot be changed by forcing (see Theorem \ref{thm: sealing for clubs}). Proving the two aforementioned results is the first step towards showing that the theory of $L(\Ord^\omega, \Gamma^\infty, \bR)([\mu_\a: \a\in \Ord])$, where $\mu_\a$ is the club filter on $\powerset_{\omega_1}(\a)$, cannot be changed by forcing.
\end{abstract}
\maketitle
\setcounter{tocdepth}{1}

\section{Introduction}

The research in this paper started by asking what sort of structures can exist on $\omega_3$ in forcing extensions of models of determinacy. This question is a local version of the more global question asking whether it is possible to force ${\sf{MM^{++}}}$ over a model of determinacy. The motivation behind our question comes from the following four results.

\begin{theorem}[Aspero-Schindler, \cite{AsSch21}]\label{asperoschindler} ${\sf{MM^{++}}}$ implies Woodin's ${\sf{Axiom\ (*)}}$.
\end{theorem}
Theorem \ref{asperoschindler} has a powerful consequence. It implies that the theorems of ${\sf{MM^{++}}}$ about the structure $(H_{\omega_2}, NS, \in)$, where $NS$ is the non-stationary ideal on $\omega_1$, are all true in a $\mathbb{P}_{max}$ extension of $L(\mathbb{R})$. Thus, studying $(H_{\omega_2}, NS, \in)$ under ${\sf{MM^{++}}}$ can be reduced to studying the internal structure of $L(\bR)$ under ${\sf{AD}}$. This is an incredible reduction of set theoretic complexity; currently all known methods for producing a model of ${\sf{MM}}$ use supercompact cardinals while the assumption $V=L(\bR)+{\sf{AD}}$ requires only $\omega$ many Woodin cardinals \cite{St02}. 

Also, Theorem \ref{asperoschindler} has the following remarkable corollary.
\begin{corollary}\label{cortoaspsch} The theory of $L(\omega_2^{\omega_1})$ conditioned on ${\sf{MM^{++}}}$ cannot be changed by forcing.  More precisely, assuming ${\sf{MM^{++}}}$, whenever $g$ is $V$-generic such that $V[g]\models {\sf{MM^{++}}}$, $(L(\omega_2^{\omega_1}))^{V[g]}$ is elementarily equivalent to $L(\omega_2^{\omega_1})$.
\end{corollary} 
This is because it follows from Theorem \ref{asperoschindler} that, assuming ${\sf{MM^{++}}}$, the theory of $L(\omega_2^{\omega_1})$ is coded into $L(\bR)$, whose theory, provided the minimal canonical inner model with $\omega$ Woodin cardinals exists, cannot be changed by forcing. In other words, studying $L(\omega_2^{\omega_1})$ under ${\sf{MM^{++}}}$ reduces to the study of $V=L(\bR)+{\sf{AD}}$. The next theorem is a remarkable generalization of Corollary \ref{cortoaspsch}.

\begin{theorem}[Viale, \cite{Vi16}]\label{viale} Assuming class of $\Sigma_2$-reflecting cardinals, the theory of $L({\sf{Ord}}^{\omega_1})$ conditioned on ${\sf{MM^{+++}}}$ cannot be changed by stationary set preserving forcing. More precisely, assuming ${\sf{MM^{+++}}}$ and the existence of a class of $\Sigma_2$-reflecting cardinals, whenever $g$ is a set generic over $V$, for a stationary set preserving forcing such that $V[g]\models {\sf{MM^{+++}}}$, $(L({\sf{Ord}}^{\omega_1}))^{V[g]}$ is elementarily equivalent to $L({\sf{Ord}}^{\omega_1})$.
\end{theorem}

Currently, there is no reduction of the study of $L({\sf{Ord}}^{\omega_1})$ under ${\sf{MM^{+++}}}$ to the study of some definable universe, and the research in this paper is partly motivated by the question whether there can be such a reduction. The next two theorems show that it is possible to force fragments of ${\sf{MM^{++}}}$ over a model of determinacy.

\begin{definition}\label{def:ThetaReg} ${\sf{AD_{\bR}}}$ is the statement that every set $A\subseteq \bR^\omega$ is determined. $\ThetaR$ is the theory ${\sf{ZF+AD_{\bR}}}+``\Theta$ is a regular cardinal".
\end{definition}
 Below $\sf{c}$ is the continuum and ${\sf{MM^{++}(\l)}}$ is ${\sf{MM^{++}}}$ for posets of size $\l$. 
\begin{theorem}[Woodin, \cite{WoPmax}]\label{woodins mmc} Assume $V=L(\powerset(\bR))$ and $\ThetaR$ holds. Then if $G\subseteq \mathbb{P}_{max}*Add(1, \omega_3)$ is $V$-generic then $V[G]\models {\sf{MM^{++}(c)}}$.
\end{theorem}

\begin{theorem}[Larson-Sargsyan, \cite{NairianModels,SaAnnouncement}]\label{larsonsarg} There is a canonical, definable transitive model $M$ satisfying $\ThetaR$ such that if $G\subseteq \mathbb{P}_{max}*Add(1, \omega_3)*Add(1, \omega_4)$ is $M$-generic then $M[G]\models {\sf{MM^{++}(c)}}+\neg\square(\omega_3)+\neg\square_{\omega_3}$. 
\end{theorem}

The conjunction ${\sf{MM^{++}(c)}}+\neg\square(\omega_3)+\neg\square_{\omega_3}$ is a consequence of ${\sf{MM}^{++}}(c^{++})$, and it is currently the strongest known consequence of ${\sf{MM}^{++}}$ that can be forced over a model of determinacy. 
The ground model $M$ of Theorem \ref{larsonsarg} is a type of Chang model over $\powerset(\bR)$. Its exact form is $M=L(A, \l^\omega, \powerset(\bR))$ where $\l=(\Theta^+)^M$ and $A\subseteq \l$. It can be easily shown that $M\not= L(\powerset(\bR)^M)$ and so it is a new type of canonical model of determinacy, and it is also exactly the reason why we, in this paper, are concerned with Chang models over $\Gamma^\infty$. The key shortcoming of Theorem \ref{larsonsarg} is that $M$ is constructed assuming very stringent inner model theoretic assumptions; $M$ is constructed inside a hod mouse. While inner model theory serves as motivation and inspiration, the results in this paper do not require any deep inner model theoretic machinery and are hence accessible to a wide audience with some general set theoretic background.

The research in this paper started by asking the following questions.\\

\begin{description}
\item[(Q1)] Is it possible to construct models like $M$ of Theorem \ref{larsonsarg} assuming just large cardinals, without any further inner model theoretic assumptions?\\
\item[(Q2)] Is it possible to reduce Viale's result to a generic absoluteness principle for some model of determinacy, such as those constructed to answer (Q1)?\\
\item[(Q3)] Is it possible to reduce the study of $L(\omega_3^{\omega_1})$ under ${\sf{MM^{++}}}$ to the study of some generically invariant, canonical model of determinacy, such as those constructed to answer (Q1)?\\
\end{description}

To make the above questions mathematically more precise, we need to introduce \textit{universally Baire sets}. Before we do that, we briefly state our main new results for readers who are already familiar with the background and notation and want to skip the introduction.


\subsection{Main results}
We show that $\sf{Generically\ Correct\ Sealing}$, a strengthening of $\Sealing$ (see Lemma \ref{lem:gcsealingsealing}) inspired by the notion of generically correct formula (see Definition \ref{def:gcsealing}), holds in a generic extension collapsing large cardinals. For the proof see Section \ref{sec:GCSealing}.

\begin{namedthm}{Theorem \ref{thm:GCSealing}}
    Let $\kappa$ be a supercompact cardinal and let $g$ be $\Col(\omega, 2^{2^\kappa})$-generic over $V$. Suppose that there is a proper class of Woodin cardinals. Then $\GCSealing$ holds in $V[g]$.
\end{namedthm}

Moreover, after collapsing large cardinals, we show $\sf{Weak\ Sealing}$, i.e., $\Sealing$ for a cofinal set of generic extensions instead of every generic extension (see \ref{def: sealing for the ub powerset}), for $L(\mathcal{A}^\infty)$, the model constructed over the uB-powerset $\mathcal{A}^\infty=\powerset_{uB}(\Gamma^\infty)$. For the proof see Section \ref{sec: sealing for ub powerset}.

\begin{namedthm}{Theorem \ref{sealing for the ub powerset}} Suppose $\kappa$ is a supercompact cardinal and there is a proper class of inaccessible limits of Woodin cardinals. Suppose $g\subseteq \Col(\omega, 2^{2^\kappa})$ is $V$-generic. Then $\sf{Weak\ Sealing}$ for the uB-powerset holds in $V[g]$. 
\end{namedthm}

Finally, we consider another extension of the model $L(\Gamma^\infty, \bR)$ and prove $\Sealing$ for $L(\Gamma^\infty, \bR)[\mathcal{C}^\infty]$, where $\mathcal{C}^\infty$ is the club filter on $\powerset_{\omega_1}(\Gamma^\infty)$, in a generic extension collapsing large cardinals. For the proof see Section \ref{sec: sealing for clubs} and for the relevant definitions see Definition \ref{def:huge cardinal}.

\begin{namedthm}{Theorem \ref{thm: sealing for clubs}} Suppose $\kappa$ is an elementarily huge cardinal and that there is a proper class of  Woodin cardinals. Let $g\subseteq \Col(\omega, 2^{{\sf{ehst}}(\k)})$ be $V$-generic. Then the $\sf{Sealing\ Theorem}$ for $L(\Gamma^\infty, \bR)[\mathcal{C}^\infty]$ holds in $V[g]$. 
 \end{namedthm}

Moreover, in the setting of Theorems \ref{sealing for the ub powerset} and \ref{thm: sealing for clubs} we obtain that $\Theta$ is regular, see \cite{MuSa_ThetaReg}.

In the remaining part of the introduction, we motivate these results and introduce the relevant concepts. Sections \ref{sec:preliminaries} and \ref{sec:preservinguB} introduce more preliminary notions and results before we prove the main derived model representation (see Theorem \ref{thm:newmain}) in Section \ref{sec:summarydmrep}, which summarizes the derived model representation, outlining the properties that will be used in the applications at the end of this paper. In Section \ref{sec:SealingTheorem} we obtain Woodin's Sealing Theorem (see Theorem \ref{thm:SealingWoodin}) as an application of the derived model representation and in Section \ref{sec:sealing for derived models} we discuss corollaries about $\Sealing$ for derived models. The main new results are proven in Sections \ref{sec:GCSealing}, \ref{sec: sealing for ub powerset}, and \ref{sec: sealing for clubs}, using the methodology introduced in Section \ref{sec:summarydmrep}.

 \subsection{The universally Baire sets and the model $L(\Gamma^\infty, \bR)$.} Universally Baire sets originate in work of Schilling and Vaught \cite{SchVa83}, and they were first systematically studied by Feng, Magidor, and Woodin \cite{FMW92}. Since then they play a prominent role in many areas of set theory. Recall that a set of reals is universally Baire if all of its continuous preimages in compact Hausdorff spaces have the property of Baire.  However, the following equivalent definition is set theoretically more useful. 

\begin{definition}[Feng-Magidor-Woodin, \cite{FMW92}]\label{def:uB}
\vspace{0.5em}\begin{enumerate}\itemsep0.5em
    \item  Let $(S,T)$ be trees on $\omega \times \kappa$ for some ordinal $\kappa$ and let $Z$ be any set. We say \emph{$(S,T)$ is $Z$-absolutely complementing} iff \[ p[S] = \BS \setminus p[T] \] in every $\Col(\omega,Z)$-generic extension of $V$.
    \item  A set of reals $A$ is \emph{universally Baire (uB)} if for every $Z$, there are $Z$-absolutely complementing trees $(S,T)$ with $p[S] = A$.
\end{enumerate}\vspace{0.5em}
\end{definition}

Following Woodin, we write $\Gamma^\infty$ for the set of universally Baire sets and, if $g$ is set generic over $V$, we write $\Gamma_g^\infty = (\Gamma^\infty)^{V[g]}$ and $\bR_g = \bR^{V[g]}$ where $\bR$ denotes the set of reals. Given a universally Baire set $A$ and a set generic $g\subseteq \mathbb{P}$ for some poset $\bP$, we write $A_g$ for the interpretation of $A$ in $V[g]$. More precisely, letting $\kappa\geq \card{\mathbb{P}}^+$ be a cardinal and $(S, T)$ be any $\kappa$-absolutely complementing trees with $p[S]=A$, we set $A_g=(p[S])^{V[g]}$. It can be easily checked using the absoluteness of well-foundedness that $A_g$ is independent of $(T, S)$. The main object this paper studies is the model $L(\Gamma^\infty, \bR)$, and our aim is to find a useful derived model representation for it\footnote{By a theorem of Woodin, all models of ${\sf{AD^+}}$ are elementary equivalent to a model of ${\sf{AD^+}}$ that can be obtained as a derived model. This refers to the so-called ``new'' derived model of the form $L(\Gamma,\bR^*)$ for $V(\bR^*)$ a $\Col(\omega, {<}\lambda)$-symmetric  extension of $V$ for $\lambda$ a limit of Woodin cardinals in $V$ and $\Gamma$ the set of all $A \subseteq \bR^*$ such that $L(A,\bR^*) \models \AD^+$ and $A \in V(\bR^*)$. The main focus in this paper is the so-called ``old'' derived model as defined in Section \ref{sec:preliminaries}.}.

It is known that sets of reals that are definable via sufficiently generically absolute formulas are universally Baire (e.g., see \cite[Lemma 4.1]{St09}). In this sense, the generically absolute fragment of ${\sf{ZFC}}$ is coded into the universally Baire sets, and in this sense, the model $L(\Gamma^\infty, \bR)$, as far as the sets of reals go, must be the maximal generically absolute inner model. However, the story of $L(\Gamma^\infty, \bR)$ is not as simple as one might guess.

Assuming a proper class of Woodin cardinals, the study of the model $L(A, \mathbb{R})$, where $A$ is a universally Baire set, is completely parallel to the study of $L(\mathbb{R})$, and most of the theorems proven for $L(\mathbb{R})$ can be easily generalized to $L(A, \mathbb{R})$. For example, assuming a proper class of Woodin cardinals, generalizing the proof for $L(\mathbb{R})$, Woodin showed that the theory of the model $L(A, \mathbb{R})$ cannot be changed by forcing, and in fact, for every  $V$-generic $g$ and $V[g]$-generic $h$, there is an elementary embedding 
\begin{center} $j:L(A_g, \mathbb{R}_g)\rightarrow L(A_{g*h}, \mathbb{R}_{g*h})$
\end{center}
such that $j\restriction \mathbb{R}_g=id$ and $j(A_g)=A_{g*h}$, and also $L(A, \bR)\models {\sf{AD^+}}$. 

However, the model $L(\Gamma^\infty, \bR)$ is much harder to analyze. For example, it is not clear that assuming large cardinals, $L(\Gamma^\infty, \bR)\models {\sf{AD^+}}$ or even 
\begin{center}
    $\powerset(\mathbb{R})\cap L(\Gamma^\infty, \bR)=\Gamma^\infty$.
\end{center}
Woodin's $\Sealing$ deals with the aforementioned issues.

\begin{definition}[Woodin]\label{def:Sealing}
$\Sealing$ is the conjunction of the following statements.
\vspace{0.5em}\begin{enumerate}\itemsep0.5em
    \item For every set generic $g$ over $V$, $L(\Gamma^\infty_g, \bR_g) \models \AD^+$ and $\powerset(\bR_g) \cap L(\Gamma^\infty_g, \bR_g) = \Gamma^\infty_g$. \label{eq:Sealing1}
    \item For every set generic $g$ over $V$ and set generic $h$ over $V[g]$, there is an elementary embedding \[ j \colon L(\Gamma^\infty_g, \bR_g) \rightarrow L(\Gamma^\infty_{g*h}, \bR_{g*h}) \] such that for every $A \in \Gamma^\infty_g$, $j(A) = A_h$. \label{eq:Sealing2}
\end{enumerate}\vspace{0.5em}
\end{definition}

Using the stationary tower forcing, Woodin showed that $\Sealing$ is consistent from a supercompact cardinal and a proper class of Woodin cardinals. More precisely, he showed the following theorem.

\begin{theorem}[Woodin's Sealing Theorem, see \cite{La04,WoodinLongExtender}]\label{thm:SealingWoodin}
Let $\kappa$ be a supercompact cardinal and let $g$ be $\Col(\omega, 2^\kappa)$-generic over $V$. Suppose that there is a proper class of Woodin cardinals. Then $\Sealing$ holds in $V[g]$.
\end{theorem}

Woodin, in addition, showed that in this setting $\Theta$ is regular in $L(\Gamma^\infty_g, \bR_g)$, see \cite{WoodinLongExtender, MuSa_ThetaReg}. Theorem \ref{thm:SealingWoodin} implies that in $V[g]$, $L(\Gamma^\infty, \bR)$ is a canonical model whose theory cannot be changed by forcing. We show that the reason for this is that $L(\Gamma^\infty, \bR)$ can be realized as a derived model at a supercompact cardinal via Woodin's Derived Model Theorem (see Theorem \ref{woodin: der model thm}). More specifically we prove the following theorem ($\Hom^*$ is defined in Section \ref{sec:preliminaries}) which implies Theorem \ref{thm:SealingWoodin} for $g \subseteq \Col(\omega, 2^{2^\kappa})$.

\begin{theorem}\label{thm:newmain}
    Let $\kappa$ be a supercompact cardinal and suppose there is a proper class of Woodin cardinals. Let $g \subseteq \Col(\omega, 2^\kappa)$ be $V$-generic, $h$ be $V[g]$-generic and $k\subseteq \Col(\omega, 2^\omega)$ be $V[g*h]$-generic. Then, in $V[g*h*k]$, there is $j: V\rightarrow M$ such that $j(\kappa)=\omega_1^{M[g*h]}$ and  $L(\Gamma^\infty_{g*h}, \bR_{g*h})$ is a derived model of $M$, i.e., for some $M$-generic $G \subseteq \Col(\omega, {<}\omega_1^{V[g*h]})$,
        \[ L(\Gamma^\infty_{g*h}, \bR_{g*h}) = (L(\Hom^*, \bR^*))^{M[G]}. \]
\end{theorem}
In addition, in the first part of this paper that concerns $\sf{Sealing}$, we show that several related corollaries obtained by Woodin using the stationary tower forcing, follow from Theorem \ref{thm:newmain} (or its proof). First, we observe that a form of $\Sealing$ for derived models follows from $\Sealing$ (see Corollary \ref{cor:newmain} and Corollary \ref{cor:diagram}). Next we show that $\sf{Generically\ Correct\ Sealing}$ holds in $V[g]$ (assuming the hypothesis of Theorem \ref{thm:newmain}), which is a generalization of $\Sealing$ inspired by the notion of generically correct formula (see Section \ref{sec:GCSealing}, Definition \ref{def:gcsealing}). $\sf{Generically\ Correct\ Sealing}$ implies $\Sealing$ (see Lemma \ref{lem:gcsealingsealing}). 

We end this section by paraphrasing (Q1) above using $\Gamma^\infty$.\\

\begin{description}
\item[(Q4)] What sort of canonical, generically absolute subsets of $\Gamma^\infty$ can be added to the model $L(\Gamma^\infty, \bR)$?\\
\end{description}

In the next section we propose a model that resembles $M$ of Theorem \ref{larsonsarg}.

\subsection{Forcing ${\sf{MM^{++}}}$ over determinacy.} Motivated by Theorems \ref{asperoschindler}, \ref{viale}, \ref{woodins mmc} and \ref{larsonsarg}, and especially by the form of the model $M$ used in \ref{larsonsarg}, we isolate our first set of test questions.

\begin{definition}\label{def: ax} ${\sf{Ax(Ord, \omega_1)}}'$ is the statement ``there is  a definable transitive model of determinacy $M$ such that
\begin{itemize}
\item ${\sf{Ord}}, \bR\subseteq M$, 
\item there is a poset $\mathbb{P}\subseteq M$ that is a definable class of $M$ and is such that $M\models ``\mathbb{P}$ is homogeneous\footnote{Here, we mean that the automorphisms themselves are definable classes.} and countably closed", and
\item there is an $M$-generic $G\subseteq \mathbb{P}$ such that ${\sf{Ord}}^{\omega_1}\subseteq M[G]$".
\end{itemize}
We let $\sf{Ax(Ord, \omega_1)}$ be the conjunction of the following statements:
\begin{itemize}
\item Letting $M=_{def}L({\sf{Ord}}^\omega, \Gamma^\infty, \bR)[(\mu_\a:\a\in Ord)]$ where $\mu_\a$ is the club filter on $\a^\omega$, $M\models {\sf{AD^+}}$.
\item Letting $\mathbb{P}=\mathbb{P}_{max}*\mathbb{Q}$ where $\mathbb{Q}\in M^{\mathbb{P}_{max}}$ is the full support iteration for adding a Cohen subset to each cardinal $>\omega_2$, there is an $M$-generic $G\subseteq \mathbb{P}$ such that ${\sf{Ord}}^{\omega_1}\subseteq M[G]$. 
\end{itemize}
We let for $n\geq 3$, $\mathsf{Ax}(\omega_n, \omega_1)$ be the conjunction of the following statements:
\begin{itemize}
\item Letting $M=_{def}L({\sf{Ord}}^\omega, \Gamma^\infty, \bR)[(\mu_\a: \a \in Ord)]$ where $\mu_\a$ is the club filter on $\a^\omega$, $M\models {\sf{AD^+}}$.
\item Letting $\mathbb{P}=\mathbb{P}_{max}*\mathbb{Q}$ where $\mathbb{Q}\in M^{\mathbb{P}_{max}}$ is the finite iteration for adding a Cohen subset to $\omega_3, \omega_4,..., \omega_n$, there is an $M$-generic $G\subseteq \mathbb{P}$ such that ${\omega_n}^{\omega_1}\subseteq M[G]$. 
\end{itemize}
\end{definition}
The version of the Chang models introduced above that are built over the $\emptyset$ instead of $\Gamma^\infty$ were studied by Woodin. For example, \cite{La04} shows that assuming large cardinals the theory of the Chang model cannot be changed by forcing. Here is our first set of test questions. 

\begin{question}\label{the mm question} Does ${\sf{MM^{++}}}$ imply $\sf{Ax(Ord, \omega_1)}'$, $\sf{Ax(Ord, \omega_1)}$ or $\forall n<\omega(\mathsf{Ax}(\omega_n, \omega_1))$?
\end{question}

The issue with the models introduced above is that they seem too simple. For example it is not clear that all canonical inner models that can be defined over $\Gamma^\infty$, the object that is usually called ${\sf{Lp}}(\Gamma^\infty)$ in inner model theory, is contained in $L({\sf{Ord}}^\omega, \Gamma^\infty, \bR)[(\mu_\a: \a \in Ord)]$. The following conjecture makes this issue more precise.

\begin{conjecture}\label{mmsharpchang} Assume ${\sf{MM^{++}}}$. Let ${\sf{Lp}}(\Gamma^\infty)$ be the union of all mice $\M$ over $\Gamma^\infty$ such that $\M$ is $\Gamma^\infty$-sound, there is a set $A\in \powerset(\Gamma^\infty)$ such that $A$ is definable over $\M$ with parameters and $A\not \in \M$, and the countable substructures of $\M$ have an $\omega_1+1$-iteration strategy. Then \begin{center}${\sf{Lp}}(\Gamma^\infty)\not \subseteq L({\sf{Ord}}^\omega, \Gamma^\infty, \bR)[(\mu_\a: \a \in Ord)]$.\end{center}
\end{conjecture}

The reason behind the conjecture is that, for example, there are several theorems in set theory that establish that the sharp of the Chang model is small (e.g., see \cite{Wo23}, \cite{Mi17}, and \cite{GaSa23}). 

Nevertheless, recalling that Theorem \ref{cortoaspsch} reduces the study of $L(\omega_2^{\omega_1})$ under ${\sf{MM^{++}}}$ to just the study of $L(\bR)$ under ${\sf{AD}}$, it is possible that Question \ref{the mm question} has a positive answer. Moreover, it follows from an unpublished theorem of the second author that the consistency strength of the hypothesis of Theorem \ref{larsonsarg} is weaker than a Woodin cardinal that is a limit of Woodin cardinals. While currently it is widely believed that ${\sf{MM^{++}}}$ is, consistency wise, as strong as a supercompact cardinal, we do not even know if it implies the existence of an iterable mouse with a Woodin cardinal that is a limit of Woodin cardinals. Given the unpublished results of Woodin on the model $M=_{def}L({\sf{Ord}}^\omega, \Gamma^\infty, \bR)[(\mu_\a: \a \in Ord)]$ and the results of \cite{Mi17} and \cite{GaSa23}, it is plausible that if $\M$ is the minimal active mouse over $\Gamma^\infty$ that has a Woodin cardinal that is a limit of Woodin cardinals  then $\M\not \in M$. Therefore, answering the following questions seems very important. For each $X$, let $\M_{wlw}(X)$, if exists, be the minimal active $\omega_1+1$-iterable mouse over $X$ that has a Woodin cardinal which is a limit of Woodin cardinals.

\begin{question}\label{questionwlw} Does ${\sf{MM^{++}}}$ imply that $\M_{wlw}(\emptyset)$ exists? Does ${\sf{MM^{++}}}$ imply that for every $X$, $\M_{wlw}(X)$ exists?
\end{question}

It was observed by Nam Trang that the methods of \cite{NeSt16} give a positive answer to Question \ref{questionwlw} if one in addition to ${\sf{MM^{++}}}$ assumes the existence of a proper class of Woodin cardinals. In the next section we introduce the first extension of $L(\Gamma^\infty, \bR)$ that we will study in this paper. 

\subsection{The ordinal definable powerset of $\Gamma^\infty$} 

We introduce a canonical set of subsets of $\Gamma^\infty$, the ordinal definable powerset of $\Gamma^\infty$.

\begin{definition}\label{ub powerset} Suppose $X$ is a set. We let $\iota_X=\max(\card{X}, \card{\Gamma^\infty})$ and $\powerset_{uB}(X)$ be the set of those $Y\subseteq X$ such that whenever $g\subseteq \Col(\omega, \iota_X)$ is $V$-generic, $Y$ is ordinal definable in $L(\Gamma^\infty_g, \bR_g)$ from parameters belonging to the set $\{X, j_g \shortpwimg \Gamma^\infty \}\cup j_g \shortpwimg \Gamma^\infty$, where $j_g: L(\Gamma^\infty, \bR)\rightarrow L(\Gamma^\infty_g, \bR_g)$ is the canonical embedding with the property that $j_g(A)=A_g$ for every  $A \in \Gamma^\infty$.   
\end{definition}

 The following is an easy consequence of $\sf{Sealing}$.
    \begin{lemma}\label{easy consequence} Assume $\sf{Sealing}$ and suppose $Y\in \powerset_{uB}(X)$. Suppose $h$ is any $V$-generic with the property that there is $k\in V[h]$ which is $V$-generic for $\Col(\omega, \iota_X)$. Then $Y$ is ordinal definable in $L(\Gamma^\infty_h, \bR_h)$ from parameters that belong to the set $\{X, j_h \shortpwimg \Gamma^\infty \}\cap j_h \shortpwimg \Gamma^\infty$.
    \end{lemma}
    The lemma holds because we have that if $k\in V[h]$ is $V$-generic for $\Col(\omega, \iota_X)$ and $h'$ is $V[k]$-generic such that $V[h]=V[k*h']$\footnote{By the general forcing theory, for some $\l$, $V[k]$ has a uniform $\l$-covering property in $V[h]$ and so by Bukovsky's Theorem, there is a desired $h'$. See \cite[Theorem 3.5]{SaSch18}.} then $j_{k, h'}(X)=X$, $j_{k, h'} \shortpwimg j_k \shortpwimg \Gamma^\infty=j_h \shortpwimg \Gamma^\infty$ and $j_{k, h'}(j_k \shortpwimg \Gamma^\infty)=j_h \shortpwimg \Gamma^\infty$ (since $j_k \shortpwimg \Gamma^\infty$ is countable in $L(\Gamma^\infty_k, \bR_k)$, the last equality is a consequence of the second equality).

One crucial consequence of $\sf{Sealing}$ is that for every infinite set $X$, $\card{\powerset_{uB}(X)}<\card{X}^+$.\footnote{This uses the fact that there is no uncountable sequence of pairwise distinct reals in $L(\Gamma^\infty_g, \bR_g)$.} In \cite{SaTrConSealing}, this fact was used to argue that $\sf{Sealing}$ is already problematic for inner model theory in the short extender region. Clearly, the ordinary inner model theoretic operators over $X$, such as $X^\#$, $\mathcal{M}_1^\#(X)$ and etc, are in $\powerset_{uB}(X)$\footnote{We need to code these operators as subsets of $X$ in order to have them inside $\powerset_{uB}(X)$.}. Intuitively, $\powerset_{uB}(X)$ is the union or, as is customary in core model induction terminology, the stack of inner model theoretic operators over $X$, and in core model induction literature, in the presence of Mouse Capturing, $\powerset_{uB}(X)$ is usually represented as $Lp^{j_g \shortpwimg \Gamma^\infty}(X)$. 

The following is one of the main theorems of this paper. First let $\mathcal{A}^\infty=\powerset_{uB}(\Gamma^\infty)$ be the uB-powerset. If $g$ is $V$-generic then we let $\mathcal{A}^\infty_g=(\mathcal{A}^\infty)^{V[g]}$. If $\sf{Sealing}$ holds and $g$, $g'$ are two consecutive generics\footnote{I.e. $g'$ is $V[g]$-generic.}, then we let
\begin{center}
    $j_{g, g'}:L(\Gamma^\infty_g, \bR_g)\rightarrow L(\Gamma^{\infty}_{g*g'}, \bR_{g*g'})$
\end{center}
be the canonical $\sf{Sealing}$ embedding, i.e., $j_{g, g'}(A)=A_{g'}$.

\begin{definition}\label{def: sealing for the ub powerset} We say $\sf{Weak\ Sealing}$ holds for the uB-powerset\footnote{This is different from the Weak Sealing notion in \cite{WoodinLongExtender}.} if 
\vspace{0.5em}\begin{enumerate}\itemsep0.5em
    \item $\sf{Sealing}$ holds, \label{eq: sealing for the ub powerset clause 1}
    \item $L(\mathcal{A}^\infty)\models {\sf{AD}}^+$, \label{eq: sealing for the ub powerset clause 2}
    \item whenever $g, g'$ are two consecutive generics such that $V[g*g']\models \card{(2^{2^\omega})^{V[g]}}=\aleph_0$, there is an elementary embedding $\pi:L(\mathcal{A}^\infty_g)\rightarrow L(\mathcal{A}^\infty_{g*g'})$ such that $\pi\restriction L(\Gamma^\infty_g, \bR_g)=j_{g, g'}$. \label{eq: sealing for the ub powerset clause 3}
\end{enumerate}\vspace{0.5em}
We say $\sf{Sealing}$ holds for the uB-powerset if Clause \eqref{eq: sealing for the ub powerset clause 3} above holds for all consecutive generics $g$ and $g'$.
\end{definition}
\begin{remark}
If $\sf{Weak\ Sealing}$ holds for the uB-powerset then the theory of the model $L(\mathcal{A}^\infty)$ cannot be changed by forcing. This is because if $g$ and $g'$ are two consecutive generics and $W$ is a generic extension of both $V[g]$ and $V[g*g']$ such that $W\models \card{\powerset(\powerset(\bR))^{V[g]}}=\card{\powerset(\powerset(\bR))^{V[g*g']}}=\aleph_0$ then both $L(\mathcal{A}^\infty_g)$ and $L(\mathcal{A}^\infty_{g*g'})$ are elementarily equivalent to $L(\mathcal{A}^\infty)^W$.
\end{remark}

\begin{theorem}\label{sealing for the ub powerset} Suppose $\kappa$ is a supercompact cardinal and there is a proper class of inaccessible limits of Woodin cardinals. Suppose $g\subseteq \Col(\omega, 2^{2^\kappa})$ is $V$-generic. Then $\sf{Weak\ Sealing}$ for the uB-powerset holds in $V[g]$. 
\end{theorem}

Theorem \ref{sealing for the ub powerset} implies that $L(\mathcal{A}^\infty)$ is a canonical model of determinacy. Moreover, in the setting of Theorem \ref{sealing for the ub powerset}, $\Theta$ is regular in $L(\mathcal{A}^\infty)$ by the results in \cite{MuSa_ThetaReg}.

\begin{question} Does ${\sf{Sealing}}$ for the uB-powerset hold in the model of Theorem \ref{sealing for the ub powerset} or in any other model? 
\end{question}

\begin{remark} We conjecture that $\mathcal{A}^\infty$ is the amenable powerset of $\Gamma^\infty$. More precisely, $\mathcal{A}^\infty$ is the set of all $A\subseteq \Gamma^\infty$ such that for every $\k<\Theta^{L(\Gamma^\infty, \bR)}$, letting $\Gamma_\k=\{ B\in \Gamma^\infty \mid$ the Wadge rank of $B$ is $<\k\}$, $A\cap \Gamma_\k\in L(\Gamma^\infty, \bR)$. The proof of this conjecture depends on the proof of Conjecture \ref{cofinality conjecture}.
\end{remark}

\subsection{Forcing ${\sf{MM^{++}}}$ over determinacy revisited.}

 Let $\Theta^\infty=\Theta^{L(\Gamma^\infty, \bR)}$. It is unlikely that ${\sf{MM^{++}}(c^{++})}$ can be forced over the model $L(\mathcal{A}^\infty)$. This is because if, for example, some generalized version of the ${\sf{Mouse\ Set\ Conjecture}}$\footnote{See \cite{Sa09, Sa15}.} is true then $\mathcal{A}^\infty$ is just a mouse that satisfies $\square_{\Theta^\infty}$. 
 \begin{conjecture} $L(\mathcal{A}^\infty)\models \square_{\Theta^\infty}$.
 \end{conjecture}
 We need to extend the model $L(\mathcal{A}^\infty)$ further, and here the clue comes from the model $M$ used in Theorem \ref{larsonsarg}. The next step is to consider the model $L(\powerset_{\omega_1}(\mathcal{A}^\infty))$ and prove a ${\sf{Sealing}}$ theorem for it. For the model to make sense, however, we need to know that $\sup (Ord\cap \mathcal{A}^\infty)$ has cofinality $\omega$. 
 
\begin{definition}\label{thetax} Suppose $X$ is a set. Let $\eta_X$ be the supremum of all ordinals $\a$ such that there is a surjective $f \colon \powerset(\bR) \rightarrow \a$ which is ordinal definable from $X$.
\end{definition}

\begin{definition}\label{def: wst} We say that the $\sf{Weak\ Sealing\ Theorem}$ for a formula $\phi$ holds if setting $M_\phi=\{ x: V\models \phi[x]\}$, the following statements hold:
\vspace{0.5em}\begin{enumerate}\itemsep0.5em
\item $M_\phi$ is a transitive model of ${\sf{AD^+}}$.\label{cl:1_def:wst}
\item For all $V$-generics $g$, the models $(M_\phi)^{V}$ and $(M_\phi)^{V[g]}$ are elementarily equivalent.\label{cl:2_def:wst}
\end{enumerate}\vspace{0.5em} 
We often drop $\phi$ from our usage, and often we will simply say that the $\sf{Weak}$ ${\sf{Sealing}}$ ${\sf{Theorem}}$ for $M_\phi$ holds to mean that the $\sf{Weak\ Sealing\ Theorem}$ holds for $\phi$. Similarly, we define the meaning of ${\sf{Sealing}}$ $\sf{Theorem}$ for $\phi$. This means that  
\vspace{0.5em}\begin{enumerate}\itemsep0.5em
\item $M_\phi$ is a transitive model of ${\sf{AD^+}}$ containing all the reals.\label{cl:1_def:st}
\item for all $V$-generic $g$, $V[g]\models \powerset(\bR)\cap M_\phi \subseteq \Gamma^\infty$.\label{cl:2_def:st}
\item For all $V$-generic $g$ and $V[g]$-generic $h$, there is an elementary embedding $j:(M_\phi)^{V[g]}\rightarrow (M_\phi)^{V[g*h]}$ such that for every $A\in \powerset(\bR)^{(M_\phi)^{V[g]}}$, $j(A)=A_h$.\label{cl:3_def:st}
\end{enumerate}\vspace{0.5em} 
We say that the Sealing embedding for $\phi$ preserves $\psi$ if letting $B_\psi=\{a: V\models \psi[a]\}$, 
\begin{enumerate}\itemsep0.5em
\item for every $V$-generic $g$, $(B_\psi)^{V[g]}\cap (M_\phi)^{V[g]} \in (M_\phi)^{V[g]}$ and
\item there exists $j$ as in item 3 above with the additional property that $$j((B_\psi)^{V[g]}\cap (M_\phi)^{V[g]})=(B_\psi)^{V[g*h]}\cap (M_\phi)^{V[g*h]}.$$
\end{enumerate}\vspace{0.5em} If $B$ is the set defined by $\psi$, then often instead of ``Sealing embedding preserves $\psi$" we will say ``Sealing embedding preserves B".
\end{definition}

\begin{definition}\label{changmartin} Set  $\eta^{\infty}= (\eta_{\Gamma^\infty})^{L(\Gamma^\infty, \bR)}$ and $\mathcal{B}^\infty=L((\eta^{\infty})^\omega, \mathcal{A}^\infty)$. Given some $V$-generic $g$, we let $\eta^{\infty}_g=(\eta^{\infty})^{V[g]}$ and $\mathcal{B}^\infty_g=(\mathcal{B}^\infty)^{V[g]}$.
\end{definition}

\begin{conjecture}\label{cofinality conjecture} Suppose $\kappa$ is a supercompact cardinal and there is a proper class of inaccessible limits of Woodin cardinals. Suppose $g\subseteq \Col(\omega, 2^{2^\kappa})$ is $V$-generic. Let $\eta^{\infty}_g= (\eta_{\Gamma^\infty_g})^{L(\Gamma^\infty_g, \bR_g)}$. Then $\cf(\eta^{\infty}_g)=\omega$. Moreover, the ${\sf{Sealing}}$ ${\sf{Theorem}}$ holds for $\mathcal{B}^\infty$ in $V[g]$.
\end{conjecture}

%

\subsection{$\sf{Sealing\ Theorem}$ for club filters}

%
%

Motivated by the results of \cite{Sa21CovCM} we make the following conjecture.

\begin{conjecture}\label{chang ad} Suppose $\kappa$ is a supercompact cardinal and there is a proper class of Woodin cardinals. Let $g\subseteq \Col(\omega, {<}\k)$. Then in $V[g]$, \[ L(\sf{Ord}^\omega, \Gamma^\infty_g, \bR_g)[(\mu_\a: \a\in {\sf{Ord}})]\models {\sf{AD^+}}. \]
\end{conjecture}


In this paper, we extend the results of Trang from \cite{TrDMSpctM} and establish the following theorem. To state it, we will need the following definitions.
 
 \begin{definition}\label{def: club filter} Given an uncountable set $X$, we let $\mathcal{C}(X)$ be the club filter on $\powerset_{\omega_1}(X)$. More precisely, we say $C\subseteq \powerset_{\omega_1}(X)$ is a club if there is a function $F: X^{<\omega}\rightarrow X$ such that $C$ is the set of closure points of $F$. $\mathcal{C}(X)$ is the filter generated by the set $\{C \subseteq \powerset_{\omega_1}(X): C$ is a club$\}$. Set $\mathcal{C}^\infty=\mathcal{C}(\Gamma^\infty)$. 
 \end{definition} 
 
 As we have done before, for any $V$-generic $g$, we let $\mathcal{C}^\infty_g=(\mathcal{C}^\infty)^{V[g]}$.
 
 \begin{definition}\label{def:huge cardinal} We say $\k$ is an \emph{elementarily huge cardinal} if there is a $j: V\rightarrow M$ such that $\cp(j)=\k$, $M^{j(\k)}\subseteq M$ and $V_{j(\k)}\prec V$. We say $\l$ is an \emph{eh-superstrongness target} for $\k$ if there is a $j: V\rightarrow M$ witnessing that $\k$ is elementarily huge and $\lambda=\sup\{ j(f)(\k): f\in V\}$. We then let ${\sf{ehst}}(\k)$ be the least eh-superstrongness target of $\k$.\footnote{The key point here is that if $E$ is the $(\k, \lambda)$-extender derived from $j$ and $k:Ult(V, E)\rightarrow M$ is the canonical factor map given by $k([a, f]_E)=j(f)(a)$ then $\pi_E(\k)=\l$, $\cp(k)=\lambda$ and $V_\lambda\subseteq Ult(V, E)$. Thus, $\pi_E$ is a superstrongness embedding with target $\l$. Notice that because $M^{j(\k)}\subseteq M$, $j(\k)$ is an inaccessible cardinal in $V$ while $\l$ is a singular cardinal. Hence, $\l<j(\k)$.}
 \end{definition}
 
 We remark that if $\kappa$ is a 2-huge cardinal as witnessed by $j:V\rightarrow M$ (so $V_{j(j(\kappa))}\in M$) then $V_{j(j(\k))}\models ``\k$ is an elementarily huge cardinal". Variants of elementarily huge cardinals for restricted levels of elementarity are called $C^{(n)}$-huge and have been studied by Bagaria in \cite{Bagaria12}.
 
 \begin{theorem}\label{thm: sealing for clubs} Suppose $\kappa$ is an elementarily huge cardinal and that there is a proper class of  Woodin cardinals. Let $g\subseteq \Col(\omega, 2^{{\sf{ehst}}(\k)})$ be $V$-generic. Then the $\sf{Sealing\ Theorem}$ for $L(\Gamma^\infty, \bR)[\mathcal{C}^\infty]$ holds in $V[g]$ and moreover, the Sealing embedding for $L(\Gamma^\infty, \bR)[\mathcal{C}^\infty]$ preserves $\mathcal{C}^\infty$. 
%
 \end{theorem}

In addition, in the setting of the previous theorem, $\Theta$ is regular in $L(\Gamma^\infty, \bR)[\mathcal{C}^\infty]$ by the results in \cite{MuSa_ThetaReg}.

\subsection{More on $\sf{Sealing}$.}
Recently, in \cite{SaTrConSealing, SaTr21}, the second author and Nam Trang placed the consistency strength of $\Sealing$ below a Woodin cardinal that is a limit of Woodin cardinals. Their proof also established a version of Theorem \ref{thm:newmain} but assuming strong inner model theoretic hypotheses. One of these hypotheses was self-iterability. Self-iterability does not follow from large cardinals\footnote{Ordinary mice are not self-iterable.}, and so the Sargsyan-Trang result, while it reduces the consistency strength of the conclusion of Theorem \ref{thm:newmain}, it only establishes it in a very restricted environment. It is then natural to ask whether one can obtain $\Sealing$ by collapsing a smaller cardinal than a supercompact cardinal. Intuitions coming from inner model theory suggest that a strong cardinal reflecting the set of strong cardinals should be enough.

\begin{definition}\label{def:sigma2strong} We say $\kappa$ is a strong cardinal reflecting the set of strong cardinals if for every $\lambda>\kappa$, there is an elementary embedding $j: V\rightarrow M$ such that $\crit(j)=\kappa$, $j(\kappa)>\lambda$, $V_\lambda\subseteq M$ and for every $\eta<\lambda$, $V\models ``\eta$ is a strong cardinal" if and only if $M\models ``\eta$ is a strong cardinal."
\end{definition}

\begin{question} Suppose $\kappa$ is a strong cardinal reflecting the set of strong cardinals and $g\subseteq \Col(\omega, 2^{2^\kappa})$ is generic. Assume the existence of a proper class of Woodin cardinals and strong cardinals. Does $\Sealing $ or $\sf{Generically\ Correct\ Sealing}$ hold in $V[g]$?
\end{question}

$\Sealing$ is deeply connected to the Inner Model Program, for details see the introduction of \cite{SaTrConSealing}. The key open question surrounding $\Sealing$ is whether some large cardinal axiom implies it. 
\begin{question}
Is there a large cardinal that implies $\Sealing$?
\end{question}
Since $\Sealing$ fails in all known canonical inner models, such as mice or hod mice, the inner model theory of a cardinal implying $\Sealing$ must be utterly different than the modern inner model theory as developed in \cite{St10} or other publications.


\section{Preliminaries, notation and conventions} \label{sec:preliminaries}

Recall that for any $Z,X$ and any ordinal $\gamma$, $\meas_\gamma(Z)$ denotes the set of all $\gamma$-additive measures on $Z^{{<}\omega}$. 
We write $\bar\mu = (\mu_s \mid s \in X^{{<}\omega})$ for a $\gamma$-complete \emph{homogeneity system} over $X$ with support $Z$ 
if for each $s \in X^{{<}\omega}$, $\mu_s \in \meas_\gamma(Z)$. For details on the definition of homogeneity systems we refer the reader to \cite{St09}.

\begin{definition}\label{def: smu}
    A set $A \subseteq X^{\omega}$ is $\gamma$-homogeneously Suslin if there is a $\gamma$-complete homogeneity system $\bar\mu = (\mu_s \mid s \in X^{{<}\omega})$ and a tree $T$ such that $A = p[T]$ and, for all $s \in X^{{<}\omega}$, $\mu_s(T_s) = 1$. In particular, \[ A = S_{\bar{\mu}}=_{def} \{ x \in X^\omega \mid (\mu_{x \upharpoonright n} \mid n < \omega) \text{ is well-founded} \}.\footnote{This means that the natural direct limit of $\Ult(V, \mu_{x \upharpoonright n})$ for $n<\omega$ is well-founded.} \]
\end{definition}

Here $T_s = \{ t \mid (s, t) \in T \}$. We write \[ \Hom_\gamma = \{ A \subseteq X^{\omega} \mid A \text{ is } \gamma\text{-homogeneously Suslin} \} \] and \[ \Hom_{{<}\eta} = \bigcup_{\gamma < \eta} \Hom_\gamma. \]

For $\eta$ a limit of Woodin cardinals and $H \subseteq \Col(\omega, {<}\eta)$ generic over $V$, write $H \upharpoonright \alpha = H \cap \Col(\omega, {<}\alpha)$. As usual, let \[ \bR^* = \bR_H^* = \bigcup_{\alpha < \eta} \bR \cap V[H \upharpoonright \alpha]. \] Moreover, for any $\alpha < \eta$ and $A \in \Hom_{{<}\eta}^{V[H \upharpoonright \alpha]}$, write \[ A^* = \bigcup_{\alpha < \beta < \eta} A_{H \upharpoonright \beta} \] and \[ \Hom^* = \{ A^* \mid \exists \alpha < \eta \, A \in (\Hom_{{<}\eta})^{V[H\upharpoonright \alpha]}\}. \]

In this setting, $L(\bR^*, \Hom^*)$ is called a derived model of $V$ at $\eta$. The following is Woodin's Derived Model Theorem (see \cite[Theorem 6.1]{St09}).

\begin{theorem}[Woodin]\label{woodin: der model thm} Suppose $\eta$ is a limit of Woodin cardinals and $H \subseteq \Col(\omega, {<}\eta)$ generic over $V$. Let $L(\bR^*, \Hom^*)$ be the model computed as above in $V(\bR^*)$. Then $L(\bR^*, \Hom^*)\models {\sf{AD^+}}$.
\end{theorem}

We will also use the following result of Martin-Steel and Woodin. The proof of this result is a combination of \cite{MaSt08} and \cite[Theorem 3.1]{St09}.

\begin{theorem}[Martin-Steel-Woodin]\label{thm:msw} Suppose there is a class of Woodin cardinals. Then  $\Gamma^\infty=\cap_{\k\in \Ord}\Hom_\k$.
\end{theorem}

In this article, our terminology of iteration trees is standard, see, for example, \cite{MaSt94, Ne10}, but we restrict ourselves to iteration trees in which all extenders have inaccessible length and satisfy that the length is equal to the strength of the extender.

The central tool that we will use from coarse inner model theory, besides general iterability results as in \cite{MaSt94}, is Neeman's genericity iteration. We recall the statement here for the reader's convenience. 

\begin{definition}[Neeman, \cite{Ne10}]
    Let $M$ be a model of $\ZFC$, $x$ a real, and $\mathbb{P} \in M$ a partial order. An iteration tree $\cT$ on $M$ is said to \emph{absorb $x$ to an extension by an image of $\mathbb{P}$} if for every well-founded cofinal branch $b$ through $\cT$, there is a generic extension $M_b^\cT[g]$ of $M_b^\cT$, the final model along $b$, by the partial order $j_{0,b}^\cT(\mathbb{P})$ so that $x \in M_b^\cT[g]$.
\end{definition}

\begin{theorem}[Neeman, \cite{Ne95,Ne10}]\label{thm:NeemanGenIt}
    Let $M$ be a model of $\ZFC$, let $\delta$ be a Woodin cardinal in $M$ such that $\powerset^M(\delta)$ is countable in $V$. Then for every real $x$ there is an iteration tree $\cT$ of length $\omega$ on $M$ which absorbs $x$ into an extension by an image of $\Col(\omega, \delta)$.
\end{theorem}


We will need the following slight generalization of \emph{flipping functions}, see \cite[Lemma 2.1]{StstattowfreeDMT}. As usual, if $Y \subseteq \meas_\gamma(Z)$ for some $\gamma$ and $Z$, we write $\TW_Y$ for the set of all towers of measures $\vec\mu = (\mu_i \mid i<\omega)$ such that $\mu_i \in Y$ for each $i<\omega$. To make the concepts used in this paper more transparent, we make the following definitions. 
\begin{definition}\label{uniform measures} We say $Y$ is a \emph{uniform set of measures} if for some $Z$ and some $\gamma$, $Y\subseteq \meas_\gamma(Z)$. If $Y$ is a uniform set of measures then we let ${\sf{set}}(Y)$ be the set $Z$ such that for some $\gamma$, $Y\subseteq \meas_\gamma(Z)$. If $Y$ is a set of measures then we let ${\sf{comp}}(Y)$ be the largest $\gamma$ such that every measure in $Y$ is $\gamma$-complete. 

If $Y$ is a uniform set of measures and $\d$ is a cardinal then we say that $\d$ \emph{splits} $Y$ if $\card{Y}<\d<\comp(Y)$.
\end{definition}

In this paper, we will use the following notation. 
\begin{notation}\label{not:homsystems}
Suppose $\mu$ is an $\eta$-complete measure for some $\eta$ and $g$ is ${<}\eta$-generic. Then $\mu_g\in V[g]$ is the ultrafilter generated by $\mu$ in $V[g]$. Similarly, if $W$ is a set of $\eta$-complete measures and $g$ is ${<}\eta$-generic then $W_g=\{ \mu_g:\mu\in W\}$, and $\TW_{W, g}=\TW_{W_g}$. Continuing with $W$ and $g$, if $\bar \mu=(\mu_s\mid s\in \omega^{<\omega})$ is a homogeneity system then the expression $\bar \mu\subseteq W$ means that for every $s\in \omega^{<\omega}$, $\mu_s\in W$. If $\bar \mu=(\mu_s: s\in \omega^{<\omega}) \subseteq W$ is a homogeneity system, then $\bar \mu_g=((\mu_s)_g: s\in \omega^{<\omega})=_{def}(\mu_{s, g}: s\in \omega^{<\omega})$ and $S_{\mu, g}=_{def}(S_{\mu_g})^{V[g]}=_{def}(S_\mu)^{V[g]}$.
\end{notation}

\begin{definition}\label{flipping system}  We say that $(f, Y, R)$ is a \emph{flipping system} if
\begin{enumerate}
    \item $Y$ and $R$ are uniform sets of measures,
    \item ${\sf{comp}}(R)< {\sf{comp}}(Y)$ and $\comp(R)$ is an inaccessible cardinal,
    \item $f \colon \TW_Y \rightarrow \TW_R $ is a $1$-to-$1$ Lipschitz function\footnote{Recall that a function $f \colon \TW_Y \rightarrow \TW_R$ is Lipschitz if the value of $f(\vec\mu)\upharpoonright n$ is determined by $\vec\mu \upharpoonright n$, for all $\vec\mu$ and $n$.}, and
    \item  for all ${<}\comp(R)$-generics $G$ and all $\vec\mu \in \TW_{Y, G}$\footnote{Since the size of the forcing is small, $f$ induces a map $f \colon \TW_{Y, G} \rightarrow \TW_{R, G}$.}, \[ \vec\mu \text{ is well-founded} \Longleftrightarrow f(\vec\mu) \text{ is ill-founded}. \]
\end{enumerate}
We say $\gamma$ is the \emph{completeness} of $(f, Y, R)$, and write $\gamma={\sf{comp}}(f, Y, R)$, if $\gamma=\comp(R)$. We say that $(f, Y, R)$ \emph{flips through $\d$} if $\d<\comp(Y)$ and $R\in V_\d$.
\end{definition}

The following lemma comes from the proof of projective determinacy presented in \cite{MaSt89}.

\begin{lemma}\label{lem:FlippingFunction}
Let $\delta$ be a Woodin cardinal and $Y$ be a uniform set of measures such that $\d$ splits $Y$. Then for any $\gamma < \delta$, there is some $f$ and $R$ such that 
\begin{enumerate}
    \item $(f, Y, R)$ is a flipping system,
    \item ${\sf{comp}}(f, Y, R)\geq \gamma$,
    \item $(f, Y, R)$ flips through $\d$,
    \item $\card{R}\leq \card{Y}$.\label{cl:4_lem:FlippingFunctiom}
\end{enumerate} 
\end{lemma}
\begin{proof} As in \cite[Lemma 2.1]{StstattowfreeDMT} we find a tree $T$ on some $W\times U$ such that $\card{W}\leq \card{Y}$ and $p[T]$ is the set of well-founded towers $\vec\mu \in \TW_Y$ (this fact remains true in $V[G]$ for any ${<}\d$-generic $G$). We now have that $T$ is $\d^+$-homogenously Suslin. The construction presented in \cite{MaSt89} and also in \cite{Ne10} yields the desired $(f, R)$. Notice that the construction presented in \cite{Ne10} (see the discussion after \cite[Exercise 5.25]{Ne10}) produces a map $s\mapsto \T_s$ with domain $W^{<\omega}$ such that for every $\vec \mu$, $\vec \mu\in p[T]$ if and only if the even branch of $\bigcup_{i<\omega}\T_{\vec \mu \restriction i}$ is ill-founded (for example, see \cite[Claim 5.19]{Ne10}, here $\T_s$ is a finite tree on $V_\d$). Converting the map $s\mapsto \T_s$ into a homogeneity system as it is done in 
\cite{Ne10} we obtain our desired $f$ and $R$. Analyzing the conversion we see that $R$ satisfies Clause \eqref{cl:4_lem:FlippingFunctiom} above. Because each $\T_s$ is an iteration tree on $V_\d$, we get that $R\in V_\d$. We can then modify the construction of $\T_s$ to ensure that $f$ is $1$-to-$1$. For details see \cite[Remark 9.21]{KasumMScThesis} and \cite[Definition 5.13]{Ne10}.
\end{proof}


The following key lemma is essentially due to Steel (see \cite{StstattowfreeDMT}). 

\begin{notation}\label{not:flippingfunctiong} Suppose $p=(f, Y, R)$ is a flipping system and $g$ is ${<}\comp(p)$-generic. Then $f_g: \TW_{Y, g}\rightarrow \TW_{R, g}$ is the extension of $f$ to $V[g]$ and has the property that $f_g(\vec{\mu}_g)=(f(\vec{\mu}))_g$. Given a tower $\vec \mu \in \TW_Y$, we let $f_g(\vec{\mu})=(f(\vec \mu))_g$.
\end{notation}

\begin{lemma}\label{lem:homsysteminWg} Suppose $\k$ is a supercompact cardinal, $q=(f, Y, R)$ is a flipping system such that $\comp(q)\geq \k$\footnote{Recall that $\comp(q)$ must be an inaccessible cardinal, but this condition is not essential.}, $\mathbb{P}\in V_{\comp(q)}$ is a poset, $j:V\rightarrow M$ is a $\comp(q)$-supercompactness embedding,  $g\subseteq \mathbb{P}$ is $V$-generic, $\bar \mu \subseteq W$ and $\bar \mu\in V[g]$. Then $$(S_{\bar \mu})^{V[g]}=(S_{j\pwimg \bar \mu})^{M[g]}.$$
\end{lemma}
\begin{proof}
Because $M[g]$ is closed under countable sequences in $V[g]$, $j \pwimg \bar\mu \in M[g]$. Moreover, since the completeness of the measures in $j \pwimg \bar\mu$ is  above the size of $\card{\mathbb{P}}^V$, $j \pwimg \bar\mu$ can be lifted to $M[g]$. Set $\bar \nu = j \pwimg \bar \mu$ and let $\bar \nu= (\nu_s: s\in \omega^{<\omega})$.


\begin{claim}
For each $x \in \bR^{V[g]} = \bR^{M[g]}$, \[ (\mu_{x\upharpoonright n, g}\mid n < \omega) \text{ is well-founded in } V[g] \text{ iff } \hspace{3cm} \]  \[ \hspace{3cm} (\nu_{x\upharpoonright n, g} \mid n < \omega) \text{ is well-founded in } M[g]. \]
\end{claim}
\begin{proof}
First, suppose $(\mu_{x\upharpoonright n, g}\mid n < \omega)$ is ill-founded in $V[g]$. Then as the direct limit along $(\mu_{x\upharpoonright n, g}\mid n < \omega)$ embeds into the direct limit along $(\nu_{x\upharpoonright n, g} \mid n < \omega)$, it is easy to see that $(\nu_{x\upharpoonright n, g} \mid n < \omega)$ is ill-founded in $M[g]$, as desired. 

For the other implication, suppose $(\mu_{x\upharpoonright n, g}\mid n < \omega)$ is well-founded in $V[g]$. Then $f_{g}((\mu_{x\upharpoonright n} \mid n < \omega))$ is ill-founded in $V[g]$. Therefore, $$(j(f((\mu_{x\upharpoonright n}\mid n < \omega))))_{g} = j(f)_{g}((\nu_{x\upharpoonright n} \mid n < \omega))$$ is ill-founded in $M[g]$. As $(j(f), j(Y), j(R))$ is a flipping system in $M$, this implies that $(\nu_{x\upharpoonright n, g} \mid n < \omega)$ is well-founded in $M[g]$, as desired.
\end{proof}
This finishes the proof of Lemma \ref{lem:homsysteminWg}.
\end{proof}

\section{Preserving universally Baire sets} \label{sec:preservinguB}

In this section we prove a general lemma on the preservation of universally Baire sets under supercompact embeddings. 

\begin{definition}\label{def:tamed} We say that the set of measures $W$ is \emph{tamed} if $\comp(W)$ is an inaccessible cardinal and there is a Woodin cardinal $\d$ such that $\card{W}<\d<\comp(W)$. In this case we say that $W$ is tamed by $\d$, or that $\d$ tames $W$. We say that $W$ is \emph{tamed} above $\k$ if some $\d>\k$ tames $W$. We say that $W$ is \emph{infinitely tamed} if there are infinitely many Woodin cardinals that tame $W$, and similarly, we say $W$ is \emph{infinitely tamed} above $\k$ if there are infinitely many Woodin cardinals $>\k$ that tame $W$.
\end{definition}

The following is the key lemma about tamed $W$. It is due to Steel (see \cite[Subclaim 1.1, 1.1a and 1.1b]{StstattowfreeDMT}), and it is a variation of Lemma \ref{lem:homsysteminWg}.

\begin{lemma}\label{lem:keylemmatamed} Suppose $W$ is tamed above $\k$, $\chi$ is such that $W\in V_\chi$, $\pi: M\rightarrow V_\chi$ is an elementary embedding with $M$ countable and transitive and $\{W, \k\}\in \rng(\pi)$. Let $g$ be ${<}\pi^{-1}(\k)$-generic over $M$ and $\bar \mu\in M[g]$ be a homogeneity system such that $\bar\mu\subseteq M$ and $\pi\pwimg \bar \mu\subseteq W$. Then $(S_{\bar \mu_g})^{M[g]}=S_{\pi\pwimg\bar \mu}\cap M[g]$.
\end{lemma}
\begin{proof}
    Let $\bar \mu=(\mu_s: s\in \omega^{<\omega})$ and $\pi\pwimg \bar\mu=\bar \nu=(\nu_s: s\in \omega^{<\omega})$. Let $\d>\k$ be such that $\d\in \rng(\pi)$ and $\d$ tames $W$. We can then find a flipping system $p=(f, W, R)$ that flips through $\d$ with $p\in \rng(\pi)$. Let $q=(f', W', R')=\pi^{-1}(p)$. Notice now that if $\vec \tau \in (TW_{W'})^{M[g]}$ then $\vec \tau$ is well-founded (in $M[g]$) if and only if $\pi\pwimg \vec \tau$ is well-founded (see the proof of Lemma \ref{lem:homsysteminWg}). 
    
    We now have that for each $x\in \bR\cap M[g]$, 
    \begin{eqnarray*}
    M[g]\models x\in S_{\bar \mu_g} &\iff& M[g]\models ``(\mu_{x\rest n, g}: n<\omega)\ \text{is well-founded}"\\
    &\iff&  (\nu_{x\rest n}: n<\omega)\ \text{is well-founded}\\
    &\iff& x\in S_{\bar \nu}.
\end{eqnarray*}
\end{proof}

The following two lemmas are useful lemmas that can be established using the same methods.

\begin{lemma}\label{lem:preserving equality} Suppose $W$ is a set of measures and $\d$ tames $W$. Suppose  $g$ is ${<}\d$-generic, and $\bar \mu$, $\bar \nu$ in $V[g]$ are two homogeneity systems such that $\bar{\mu}\subseteq W$ and $\bar{\nu}\subseteq W$. Let $h$ be ${<}\d$-generic over $V[g]$. Then $V[g]\models S_{\bar \mu_g}=S_{\bar \nu_g}$ if and only if $V[g*h]\models S_{\bar \mu_g}=S_{\bar \nu_g}$.
\end{lemma}
\begin{proof} Toward a contradiction, suppose not. Let $\chi$ be a limit cardinal such that $W\in V_\chi$. Let $\pi: M\rightarrow V_\chi[g]$ be an elementary embedding with $M$ countable and transitive, and $\{W, \d\}\in \rng(\pi)$. Let $(\bar{W}, \bar g, \bar \d, \bar \tau, \bar \xi)=\pi^{-1}(W, g, \d, \bar \mu, \bar\nu)$. Let now $k$ be $M$-generic for a poset of size $<\bar \d$, and suppose that $M[k]\models S_{\bar \tau_{\bar g*k}}\not=S_{\bar \xi_{\bar g*k}}$. Suppose that there is $x\in M[k]\cap \bR_g$ such that $x\in S_{\bar \tau_{\bar g*k}}$ but $x\not \in S_{\bar \xi_{\bar g*k}}$. Then it follows from Lemma \ref{lem:keylemmatamed} that $x\in S_{\bar \mu_{g}}$ but $x\not \in S_{\bar \nu_{g}}$, which is a contradiction.
\end{proof}

\begin{lemma}\label{lem:equivalence} Suppose $W$ is infinitely tamed above $\k$, $\chi$ is such that $W\in V_\chi$, $\pi: M\rightarrow V_\chi$ is an elementary embedding with $M$ countable and transitive, and $\{W, \k\}\in \rng(\pi)$. Let $g$ be ${<}\pi^{-1}(\k)$-generic over $M$ and $\bar \mu, \bar \nu\in M[g]$ be homogeneity systems such that $\bar\mu, \bar \nu\subseteq M$, $\pi\pwimg \bar \mu\subseteq W$ and $\pi\pwimg \bar \nu\subseteq W$. Suppose that in $M[g]$, $S_{\bar \mu_g}=S_{\bar \nu_g}$. Then $S_{\pi\pwimg \bar \mu}=S_{\pi\pwimg \bar \nu}$.
\end{lemma}
\begin{proof}
    Let $\bar \mu=(\mu_s: s\in \omega^{<\omega})$ and $\bar \nu=(\nu_s: s\in \omega^{<\omega})$. We let $\pi\pwimg \bar \mu= (\mu^+_s: s\in \omega^{<\omega})$ and  $\pi\pwimg \bar \nu= (\nu^+_s: s\in \omega^{<\omega})$. Let $\d_1>\d_0>\k$ be such that $\{\d_0, \d_1\}\in \rng(\pi)$ and both $\d_0$ and $\d_1$ tame $W$ (and in particular they are Woodin cardinals). Let $(\eta_0, \eta_1)=\pi^{-1}(\d_0, \d_1)$. We can then find a flipping system $p=(f, W, R)$ that flips through $\d_1$ with $p\in \rng(\pi)$ and $\d_0<\comp(p)$. Let $q=(f', W', R')=\pi^{-1}(p)$. 
    
    Fix now $x\in \bR$ and using Theorem \ref{thm:NeemanGenIt} find $\pi': M'\rightarrow V_\chi$ and $\sigma: M\rightarrow M'$ such that $\pi=\pi'\circ \sigma$, $\cp(\sigma)>\pi^{-1}(\k)>\card{g}$ and for some $M'[g]$-generic $h\subseteq \Col(\omega, \sigma(\eta_0))$, $x\in M'[g*h]$. Notice now that because $p\in \rng(\pi')$ if $\vec \tau \in (TW_{\sigma(W')})^{M'[g*h]}$ then $\vec \tau$ is well-founded (in $M'[g*h]$) if and only if $\pi'\pwimg \vec \tau$ is well-founded (see the proof of Lemma \ref{lem:homsysteminWg}).  Let $\sigma^+: M[g]\rightarrow M'[g]$ be the liftup of $\sigma$. We now have that, 
    \begin{eqnarray*}
    x\in S_{\pi\pwimg \bar \mu} &\iff & (\mu^+_{x\rest n}: n\in \omega)\ \text{is well-founded}\\
    &\iff & (\sigma^+(\mu_{x\rest n})_{g*h}: n\in \omega)\ \text{is well-founded in $M'[g*h]$}\\
    &\iff & M'[g*h]\models x\in S_{\sigma^+(\bar \mu_g)} \\
    &\iff &  M'[g*h]\models x\in S_{\sigma^+(\bar \nu_g)} \\
    &\iff & (\sigma^+(\nu_{x\rest n})_{g*h}: n\in \omega)\ \text{is well-founded in $M'[g*h]$}\\
    &\iff &  (\nu_{x\rest n}^+: n<\omega)\ \text{is well-founded}\\
    &\iff & x\in S_{\pi\pwimg\bar \nu}.
\end{eqnarray*}
The fourth equivalence follows from Lemma \ref{lem:preserving equality}.
\end{proof}

\begin{definition}\label{def:uBpreservation}
Suppose $\k$ is an infinite cardinal, $W$ is a set of measures that is tamed above $\k$ and $E$ is an extender with $\cp(E)=\k$. 
We then say that $E$ is \emph{wf-preserving for $W$}\footnote{``wf" stands for ``well-founded".} (in $V$) if 
\vspace{0.5em}\begin{enumerate}\itemsep0.5em
    \item $\lh(E) = \str(E)>\a+\omega$ where $\a$ is the set theoretic rank of $W$, and
    \item in $\Ult(V,E)$ there is a Lipschitz function \[ h \colon \TW_{W} \rightarrow \pi_E(\TW_{W}) \] 
    such that  whenever $\d$ tames $W$ and $g$ is ${<}\d$-generic, for all $\vec{\mu}\in (\TW_{W})^{V[g]}$, $h"\vec{\mu}$ is well-founded if and only if $\vec{\mu}$ is well-founded.

\end{enumerate}\vspace{0.5em}
Moreover, we say a cardinal $\kappa$ is \emph{uB-preserving} if for all sets of measures $W$ that are tamed above $\k$, there is an extender $E$ with critical point $\kappa$ such that $E$ is wf-preserving for $W$.

In what follows, we will say that the function $h$ above is \emph{wf-preserving}.
\end{definition}

\begin{lemma}\label{lem:kappauBpreserving}
Suppose $\kappa$ is a supercompact cardinal. Then $\kappa$ is uB-preserving. 
\end{lemma}
\begin{proof}
Let $W$ be any set of measures that is tamed above $\k$. Let $\l$ be a strong limit cardinal such that $W\in V_{\l}$. Let $j \colon V \rightarrow M$ be an elementary embedding witnessing that $\kappa$ is $\l$-supercompact and let $E$ be the $(\kappa, \l)$-extender derived from $j$. We want to see that $E$ is wf-preserving, which amounts to showing that in $Ult(V, E)$, whenever $\d>\k$ tames $W$, there is a wf-preserving Lipschitz $h:TW_{W}\rightarrow TW_{j(W)}$ such that for all ${<}\d$-generics $g$ and for all $\vec{\mu}\in (\TW_{W})^{V[g]}$, $h"\vec{\mu}$ is well-founded if and only if $\vec{\mu}$ is well-founded. We fix $\d$ and $g$ as above and show the existence of $h$.

We start by appealing to Lemma \ref{lem:FlippingFunction} and fixing a flipping system $(f, W, Y)$ such that $g$ is $<\comp(Y)$-generic and $Y\in V_\d$. Now let $h'= j \upharpoonright \TW_{W}$. Since $j$ is a $\l$-supercompactness embedding, we have $h'\in M$. It then follows from Lemma \ref{lem:homsysteminWg} that $h'$ is wf-preserving for $W$. 

We now have that if we let $\varphi(W, j(W))$ denote the formula ``$\exists h: TW_W\rightarrow TW_{j(W)}$ Lipschitz function such that $h$ is wf-preserving"  then
\[ M \models \varphi(W, j(W)).\]
Let $k \colon \Ult(V,E) \rightarrow M$ be the factor map. Then $k \upharpoonright V_{\l} = \id$ and hence $$(W, j(W)) \in \rng(k).$$ In fact, $k^{-1}(W) = W$ and $k^{-1}(j(W)) = \pi_E(W)$. Therefore, \[ \Ult(V,E) \models \varphi(W, \pi_E(W)). \]
If now $h\in Ult(V, E)$ is the witness to $\varphi(W, \pi_E(W))$, then $h$ is as desired.
\end{proof}

\section{The UB-Capturing Principle}\label{sec:blocks}

Our ultimate goal is to find a derived model representation for $L(\Gamma^\infty, \bR)$. One natural attempt is to start with some $\pi: M\rightarrow V_\xi$ with $M$ countable and realize $L(\Gamma^\infty, \bR)$ as a derived model of $M$. This is the approach taken in \cite{StstattowfreeDMT}, and this idea leads to the proof of the Derived Model Theorem. However, \cite[Lemma 1.1]{StstattowfreeDMT} implies that this approach may not work for our purposes (see also \cite[Subclaim 1.1]{StstattowfreeDMT}), as this lemma suggests that there is a homogeneously Suslin set that is not in the derived model of $M$, whereas our goal is to realize $L(\Gamma^\infty, \bR)$ as the derived model of $M$. More precisely, \cite[Lemma 1.1]{StstattowfreeDMT}, known as the Windszus' Lemma, states essentially the following.

\begin{lemma}\label{lem:lemma 1.1} Suppose $\pi: M\rightarrow V_\xi$ is such that $\pi$ is elementary and $M$ is countable and transitive. Let $S$ be the set of reals $x$ that code an iteration tree $\T$ on $M$ such that $\lh(\T)=\omega+1$ and $\pi\T$, the copy of $\T$ on $V_{\xi}$, is well-founded. Then $S$ is homogeneously Suslin.
\end{lemma}

Lemma \ref{lem:lemma 1.1} seems to suggest that the set $S$ can never be in the derived model of $M$, and so ``to put" $S$ in the derived model of $M$ we need to enlarge $M$ to some $N$ which can capture $S$ in its derived model. We informally call this process of enlargement ``changing the model". Before we continue, we state the above as a question. 

\begin{question} Suppose there is a class of Woodin cardinals and suppose $\xi$ is an inaccessible cardinal that is a limit of Woodin cardinals. Is there a pair $(\pi, M)$ such that $\pi: M\rightarrow V_\xi$, $\pi$ is elementary, $M$ is countable and transitive and such that letting $S$ be as in Lemma \ref{lem:lemma 1.1}, there are a $V$-generic $g\subseteq \Col(\omega, \bR)$  and an $\bR$-genericity iterate $N\in V[g]$ of $M$ via an iteration tree $\T$\footnote{This means that for some $\l$ that is a limit of Woodin cardinals of $M$, $\pi^\T(\l)=\omega_1$, $\lh(\T)=\omega_1^V+1$, for every $\a<\omega_1^V$, $\T\rest \a\in V$, and for every $x\in \bR$ there is $\b<\omega_1^V$ and $k\subseteq \Col(\omega, \b)$ such that $k\in V$ is $N$-generic and $x\in N[k]$.} such that $N$ is $\pi$-realizable\footnote{I.e. there is $k: N\rightarrow V_\xi$ such that $\pi=k\circ \pi^\T$.} and if $h\subseteq \Col(\omega, {<}\omega_1^V)$ is $N$-generic, $h\in V[g]$ and $(\bR^*)^{N[h]}=\bR^V$ then $S$ is in the derived model of $N$ at $\omega_1^V$ as computed by $h$?
\end{question}

Our initial idea, which originated in \cite{SaTr21}, was to use an $\omega$-sequence of countable substructures of $V_\xi$, which we were calling \textit{blocks}. Our goal was to iterate these $\omega$ many models simultaneously to make the reals generic in such a way that $L(\Gamma^\infty, \bR)$ could be realized as the derived model of the final iterate. This approach appeared in \cite{SaTr21}. It presents the models at the beginning of the process, each model on the sequence corresponds to the ``changing the model" step applied to the previous model on the sequence.

However, thanks to the discussions that the second author had with Obrad Kasum and Boban Velickovic, we have discovered a natural principle (see Definition \ref{def:abstractblock}) that allows us to work with one model at a time. We still need to change models as we make more reals and universally Baire sets generic; however, because good models are abundant (see Theorem \ref{thm:ubcapturingprinciple}), the change of the model will happen automatically. We will now make these vague ideas more precise.




We now set up some notation and recall some basic definitions. Below if $U$ is transitive and $\a$ is an ordinal then $U_\a=(V_\a)^U=U|\a$. If $\a$ is a cardinal of $U$ then $H_\a^U$ is the  $\a$-th level of the $H$-hierarchy of $U$. Recall from \cite{La04} that if $X\subseteq V_\gamma$, $n\in \omega$ and $\{a_0, ..., a_n, Y\}\subseteq V_\gamma$ then $X[a_0, a_1,..., a_n]=\{ f(a_0, a_1, ..., a_n): f \in X\}$ and $X(Y)=\bigcup\{ X[a_0, a_1, ..., a_n]: (a_0, ..., a_n)\in Y^{n+1}\}$. If now $\gamma'>\gamma$, $X\prec V_{\gamma'}$, $\gamma\in X$ and $\{a_0, ..., a_n, Y\}\subseteq V_\gamma$ then $X[a_0, ...a_n]\cap V_\gamma \prec V_\gamma$ and $X(Y)\cap V_\gamma\prec V_\gamma$. Moreover, we have the following lemma.

\begin{lemma}\label{lem:endext} Suppose $\gamma$ is an ordinal, $\gamma'>\gamma$, $X\prec V_{\gamma'}$ and $\gamma\in X$. Then the following holds.
\begin{enumerate}
    \item If $Y\prec V_\gamma$ is such that $X\cap V_\gamma\subseteq Y$ and $\sup (X\cap \gamma)=\sup( Y\cap \gamma)$ then $Y=\cup \{ X[a]\cap V_\gamma: a\in Y\}$.
    \item Suppose further that $\gamma$ is an inaccessible cardinal. If $\xi\in X\cap \gamma$ and $Y\subseteq V_\xi$, then $\sup(X\cap \gamma)=\sup(X(Y)\cap \gamma)$ and $(X\cap V_\gamma)(Y)\prec V_\gamma$.\label{cl:2_Lemma4.3}
\end{enumerate}
\end{lemma}
\begin{proof}
To show the first claim, it suffices to see that $X[a]\cap V_\gamma\subseteq Y$ given that $a\in Y$. Let $\xi$ be the rank of $a$ and let $\zeta\in X\cap \gamma$ be larger than $\xi$, which exists because $\sup (X\cap \gamma)=\sup( Y\cap \gamma)$ and $\xi\in Y$. Suppose now that $f\in X$, $f(a)\in V_\gamma$ and we want to see that $f(a)\in Y$. We then have that $f\rest V_\zeta\in X\cap V_\gamma\subseteq Y$ and hence, $f(a)=(f\rest V_\zeta)(a)\in Y$. 

To show the second clause, it is enough to show the claim holds for $Y=\{a\}$. Suppose then that $\gamma$ is an inaccessible cardinal, $\xi\in X\cap \gamma$ and $a\in V_\xi$. We want to see that if $f\in X$ is such that $f(a)$ is ordinal, then $f(a)<\sup(X\cap \gamma)$. Let $A\subseteq V_\xi$ be the set of those $u\in V_\xi$ such that $f(u)$ is an ordinal. Then $A\in X$. Let then $\tau=\sup\{ f(u): u\in A\}$. Because $\gamma$ is inaccessible we have $\tau<\gamma$. Since clearly $\tau\in X$, we have $f(a)<\sup(X\cap \gamma)$.

 Finally, suppose $\gamma$ is an inaccessible cardinal, $\xi\in X\cap \gamma$ and $a\in V_\xi$. We want to see that $(X\cap V_\gamma)[a]\prec V_\gamma$. This is because $X[a]\cap V_\gamma=(X\cap V_\gamma)[a]$. Indeed, let $f\in X$ be such that $f(a)\in V_\gamma$ and let $A\subseteq V_\xi$ be the set of those $u$ such that $f(u)\in V_\gamma$. Then $A\in X$ and hence, $A\in X\cap V_\gamma$. Therefore, $f\rest A\in X\cap V_\gamma$. But we have $f(a)=f\rest A(a)$, which implies that $f(a)\in (X\cap V_\gamma)[a]$.
\end{proof}

\begin{definition}\label{def:strongsub} Suppose $\gamma$ is an inaccessible cardinal and $X\prec V_\gamma$. We say $X$ is a \emph{strong substructure} if for every $\xi\in X\cap \gamma$ and $a\in V_\xi$, $X[a]\prec V_\gamma$. 
    
\end{definition}

Before we introduce the UB-Capturing Principle, we need to introduce a few more basic concepts.

\begin{definition}\label{def:wetafull} Suppose $\chi>\eta$ are inaccessible cardinals, $W\in V_\chi$, and  $X\prec V_\chi$ is a strong substructure of $V_\chi$. We say that $X$ is \emph{$(W\mid \eta)$-full} if $\eta\in X$ and for every $b\in V_\eta$, 
$$X(W\cup \{b\})\cap \powerset(V_\eta)=X[b]\cap \powerset(V_\eta).$$
\end{definition}

The above concepts also make sense in generic extensions of $V$. If $W$ is any set then we let $\comp(W)$ be the largest cardinal $\k$ such that every ultrafilter in $W$ is $\k$-complete. 

\begin{definition}\label{def:hullscapturinguBse} Suppose $\chi>\eta$ are inaccessible cardinals, $g$ is ${<}\eta$-generic, $A\in \Gamma^\infty_g$, $X\prec V_\chi$ with $X\in V[g]$ and $W\in V_\chi$ with $\comp(W)>\eta$.  We say $(X, W)$ \emph{captures} $A$ if there is a homogeneity system $\bar \mu\subseteq X\cap W$ such that (in $V[g]$) $A=S_{\bar\mu_g}$.
\end{definition}

 If there is a $W$ such that $(X,W)$ captures $A$ we also just say that $X$ captures $A$.
 We are now ready to introduce the UB-Capturing Principle. 

\begin{definition}[The UB-Capturing Principle]\label{def:abstractblock}  Suppose $\chi>\eta>\l$ are inaccessible cardinals. We then say that the \emph{UB-Capturing Principle} holds at $(\chi, \eta, \l)$ if there is a pair $(W, X)$ such that
   \vspace{0.5em}\begin{enumerate}\itemsep0.5em
    \item $X$ is a strong substructure of $V_\chi$ of size $<\l$,
    \item $\eta\in X$ and $X$ is $(W\mid \eta)$-full,\label{cl:2_def:UBCapturingPrinciple}
    \item there is $W'\in X$ such that $W\subseteq W'$ and $W'$ is infinitely tamed above $\eta$ (see Definition \ref{def:tamed}),\label{cl:3_def:UBCapturingPrinciple}
    \item for every ${<}\l$-generic $g$ and every $A\in \Gamma^\infty_g$, in $V[g]$, there is a countable $z\subseteq W$ such that $(X(z), W)$ captures $A$. \label{cl:4_def:UBCapturingPrinciple}
\end{enumerate}\vspace{0.5em}
\end{definition}

The following simple lemma will be our main way of using the UB-Capturing Principle.

\begin{lemma}\label{lem:mainappubcap} Suppose the UB-Capturing Principle holds at $(\chi, \eta, \l)$ as witnessed by $(W, X)$.
    Let $h$ be ${<}\l$-generic, $A\in \Gamma^\infty_h$, $z\subseteq W$ witness Clause \eqref{cl:4_def:UBCapturingPrinciple} of Definition \ref{def:abstractblock} applied to $(h, A)$ and let $y\in V_\eta[h]\cap \powerset(V_\eta)$. Set $Y=X(y)$ and $Z=Y(z)$, and let $M_Y$ be the transitive collapse of $Y$ and $M_Z$ be the transitive collapse of $Z$. Let $\pi_Y: M_Y\rightarrow V_\gamma$ and $\pi_Z: M_Z\rightarrow V_\gamma$ be the inverses of the transitive collapses, and let $\pi_{Y, Z}=(\pi_Z)^{-1}\circ \pi_Y$. Then $\pi_Y$ and $\pi_Z$ are elementary, $\cp(\pi_{Y, Z})>\pi_Y^{-1}(\eta)$ and $Z$ captures $A$.
\end{lemma}
\begin{proof}
 Clause \eqref{cl:2_Lemma4.3} of Lemma \ref{lem:endext} and the fact that $W\in X$ imply that $Y$ and $Z$ are elementary substructures of $V_\chi$, and hence, $\pi_Y$ and $\pi_Z$ are elementary. Clearly, since $X(z)$ captures $A$, $Z=Y(z)$ captures $A$.
 
   It remains to prove that  $\cp(\pi_{Y, Z})>\pi_Y^{-1}(\eta)$. To show this, it is enough to prove that $Y\cap V_\eta=Z\cap V_\eta$. Notice that $Z=X(y, z)$. Let then $u\in Z\cap V_\eta$. We then have an $a\in y$ such that $u\in X(z\cup\{a\})$. Because $X$ is $(W\mid \eta)$-full, we have that $u\in X[a]$. Hence, $u\in Y$. 
\end{proof}

We now work towards proving the consistency of UB-Capturing Principle. We start with the concept of \textit{saturated set of measures}, which is a set of measures that can be used to represent any universally Baire set as a homogeneously Suslin set. The reader may wish to recall Definition \ref{uniform measures}, Notation \ref{not:homsystems} and Definition \ref{flipping system}.

\begin{definition}\label{def:saturated measures} Suppose $\eta\geq \l$ are cardinals. We say that a uniform set of measures $W$ is \emph{$\eta$-saturated} if $\eta$ splits $W$ and for every $A\in \Gamma^\infty$ there is a homogeneity system $\bar{\mu}\subseteq W$ such that $A=S_{\bar{\mu}}$. We say $W$ is \emph{$(\eta, \l)$-saturated} if for all ${<}\l$-generics $g$, $W_g$ is $\eta$-saturated in $V[g]$. 

We say $W$ is \emph{$(\chi, \eta, \l)$-saturated} if $\eta$ splits $W$ and for every ${<}\l$-generic $h$ and for every $A\subseteq \bR$ that is ${<}\chi$-universally Baire, there is  a homogeneity system $\bar{\mu}\subseteq W$ such that $A=S_{\bar{\mu}}$.
\end{definition}

It is not hard to show that provided there is a class of Woodin cardinals, if $\eta\geq \l$ are inaccessible cardinals then there is a $(\l, \eta)$-saturated $W$. This follows from Theorem \ref{thm:msw}. We can now prove our main theorem on the UB-Capturing Principle. 
\begin{theorem}\label{thm:ubcapturingprinciple} Suppose $\k<\l<\chi$ are inaccessible cardinals such that $\k$ is a supercompact cardinal, $W\in V_\chi$ is a $(\chi, \l^+, \l)$-saturated set of ultrafilters and $j: V\rightarrow M$ witnesses that $\k$ is $\chi$-supercompact. Moreover, assume that the interval $(\l, \comp(W))$ contains infinitely many Woodin cardinals. Let $X\prec V_{\chi+1}$ be elementary and such that $V_{\k+1}\subseteq X$ and $\card{X}=\card{V_{\k+1}}$.  Then the UB-Capturing Principle holds in $M$ at $(j(\chi), j(\k), \l)$ as witnessed by $(j[W], j[X\cap V_\chi])$.
\end{theorem}
\begin{proof} Notice that there is a flipping system $p=(f, W, R)\in V_\chi\cap X$ with $\comp(p)>\l$. This means that we can use Lemma \ref{lem:homsysteminWg}.

First, since $X\prec V_{\chi+1}$, we have that $j[X\cap V_\chi]\prec M_{j(\chi)}$ and $j[X\cap V_\chi]$ is a strong substructure of $M_{j(\chi)}$ (see Clause \eqref{cl:2_Lemma4.3} of Lemma \ref{lem:endext}). Next, since $j\rest X: X\rightarrow j[X]$ is a surjection and $\card{X}^V<\l$, $\card{(j[X])}^M<\l$.

We now list some facts which we will need to establish Clause \eqref{cl:3_def:UBCapturingPrinciple} of Definition \ref{def:abstractblock}. We clearly have that\\\\
(1) $j[W]\in M_{j(\chi)}$ and $M\models ``j[W]$ is $(j(\chi), j(\l^+), \l)$-saturated" (see Lemma \ref{lem:homsysteminWg}).\\\\
But since $j(\k)$ is supercompact (which means ${<}j(\k)$-universally Baire sets are universally Baire), we have that\\\\
(2) $M\models ``j[W]$ is $(j(\chi), j(\k), \l)$-saturated". \\\\
(2) then implies Clause \eqref{cl:4_def:UBCapturingPrinciple} of Definition \ref{def:abstractblock}. Clause \eqref{cl:3_def:UBCapturingPrinciple} of Definition \ref{def:abstractblock} is witnessed by $j(W)$. It remains to prove Clause \eqref{cl:2_def:UBCapturingPrinciple}
 of Definition \ref{def:abstractblock}. Notice first that we have that $j(\k)\in j[X\cap V_\chi]$. We thus need to show that in $M$, $j[X\cap V_\chi]$ is $(j[W]\mid j(\k))$-full. 
 To see this, it is enough to prove the following claim.

\begin{claim}
    Suppose $a\in V_\chi$ and $b\in j(V_\k)$. Then $$j[X\cap V_\chi][j(a), b]\cap \powerset(M_{j(\k)})= j[X\cap V_\chi][b]\cap \powerset(M_{j(\k)}).$$
\end{claim}
\begin{proof}
   Let $B\in j[X][j(a), b]\cap \powerset(M_{j(\k)})$ and let $g\in X$ be such that $B=j(g)(j(a), b)$. Let $k$ be the function in $V$ given by $k(u)=g(a, u)$. We have that $k\in X[a]$. Let $m=k\rest V_\k$ and $C\subseteq V_\k$ be the set of those $u\in V_\k$ such that $m(u)\subseteq V_\k$. Let $D=\{ (w, z): w\in C \wedge z\in m(w)\}$. We have that\\\\
   (3) $B=\{ z: (b, z)\in j(D)\}$.\\\\
   Now, because $V_{\k+1}\subseteq X$ and $D\in V_{\kappa+1}$, we have that $D\in X\cap V_\chi$. But then $j(D)\in j[X\cap V_\chi]$, and so (1) implies that $B\in j[X\cap V_\chi][b]$.  
\end{proof}
This finishes the proof of Theorem \ref{thm:ubcapturingprinciple}.
\end{proof}

In what follows we will describe how to realize $L(\Gamma^\infty, \bR)$ as the derived model of the direct limit model of a linear system of models. Each model in this system will be obtained via the \textit{one step construction} described in the next section. The construction of the entire linear system will be performed in $V[g]$ where $g$ is a generic enumeration of $\Gamma^\infty$ in order type $\omega$ or $\omega_1$. Let $\iota\in \{ \omega, \omega_1\}$ and let $(M_\a, n_{\a, \b} \mid \a<\b<\iota)$ be the linear sequence our construction produces with $n_{\a,\b}: M_\a \rightarrow M_\b$. Also, let $(A_\a \mid \a<\iota)$ be the generic enumeration of $\Gamma^\infty$. We will have an inaccessible cardinal $\chi$ and embeddings $\pi_\a: M_\a\rightarrow V_\chi$ such that for $\a<\b<\iota$, $\pi_\a=\pi_\b\circ n_{\a, \b}$. The one step construction is a construction that describes the procedure for obtaining $M_{\a+1}$ from $M_\a$. Typically, this will have two parts. The first part produces an intermediate model $N$ from $M_\a$ along with an embedding $\sigma: N\rightarrow V_\chi$ such that $\rng(\pi_\a)\subseteq \rng(\sigma)$. $N$'s job is to ensure that $A_\a$ will be in the derived model of $M_\iota$, the direct limit of $(M_\a, n_{\a, \b} \mid \a<\b<\iota)$. The construction of $N$ will involve genericity iterations that make homogeneously Suslin representations of $A_\a$ generic. The second part will be executed by appealing to Lemma \ref{lem:mainappubcap}. This part will produce $M_{\a+1}$ from $N$ using Lemma \ref{lem:mainappubcap}, which we will apply to $Y=_{def} \sigma[N]$ and $A_{\a+1}$. The $M_Z$ of that lemma will be our $M_{\a+1}$. The entire sequence will be produced by consecutively applying the one step construction and taking direct limits at limit stages. In the next section, we will describe the one step construction in more detail. 

\section{The one step construction}

As was mentioned in the previous paragraph, part one of the one step construction involves genericity iterations and a construction of a $\sigma: N\rightarrow V_\chi$ such that $\rng(\pi_\a)\subseteq \rng(\sigma)$. Such a $\sigma$ is usually constructed by using the Martin-Steel Iterability Theorem (see \cite{MaSt94} and \cite[Theorem 2.3]{Ne10}). We start by adopting this theorem to our context.

\subsection{Martin-Steel Iterability Theorem} If $\T$ is an iteration tree then $\gen(\T)$ is the supremum of the lengths of the extenders used in $\T$. First we isolate the models that we will iterate.

\begin{definition}\label{:def:niceext} Suppose $M$ is a transitive model of $\ZFC$ and $E\in M$. We say that $E$ is a \emph{nice extender} in $M$ or just $E$ is a nice extender or $M\models ``E$ is a nice extender" if $\lh(E)$ is an inaccessible cardinal of $M$ and $V_{\lh(E)}^M=V_{\lh(E)}^{Ult(M, E)}$. 
\end{definition}

We adopt the convention that iteration trees are normal and use nice extenders, i.e., if $\T$ is an iteration tree on $M$ then for all $\a<\lh(\T)$, $M_\a^\T\models ``E_\a^\T$ is nice". These are the conventions used in \cite{Ne10}.

\begin{definition}\label{def:rock} We say that $\sfa=(M^\sfa, \k^\sfa)$ is a \emph{rock} if $M^\sfa$ is a countable and transitive model of $\ZFC$, and $\k^\sfa\in M^\sfa$ is an inaccessible cardinal in $M^\sfa$.
\end{definition}

Next we introduce the iterations that we will use to iterate rocks.

\begin{definition}\label{def:blockiterate} Suppose $\sfa$ and $\sfb$ are two rocks. We say $\sfb$ is a \emph{(normal) iterate} of $\sfa$ if there is a normal iteration tree $\T$ on $M^\sfa$ such that 
\begin{enumerate}\itemsep0.5em
    \item $\T$ is above $\k^\sfa$\footnote{I.e., every extender used in $\T$ has critical point $>\k^\sfa$.},
    \item $\lh(\T)<\omega_1$ and $M^\sfb$ is the last model of $\T$, and
    \item $\k^{\sfa}=\k^\sfb$.
\end{enumerate} 
We also say that $\T$ is an \emph{iteration tree} on $\sfa$ and that $\sfb$ is the \emph{last model} of $\T$. 
We say that $\sfb$ is a \emph{sealing iterate} of $\sfa$ if there is an iterate $\sfd$ of $\sfa$ via $\T$ and $E\in M^{\sfd}$ such that
\begin{enumerate}\itemsep0.5em
    \item $M^\sfd\models ``E$ is a nice extender"
    \item $\gen(\T)<\lh(E)$, 
    \item $\cp(E)=\k^\sfd(=\k^\sfa)$,
    \item $M^\sfb=Ult(M^\sfa, E)$, and
    \item $\k^\sfb=\pi_E^{M^\sfa}(\k^\sfa)$. 
\end{enumerate} 
In the above situation, we write $\sfb=Ult(\sfa, E)$ and also say that $(\T, E)$ is a \emph{sealing iteration} of $\sfa$ and that $\sfb$ is the \emph{last model} of $(\T, E)$.
\end{definition}

At the time of writing this paper, it was unclear how to use large cardinal notions in the region of a supercompact cardinal to carry out the arguments presented in this paper by using models that embed into rank initial segments of $V$. The issue is that in Theorem \ref{thm:ubcapturingprinciple} the $X$ we produce is not countable but only becomes countable after collapsing $2^\kappa$ to be countable. We then have to apply Martin-Steel Iterability Theorems to countable hulls of $V$ that are in $V[g]$ where $g\subseteq \Col(\omega, 2^\k)$ is the generic. Here $V$ corresponds to $M$ of Theorem \ref{thm:ubcapturingprinciple} and $V[g]$ to $M[g]$. We will then need the following concept to present our constructions\footnote{For this and other similar notions see \cite{GroundAxiom}.}.

\begin{definition}\label{def:ubclose} Suppose $\l$ is an inaccessible cardinal and $U$ is a transitive class of $V$ such that $Ord\subseteq U$ and $U\models \ZFC$. We say $U$ is a \emph{$\l$-ground} of $V$ if for some poset $\mathbb{P}\in U_\l$ and some $g\in V$ that is $U$-generic over $\mathbb{P}$, $V=U[g]$. 
\end{definition}

Typically in this paper we will have a supercompact cardinal $\k$, a poset $\mathbb{P}$, a $V$-generic $g$, an inaccessible cardinal $\chi$ such that $\mathbb{P}\in V_\chi$ and an elementary embedding $j: V\rightarrow M$ witnessing that $\k$ is $\chi$-supercompact (this is the set up of Theorem \ref{thm:ubcapturingprinciple}). Then $U$ of Definition \ref{def:ubclose} is $M$ and $V$ is $M[g]$. The reader may wish to keep this set up in mind, as it might make Definition \ref{def:relrock} more clear. Below if $U$ is transitive and $\a$ is an ordinal then $U_\a=(V_\a)^U$. We now introduce the realizable rocks.

\begin{definition}\label{def:relrock} Suppose $\l<\chi$ are inaccessible cardinals, $U$ is a $\l$-ground of $V$ and $p\in U_\chi$. We then say that $\sfa=(\sigma^\sfa, M^\sfa, \k^\sfa)$ is a $(U, \chi, \l, p)$-\emph{realizable rock} if $(M^\sfa, \k^\sfa)$ is a rock and $\sigma^\sfa: M^\sfa \rightarrow U_\chi$ is an elementary embedding such that $\sigma^\sfa(\k^\sfa)>\l$ and $p\in \rng(\sigma^\sfa)$. We let $\sfa^-=(M^\sfa, \k^\sfa)$. Also, if $p=\emptyset$ then we omit it from our notation.
\end{definition}

\begin{definition}\label{def:itrelrock} Suppose $\l<\chi$ are inaccessible cardinals, $U$ is a $\l$-ground of $V$, $p\in U_\chi$ and $\sfa=(\sigma^\sfa, M^\sfa, \k^\sfa)$ is a $(U, \chi, \l, p)$-realizable rock. We then say that $\sfb$ is an $\sfa$-\emph{realizable rock} if $\sfb$ is a $(U, \chi, \l, p)$-realizable rock, $\rng(\sigma^\sfa)\subseteq \rng(\sigma^\sfb)$ and $\sigma^\sfa(\k^\sfa)=\sigma^\sfb(\k^\sfb)$. In this case, we let $\sigma^{\sfa, \sfb}:M^\sfa\rightarrow M^\sfb$ be $(\sigma^\sfb)^{-1}\circ \sigma^\sfa$.

We say $\sfb$ is an \emph{iterate} of $\sfa$ via $\T$ if $\sfb$ is an $\sfa$-realizable rock, $\sfb^-$ is the last model of $\T$ and $\sigma^{\sfa, \sfb}=\pi^\T$ (hence, $\sigma^\sfa=\sigma^\sfb \circ \pi^\T)$. Similarly we define the meaning of \emph{sealing iterate}.
\end{definition}

It is now not hard to adopt the Martin-Steel iterability theorem (see \cite{MaSt94} and \cite[Theorem 2.3]{Ne10}) to realizable rocks, which we do below. The important difference between our context and the Martin-Steel context is that our models are not realized into $V$ itself but rather into a $\l$-ground of $V$. But because our iterations, when copied on $V$, are above $\l$, this causes little problems. The following two facts are the key realizability facts that we will use in this paper.

\begin{lemma}\label{lem:keyrealizability} Suppose $\chi>\l$ are inaccessible cardinals, $U$ is a $\l$-ground of $V$ as witnessed by $\mathbb{P}\in U_\l$ and $U$-generic $g\subseteq \mathbb{P}$ (thus, $V=U[g]$) and $\sfa$ is a $(U, \chi, \l)$-realizable rock. Suppose that $\T$ is a normal iteration tree on $\sfa$ such that $\lh(\T)=\omega$.\footnote{This means that $\T$ is above $\k^\sfa$.} There is then a well-founded branch $b$ of $\T$ and an embedding $k: M^\T_b\rightarrow U_\chi$ such that $(k, M^\T_b, \k^\sfa)$ is an iterate of $\sfa$ via $\T^\frown \{b\}$ (so that $\sigma^\sfa = k \circ \pi^\T_b$). 
\end{lemma}
\begin{proof}  Let $\pi: W\rightarrow V_{\chi+1}$ be an elementary embedding with $W$ countable (in $V$) and transitive and such that $U_\chi\in \rng(\pi)$ and $\rng(\sigma^\sfa)\subseteq \rng(\pi)$. Set $N=(\pi)^{-1}(U_\chi)$ and $m=\pi\restriction N$. Then there is $n: M^{\sfa}\rightarrow N$ such that $\sigma^\sfa=m\circ n$. Let now $\U$ be the $n$-copy of $\T$. Notice that because $\U$ is above $m^{-1}(\l)$\footnote{We have that $\T$ is above $\k^\sfa$ and $\sigma^\sfa(\k^\sfa)>\l$. Thus, $\U$ is above $m^{-1}(\l)$.}, $\U$ lifts to an iteration tree $\U'$ on $W$. This is because $(\pi)^{-1}(g)$ is a small forcing relative to the critical points of extenders used in $\U$\footnote{Recall that our iteration trees are nice, see Definition \ref{:def:niceext}. These are the trees used in \cite{Ne10}. See \cite[Remark 2.1]{Ne10}.}.  It now follows from the Martin-Steel Iterability Theorem (\cite[Theorem 2.3]{Ne10}) that there is a branch $b$ of $\U'$ and an elementary embedding $k': M^{\U'}_b\rightarrow V_{\chi+1}$ such that $\pi=k'\circ \pi^{\U'}_b$. 
    
A careful examination of the relationship between $\U$ and $\U'$ shows that if $h=(\pi)^{-1}(g)$ then $\pi^{\U'}_b\restriction N[h]$ is simply the lift-up of $\pi^\U_b$ to $N[h]$. More precisely, $\pi^{\U'}_b(\tau_h)=(\pi^\U_b(\tau))_h$ where $\tau_h$ and $(\pi^\U_b(\tau))_h$ are the realizations of the names $\tau$ and $\pi^\U_b(\tau)$ by $h$.

Hence, setting $k''=k'\restriction M^\U_b$, we have that $m=k''\circ \pi^\U_b$. Let now $n': M^\T_b\rightarrow M^\U_b$ be the embedding given by the copying construction (thus, $\pi^{\U}_b\circ n=n'\circ \pi^\T_b$) and set $k=k''\circ n'$. It now follows that $k: M^\T_b\rightarrow U_\chi$ and $\sigma^\sfa=k\circ \pi^\T_b$.
\end{proof}

The proof of Lemma \ref{lem:keyrealizability} can be used to prove the following lemma.

\begin{lemma}\label{lem:keyrealizabilityextender} Suppose $\chi>\l$ are inaccessible cardinals, $U$ is a $\l$-ground of $V$ as witnessed by $\mathbb{P}\in U_\l$ and a $U$-generic $g\subseteq \mathbb{P}$ (thus, $V=U[g]$) and $\sfa$ is a $(U, \chi, \l)$-realizable rock. Suppose $E\in M^\sfa$ is a nice extender such that $\sigma^\sfa(\cp(E))\geq\l$. There is then an elementary embedding $k: Ult(M^\sfa, E)\rightarrow U_\chi$ such that $\sigma^\sfa=k\circ \pi_E^{M^\sfa}$\footnote{Equivalently, $\sfb=(k, Ult(M^\sfa, E), \pi_E(\k^\sfa))$ is an $ \sfa$-realizable rock such that $\pi_E^{M^\sfa}=\sigma^{\sfa, \sfb}$.}.
\end{lemma}

Finally, to realize sealing iterates of rocks, we will need the following strengthening of Lemma \ref{lem:keyrealizabilityextender}, which uses the following definition.

\begin{definition}\label{def:relextender} Suppose $\chi>\l$ are inaccessible cardinals, $U$ is a $\l$-ground of $V$ as witnessed by $\mathbb{P}\in U_\l$ and a $U$-generic $g\subseteq \mathbb{P}$ (thus, $V=U[g]$) and $\pi: M\rightarrow U_\chi$ is an elementary embedding with $M$ transitive and countable in $V$. Suppose $E$ is an $M$-extender such that $\pi(\cp(E))>\l$. We say that $E$ is \emph{derivable from $\pi$} if there is $\xi<\chi$ and an injective order-preserving map $m: \lh(E)\rightarrow \xi$ such that for every $a\in \lh(E)^{<\omega}$ and for every $A\in \powerset([\cp(E)]^{\card{a}})\cap M$, $(a, A)\in E$ if and only if $m[a]\in \pi(A)$.

In the above situation, we let $\sigma_{E, m}^{M}: Ult(M, E)\rightarrow U_\chi$ be given by $$\sigma^{M}_{E, m}(\pi_E^{M}(f)(a))=\pi(f)(m(a)),$$ for $f\in M$, $a\in \lh(E)^{<\omega}$ and $f: [\cp(E)]^\card{a}\rightarrow M$. 
\end{definition}

It is straightforward to check that in the setup of Definition \ref{def:relextender}, $\sigma_{E, m}^{M}$ is elementary and $\pi=\sigma_{E, m}^{M}\circ \pi_E^M$. Our source of derivable extenders is Lemma \ref{lem:keyrealizabilityextender}. We apply it in the following way. Assuming the setup of Definition \ref{def:relextender}, suppose $\pi: M\rightarrow U_\xi$ and $\tau: N\rightarrow U_\xi$ are elementary embeddings such that $M$ and $N$ are transitive countable models of $\ZFC$, and for some $\nu$, $\powerset(\nu)\cap M=\powerset(\nu)\cap N$ and $\pi\restriction \powerset(\nu)^M=\tau\restriction \powerset(\nu)^N$. Suppose $E\in N$ is such that $\cp(E)=\nu$ and $E$ is nice. It follows from Lemma \ref{lem:keyrealizabilityextender} that there is $k: Ult(N, E)\rightarrow U_\chi$ such that $\tau=k\circ \pi_E^N$. Let $m=k\restriction \lh(E)$. Then $E$ is derivable from $\pi$ as witnessed by $m$. This is because $E$ is derivable from $\tau$ as witnessed by $m$ and  $\pi\restriction \powerset(\nu)^M=\tau\restriction \powerset(\nu)^N$. We summarize what we need below.

\begin{lemma}\label{lem:tworealizability} Suppose $\chi>\l$ are inaccessible cardinals, $U$ is a $\l$-ground of $V$ as witnessed by $\mathbb{P}\in U_\l$ and a $U$-generic $g\subseteq \mathbb{P}$ (thus, $V=U[g]$), $\sfa$ is a $(U, \chi, \l)$-realizable rock and $\sfb$ is an $\sfa$-realizable rock such that $\cp(\sigma^{\sfa, \sfb})>\k^{\sfa}$ (thus, $\k^\sfa=\k^\sfb$). Suppose $E\in M^\sfb$ is a nice extender with $\cp(E)=\k^\sfa$. Let $k: Ult(M^\sfb, E)\rightarrow U_\chi$ be as in Lemma \ref{lem:keyrealizabilityextender} (so $\sigma^{\sfb}=k\circ \pi_E^{M^\sfb}$) and set $\xi=k(\lh(E))$ and $m=k\rest \lh(E)$. Then 
\begin{enumerate}\itemsep0.5em
    \item $\sigma^{M^\sfa}_{E, m}:Ult(M^\sfa, E) \rightarrow U_\chi$ is an elementary embedding,
    \item $\sigma^{M^\sfa}_{E, m}\rest Ult(M^\sfa, E)|\lh(E)=k\rest Ult(M^\sfb, E)|\lh(E)$, and
    \item $(\sigma^{M^\sfa}_{E, m}, Ult(M^\sfa, E), \pi_E^{M^\sfa}(\k^\sfa))$ is an $\sfa$-realizable rock.
\end{enumerate}\vspace{0.5em}
\end{lemma}

\subsection{Supporting systems}\label{sec:ssblocks}

In this subsection we describe a set of flipping systems that we will use to show that various universally Baire sets are inside the derived model of the direct limit  model alluded before. These systems will be used much in the same way as flipping systems were used in the proofs of Lemma \ref{lem:homsysteminWg} and Lemma \ref{lem:kappauBpreserving}.

\begin{definition}\label{def:supporting system} We say that $$p=(W_2, W_1, W_0, f_1, f_0)$$ is a \emph{supporting system} if  $(f_1, W_2, W_1)$ and $(f_0, W_1, W_0)$ are flipping systems (see Figure \ref{fig:setupA}). We let $\comp(p)$ be the minimum of $\comp(f_1, W_2, W_1)$ and $\comp(f_0, W_1, W_0)$.

We say that the supporting system $(W_2, W_1, W_0, f_1, f_0)$ \emph{flips through} $(\d_0: i\leq 3)$ if the following conditions hold:
\vspace{0.5em}\begin{enumerate}\itemsep0.5em
    \item $(\d_i: i\leq 3)$ is a strictly increasing sequence of inaccessible cardinals.
    \item $(f_1, W_2, W_1)$ flips through $\d_3$.
    \item $(f_0, W_1, W_0)$ flips through $\d_2$, $\d_1$, and $\comp(f_0, W_1, W_0)\geq \d_0^+$.
\end{enumerate}\vspace{0.5em}
\end{definition}

\begin{figure}[htb]
      \begin{tikzpicture}

        \draw[-] (0,-.2) -- (0,-5);

        \draw[-] (-0.2,-1.5) -- (0.2,-1.5) node[right] {$\delta_3$};
        \draw[-] (-0.2,-2.5) -- (0.2,-2.5) node[right] {$\delta_2$};
        \draw[-] (-0.2,-3.5) -- (0.2,-3.5) node[right] {$\delta_1$};
        \draw[-] (-0.2,-4.5) -- (0.2,-4.5) node[right] {$\delta_0$};

       \node (W_2) at (1,-1) {$W_2$};
       \node (W_1) at (1,-2) {$W_1$};

       \node (W_0) at (-0.5,-4) {$W_0$};

       \path[->] (1.4,-1.1) edge[bend left] node[right] {$f_1$} (1.4,-1.9);

       \path[->] (-0.9,-2.1) edge[bend right] node[left] {$f_0$} (-0.9,-4);
       
      \end{tikzpicture}
      \caption{Illustration of the setup in $V$.}\label{fig:setupA}
    \end{figure}
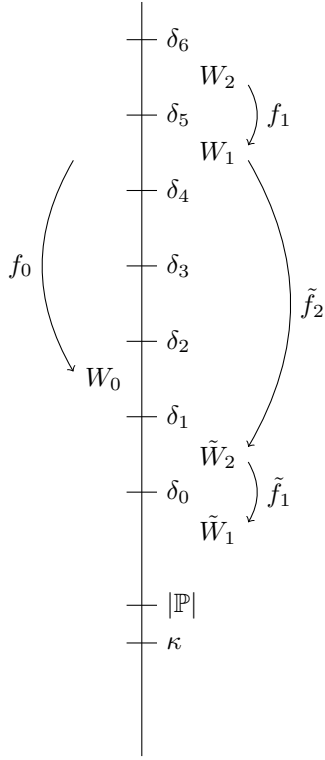

\medskip

 The next lemma can be proved by repeatedly applying Lemma \ref{lem:FlippingFunction}.

\begin{lemma}\label{lem:supporting system} Suppose $(\d_i: i\leq 3)$ is an increasing sequence of Woodin cardinals and $W_2$ is a uniform set of measures such that $\comp(W_2)\geq \d_3^+$ and $\card{W_2}<\d_0$. Then there is a supporting system $(W_2, W_1, W_0, f_1, f_0)$ that flips through $(\d_i: i\leq 3)$.
\end{lemma}

We remark that in the above lemma we do not need that $\d_0$ is a Woodin cardinal. The following lemma summarizes properties of a supporting system.

\begin{lemma}\label{lem:fancyseqproperties}
   Suppose $(W_2, W_1, W_0, f_1, f_0)$ is a supporting system that flips through $(\d_i: i\leq 3)$. Then the following holds (see Figure \ref{fig:setupA}):
   \begin{enumerate}
       \item For all ${<}\delta_2^+$-generics $G$ and all $\vec\mu \in (\TW_{W_2})^{V[G]}$, \[ \vec\mu \text{ is well-founded} \Longleftrightarrow f_1(\vec\mu) \text{ is ill-founded}. \]
       \item For all ${<}\delta_0^+$-generics $G$ and all $\vec\mu \in (\TW_{W_1})^{V[G]}$, \[ \vec\mu \text{ is well-founded} \Longleftrightarrow f_0(\vec\mu) \text{ is ill-founded}. \]
   \end{enumerate}
\end{lemma}

\subsection{Making universally Baire sets generic: The One Step Construction}

In this section, our goal is to present our method of making universally Baire sets generic.  We start by fixing 
\vspace{0.5em}\begin{enumerate}\itemsep0.5em
    \item inaccessible cardinals $\chi>\eta>\l$,
    \item a $\l$-ground $U$ of $V$ such that $U\models ``\eta$ is uB-preserving",\footnote{See Definition \ref{def:uBpreservation}.} and
    \item $W\in U_\chi$ such that in $U$, $W$ is $(\chi, \eta, \l)$-saturated and infinitely tamed above $\eta$.\footnote{See Definition \ref{def:saturated measures} and Definition \ref{def:tamed}.}
\end{enumerate} \vspace{0.5em}
Suppose $\sfa$ is a $(U, \chi, \l, \{\eta, W\})$-realizable rock. In this section, we will use superscript $\sfa$ to denote the preimages of objects in the range of $\sigma^\sfa$. Thus, $W^\sfa$ is the $\sigma^\sfa$-preimage of $W$. We will use this notation with  elementary embeddings as well. Thus, if $r: N\rightarrow U_\chi$ then $W^N$ is the $r$-preimage of $W$ (if it exists). The notation depends on the embedding as well, but omitting this dependence will not cause confusion. 

We will also use the following notation for ultrapowers. If $E$ is an $N$ extender then $\sigma_E^N$ is the ultrapower embedding and $N_E=Ult(N, E)$. 

Our first definition isolates those rocks $\sfb$ over which a given universally Baire set is generic. In this section, our definitions are relative to the objects we have fixed above.

\begin{definition}\label{def:catchingpair} Suppose $A$ is a universally Baire set and  $\sfb \in V$ is a \begin{center}$(U, \chi, \l, \{\eta, W\})$-realizable rock\end{center} such that $\eta=\sigma^\sfb(\k^\sfb)$. We say $\sfb$ \emph{b-catches}\footnote{``$b$-'' stands for ``bottom'' and ``$t$-'' stands for ``top'', and we think of the universe above $\eta$ as the ``top'' part and of the universe below $\eta$ as the ``bottom'' part.} $A$ if for some $\b<\k^\sfb$ and $M^\sfb$-generic $k\subseteq \Col(\omega, \b)$ with $k\in V$, there is a homogeneity system $\bar \mu\in M^\sfb[k]$ such that $\bar \mu\subseteq M^\sfb$, $\sigma^\sfb\pwimg \bar \mu \subseteq W$ and in $V$,\ $A=S_{\sigma^\sfb\pwimg \bar \mu}$\footnote{Letting $g$ be such that $V=U[g]$, the correct notation is $S_{(\sigma^\sfb\pwimg \bar \mu)_g}$.}.

We say $\sfb$ \emph{t-catches} $A$ if there is a $\d\in \rng(\sigma^\sfb)$ such that $\d>\eta$, $W$ is infinitely tamed above $\d$ and for some $M^\sfb$-generic $k\subseteq \Col(\omega, \d^\sfb)$ there is $\bar \mu\in M^\sfb[k]$ such that $\bar \mu\subseteq M^\sfb$, $\sigma^\sfb\pwimg \bar \mu\subseteq W$ and $A=S_{\sigma^\sfb \pwimg \bar \mu}$.
\end{definition}
We once again remark that the correct terminology should be that $\sfb$ b-catches $A$ relative to $(U, \chi, \l, \{\eta, W\})$, but since this object is fixed in this section we will omit it from our terminology. In order to show that universally Baire sets end up in derived models of rocks, we need to develop a method for b-catching universally Baire sets. The next lemma shows how to convert t-catching to b-catching. It is not too hard to t-catch universally Baire sets but this must be done in a way that preserves the fact that other sets have been b-caught. This is the step that requires the UB-Capturing Principle, and it will appear in applications of Lemma \ref{lem:presub}. The reader may profit from reviewing Definition \ref{def:uBpreservation}.

\begin{lemma}[t to b catching]\label{lem:keysealing} Suppose
\begin{enumerate}\itemsep0.5em
    \item $A$ is a universally Baire set,
    \item $\sfa$ is a  $(U, \chi, \l, \{\eta, W\})$-realizable rock such that $\sigma^\sfa(\k^\sfa)=\eta$, 
    \item  $\sfb$ is $\sfa$-realizable and t-catches $A$,
    \item $\cp(\sigma^{\sfa, \sfb})>\k^\sfa$ and $M^\sfa\cap \powerset(\k^\sfa)=M^\sfb\cap \powerset(\k^\sfb)$, 
    \item $E\in M^\sfb$ is an extender such that $\cp(E)=\k^\sfb$, $W^\sfb\in M^\sfb|\lh(E)$ and $M^\sfb\models ``E$ is a nice wf-preserving for $W^\sfb$ extender".
\end{enumerate} Then there is an $\sfa$-realizable $\sfe$ that b-catches $A$, $M^\sfe=M^\sfa_E$, $\k^\sfe=\pi_E^{M^\sfa}(\k^\sfa)$ and $\pi_E^{M^\sfa}=\sigma^{\sfa, \sfe}$ (thus, $\sigma^\sfa=\sigma^\sfe\circ \pi_E^{M^\sfa}$). 
\end{lemma}
\begin{proof} For convenience, we set $\sfa=(m, M, \k)$, $\sfb=(n, N, \k)$, $\sigma^{\sfa, \sfb}=\sigma$ (thus, $m=n\circ \sigma$), and also let $q:M_E\rightarrow N_E$ be given by $q(\pi_E^M(f)(a))=\pi_E^N(\sigma(f))(a)$ where $a\in \lh(E)^{<\omega}$ and $f\in M$ is a function $f:[\k]^{\card{a}}\rightarrow M$. Since $\cp(\sigma)>(\k^+)^{M}$, we have that\\\\
(0.1) $q$ is an elementary embedding\\
(0.2) $q\rest (M_E\cap \powerset(\pi_E^M(\k)))=id$, and\\
(0.3) $M_E\cap \powerset(\pi_E^M(\k))=N_E\cap \powerset(\pi_E^N(\k))$.\\\\
 Using the facts that $E$ is wf-preserving for $W^N$ in $N$, $q(W^N)=W^N$ (this follows from (0.2)) and $q(W^{M_E})=W^{N_E}$, we now fix $h\in M_E$ such that (see Definition \ref{def:uBpreservation}),\\\\
    (1.1) $q(h): W^N\rightarrow W^{N_E}$ and\\
    (1.2) $N_E\models ``q(h)$ is a wf-preserving map".\\\\
Set $F=n(E)$ and let $s: N_E\rightarrow \pi_F^U(U_\chi)$ be the copy map given by $s(\pi_E(f)(a))=\pi_F(n(f))(n(a))$ where $a\in \lh(E)^{<\omega}$ and $f\in N$ is a function $f:[\k]^{\card{a}}\rightarrow N$. Notice that because $U$ is a $\l$-ground and $n(\k)>\l$, we have that $\pi_F$ can be extended to act on $V$. We thus have $\pi_F^+: V\rightarrow V_F$ with the property that $\pi_F^U=\pi_F^+\rest U$. To keep the notation simple, we will identify $\pi_F^U$ with $\pi_F^+$ and call them simply $\pi_F$.

We thus have that $(\pi_F\rest U_\chi)\circ n =s\circ \pi_E^N$ and $(\pi_F\rest U_\chi)\circ m =s\circ q\circ \pi_E^M$. Next, it follows that\\\\
(2) $n\rest \lh(E)=s\rest \lh(E)$.\\\\
(2) is a consequence of the standard copying construction. Indeed, pick $\a<\lh(E)$ and let $id$ be the identity function on $\k$. We then have that $\a=\pi_E^N(id)(\a)$. Hence, 
\begin{eqnarray*}
    s(\a)&=&\pi_F(n(id))(n(\a))\\&=&n(\a).
\end{eqnarray*}
Our strategy now is to show that $s\circ q$ has the desired properties of $\sigma^\sfe$  but in $V_F$. Notice that $s\circ q\in V_F$, as $V_F$ is countably closed\footnote{It is not essential that $s\circ q\in V_F$. A standard absoluteness argument could have been used to obtain $s'\in V_F$ that resembles $s\circ q$.}. 

Let $\bar \nu$, $\d$, $\bar t$ and $k$ be such that \\\\
    (3.1) $k\in V$, $\d\in \rng(n)$ and $k\subseteq \Col(\omega, \d^N)$ is $N$-generic,\\
    (3.2) $\bar t\in N[k]$ and $\bar \nu=n\pwimg \bar t$, \\
    (3.3) (in $V$, see the footnote in Definition \ref{def:catchingpair}) $A=S_{\bar \nu}$,\\
    (3.4) $\bar \nu\subseteq W$,\\
    (3.5) $\bar t \subseteq W^N$,\\
    (3.6) $W$ is infinitely tamed above $\d$.\\\\
    The difficulty now is that $\pi_E^N(\bar t)$ does not make sense even though $\bar t \in N[k]$. This forces us to change our homogeneity systems. Notice that we still have that $\bar t \in N_E[k]$ and also $\bar t \in M_E[k]$. We now set\\\\
    (4) $\bar \mu= h \pwimg \bar t$ and $\bar \tau= s\circ q \pwimg \bar \mu$.\\\\
    It follows that $\bar \mu \in M_E[k]$, $q\pwimg \bar \mu\in N_E[k]$ and $\bar \tau \in V_F$ (as $V_F$ is closed under countable sequences), but it is no longer clear that $S_{\bar \tau}=A$. It is this that we need to prove. More specifically, we now work towards establishing the following.\\\\
    (*) The following holds in $V_F$.
    \begin{enumerate}[(a)]\itemsep0.5em
\item $\pi_F(m)=s\circ q \circ \pi_E^M$,
\item for some $\b<\pi_E^M(\k)$, $k$ is $M_E$-generic for $\Col(\omega, \b)$, 
\item $\bar \mu\subseteq M_E$ and $\bar \mu\in M_E[k]$, 
\item $s\circ q \pwimg \bar \mu \subseteq \pi_F(W)$ and 
\item (again see the footnote in Definition \ref{def:catchingpair}) $A=S_{s\circ q \pwimg \bar \mu}(=S_{\bar \tau})$.
\end{enumerate}\vspace{0.5em}
Clause (e) above is the only clause that needs to be verified\footnote{In Clause (b), $\b=\d^N$.}, which is what we do now.
    
Notice that it follows from (2) that \\\\
(5) $\bar \nu=n\pwimg \bar t= s\pwimg \bar t$.\\\\
Therefore, setting $l=s\circ q$, we get that\footnote{Notice that $\cp(q)>\pi^M_E(\k)$, implying that $q\pwimg \bar t =\bar t$.}
\begin{eqnarray*}
    \bar \tau &=& l \pwimg \bar \mu\\
    &=&  l \pwimg (h\pwimg \bar t)\\
    &=&  l(h)\pwimg (l\pwimg \bar t)\\
    &=& l(h) \pwimg \bar \nu\\
    &=& l(h)(\bar \nu).
\end{eqnarray*}
We thus have that\\\\ 
(6) $\bar \tau = s(q(h))(\bar \nu)$.\\\\
It now follows from (3.3) and (1.2) that\\\\
(7) $V_F\models S_{\bar \tau}=S_{\bar\nu}=A$.\\\\
This finishes the proof of (e). We thus have that\\\\
(8) in $V_F$, there is $w: M_E\rightarrow \pi_F(U_\chi)$ such that 
\begin{enumerate}[(a)]\itemsep0.5em
\item $\pi_F(m)=w\circ \pi_E^M$,
\item for some $\b<\pi_E^M(\k)$, $k$ is $M_E$-generic for $\Col(\omega, \b)$, 
\item $\bar \mu\subseteq M_E$ and $\bar \mu \in M_E[k]$, 
\item $w \pwimg \bar \mu \subseteq \pi_F(W)$ and 
\item $A=S_{w \pwimg \bar \mu }$.
\end{enumerate}\vspace{0.5em}
Set $\sfe=(w, M_E, \pi_E^M(\k))$. Notice that $\pi_F(\sfa)=(\pi_F(n), M, \k)$. It then follows from (8) that\\\\
(9) in $V_F$, 
\begin{enumerate}[(a)]\itemsep0.5em
\item $\sfe$ is $\pi_F(\sfa)$-realizable, $M^\sfe=M^{\pi_F(\sfa)}_E$ and $\pi_E^{M^{\pi_F(\sfa)}}=\sigma^{\pi_F(\sfa), \sfe}$ (thus, $\sigma^{\pi_F(\sfa)}=\sigma^\sfe\circ \pi_E^{M^\sfa}$),
\item for some $\b<\k^{\pi_F(\sfa)}$, $k\in V$ is $M^{\pi_F(\sfa)}$-generic for $\Col(\omega, \b)$, 
\item $\bar \mu\subseteq M^\sfe$ and $\bar \mu\in M^\sfe[k]$, 
\item $\sigma^\sfe\pwimg \bar \mu \subseteq \pi_F(W)$ and 
\item in $V_F$,\ $A=S_{\sigma^\sfe \pwimg \bar \mu}$ (see the footnote in Definition \ref{def:catchingpair}).
\end{enumerate}\vspace{0.5em}
Because $\pi_F$ is elementary, (9) implies the conclusion of Lemma \ref{lem:keysealing}.
\end{proof}

Our last lemma of the subsection shows that when we ``change the model" we do not ``lose" our accomplishments.

\begin{lemma}\label{lem:presub} Suppose $A$ is a universally Baire set, $\sfb$ is a $(U, \chi, \l, \{\eta, W\})$-realizable rock such that $\sigma^\sfb(\k^\sfb)=\eta$ and $\sfb$ b-catches $A$, and $\sfd$ is a $\sfb$-realizable rock such that $\cp(\sigma^{\sfb, \sfd})\geq \k^\sfb$. Let $\b<\k^\sfb$, $k\subseteq \Col(\omega, \b)$ and $\bar\mu\in M^\sfb[k]$ witness that $\sfb$ b-catches $A$. Then $\sfd$ b-catches $A$ as witnessed by $(\b, k, \sigma^{\sfb, \sfd}\pwimg \bar \mu)$.
\end{lemma}
\begin{proof}
    Because $\cp(\sigma^{\sfb, \sfd})\geq \k^\sfb$ and because $\sigma^\sfb\pwimg \bar \mu=\sigma^\sfd \pwimg (\sigma^{\sfb, \sfd}\pwimg \bar \mu)$, we only need to prove that $\sigma^{\sfb, \sfd}\pwimg \bar \mu\in M^\sfd[k]$. To see this, notice that $\sigma^{\sfb, \sfd}$ can be lifted to $M^\sfb[k]$ (because $\cp(\sigma^{\sfb, \sfd})\geq \k^\sfb$ and $\b<\k^\sfb$). Let $j: M^\sfb[k]\rightarrow M^\sfd[k]$ be this embedding. We thus have that $j\rest M^\sfb=\sigma^{\sfb, \sfd}$. It then follows that $\sigma^{\sfb,\sfd}\pwimg \bar \mu=j\pwimg \bar \mu = j(\bar\mu) \in M^\sfd[k]$ (because $\bar \mu$ is a countable set in $M^\sfb[k]$).
    \end{proof}



\subsection{Summary} Here we summarize the results of the previous subsections. Again we start by letting 
\vspace{0.5em}\begin{enumerate}\itemsep0.5em
    \item $\chi>\eta>\l$ be inaccessible cardinals,
    \item $U$ be a $\l$-ground of $V$
    such that $U\models ``\eta$ is uB-preserving",\footnote{See Definition \ref{def:uBpreservation}.}
     \item $W\in U_\chi$ be such that in $U$, $W$ is $(\chi, \eta, \l)$-saturated and infinitely tamed above $\eta$\footnote{See Definition \ref{def:saturated measures} and Definition \ref{def:tamed}.},
    \item $\sfa$ be a $(U, \chi, \l, \{\eta, W\})$-realizable rock such that $\sigma^\sfa(\k^\sfa)=\eta$.
\end{enumerate} \vspace{0.5em}
The meanings of these objects are fixed until Theorem \ref{thm:summarysec5}, which will reintroduce some of these objects.
\begin{definition}\label{def:inducedrock} Suppose $X\prec U_\chi$ is such that $(\eta, \l)\in X$. We then say that $\sfb$ is the \emph{$(U, \chi, \l , \{\eta\})$-realizable rock induced by $X$} (or just the rock induced by $X$) if $M^\sfb$ is the transitive collapse of $X$, $\sigma^\sfb$ is the inverse of the transitive collapse and $\sigma^\sfb(\k^\sfb)=\eta$.
\end{definition}



We will catch universally Baire sets by making their \textit{codes} generic. Again the definitions below are relative to the objects we have fixed. 

\begin{definition}\label{def:codeofA0}
Suppose $A$ is a universally Baire set. We then say that $(y, z)$ is a \emph{code} (or an $\sfa$-code) of $A$ if the following condition holds: for some supporting system  $p=(W_2, W_1, W_0, f_1, f_0)$\footnote{See Definition \ref{def:supporting system}.} and an increasing sequence of inaccessible cardinals $(\d_i: i\leq 3)$, 
\vspace{0.5em}\begin{enumerate}\itemsep0.5em
\item $p$ flips through $(\d_i: i\leq 3)$,
\item $\eta<\d_0<\d_3<\chi$,
\item $W=W_2$ and $W$ is infinitely tamed above $\d_3$,
\item $\{ p, (\d_i: i\leq 3)\}\in \rng(\sigma^\sfa)$,
\item $y\in \bR^V$ codes a bijection  $u_y: \omega \rightarrow M^\sfa$, and
\item $z\in \powerset(\omega)^V$ is such that $\bar \nu = (u_y(i): i \in z)$ is a homogeneity system, $\bar \mu=_{def}\sigma^\sfa\pwimg \bar \nu\subseteq W_0$ and $(S_{\bar \mu})^V=A$.
\end{enumerate}\vspace{0.5em}
\end{definition}
Clearly $A$ may not have an $\sfa$-code. Next we introduce $A$-genericity pairs. 

\begin{definition}\label{def:geniterate0} Suppose $A$ is a universally Baire set, $(y, z)$ is a code of $A$ and $\sfb$ is an $\sfa$-realizable rock. We then say that $(\sfb, E)$ is an \emph{$(A\mid y, z)$-genericity $\sfa$-realization} if letting $(p, (\d_i: i\leq 3))$ witness that $(y, z)$ is an $\sfa$-code,
\vspace{0.5em}\begin{enumerate}\itemsep0.5em
    \item $\k^\sfb=\k^\sfa$ (this follows from the clause below),
    \item $\cp(\sigma^{\sfa, \sfb})>\d_1^\sfa$,\footnote{Recall that $\d_1^\sfa = (\sigma^\sfa)^{-1}(\delta_1)$, etc.} 
    \item for some $k\subseteq \Col(\omega, \d_2^\sfb)$, $(y, z)\in M^\sfb[k]$.
    \item $E\in M^\sfb$ is such that $\cp(E)=\k^\sfb(=\eta^\sfb)$, $\d^\sfb_3<\lh(E)$ and $M^\sfb\models ``E$ is a nice wf-preserving for $W^\sfb$ extender". 
\end{enumerate} \vspace{0.5em}
\end{definition}

The following is an easy fact. The reader may wish to review Definition \ref{def:catchingpair}.

\begin{lemma}\label{lem:genbcatch} Suppose $A$ is a universally Baire set, $(y, z)$ is a code of $A$ and $(\sfb, E)$ is an $(A \mid y, z)$-genericity $\sfa$-realization. Then $\sfb$ t-catches $A$ (relative to $(U, \chi, \l, \{\eta, W \})$).
\end{lemma}
\begin{proof} 
    Let $(p, (\d_i: i\leq 3))$ witness that $(y, z)$ is an $\sfa$-code and $k\subseteq \Col(\omega, \d_2^\sfb)$ be $M^\sfb$-generic such that $(y, z)\in M^\sfb[k]$. Let then $\bar \nu=(u_y(i): i\in z)$ be as in Definition \ref{def:codeofA0}. We have that $\bar \nu\in M^\sfb[k]$ and $A= S_{\sigma^\sfb\pwimg \bar \nu}$. Let $v=(f_1\circ f_0)^{-1}$ and set $\bar \mu=v^{\sfb}(\bar \nu)$. Then $\bar \mu\in M^\sfb[k]$ and $\sigma^\sfb\pwimg \bar \mu\subseteq W_2(=W)$. To show that $(k, \bar \mu)$ witnesses that $\sfb$ t-catches $A$, it is enough to show that $A=S_{\sigma^\sfb \pwimg \bar \mu}$. This, however, easily follows from the fact that $(f_1, W_2, W_1)$ and $(f_0, W_1, W_0)$ are flipping systems.
\end{proof}

If $(\sfb, E)$ is an $(A\mid y, z)$-genericity $\sfa$-realization then we will informally say that $E$ is the $A$-sealing extender or $E$ seals $A$. The reason for the terminology is that we will show that $A$ is in the derived model of $M^\sfa_E$. We will obtain an $A$-genericity $\sfa$-realization $(\sfb, E)$ by performing genericity iterations. In particular, $\sfb$ will be an iterate of $\sfa$ via  an iteration tree $\T$ of length $\omega+1$. The following theorem is the summary of what we have shown. For convenience, we will re-introduce all previously fixed objects. The reader may wish to review Definition \ref{def:catchingpair}.

\begin{theorem}\label{thm:summarysec5}
 Suppose $\chi>\eta>\l$ are inaccessible cardinals and the UB-Capturing Principle holds at $(\chi, \eta, \l)$ as witnessed by $(X, W)$. 
 Let $W'\in X$ be as in Clause \eqref{cl:3_def:UBCapturingPrinciple} of Definition \ref{def:abstractblock} and $g$ be ${<}\l$-generic such that $\card{X}^{V[g]}=\aleph_0$. Then the following holds in $V[g]$\footnote{We apply the concepts we have developed to $U=V$ and $V=V[g]$. Clearly, $U$ is a $\l$-ground of $V$.}.
 
Suppose $B\in \Gamma^\infty_g$, $\vec{A}\in (\Gamma^\infty_g)^{<\omega_1}$\footnote{This is $\omega_1$ of $V[g]$.}, $z\subseteq W$ is countable, $a\subseteq \eta$ is bounded and $\sfa$ is the rock induced by $X(z\cup a)$. Assume $\sfa$ b-catches every member of $\vec{A}$ (relative to $(U, \chi, \l, \{\eta, W'\})$). There is then an $\sfa$-realizable $\sfb$, an $\sfa$-realizable $\sfe$ and $E\in M^\sfb$ such that
\vspace{0.5em}\begin{enumerate}\itemsep0.5em
 \item $\sfb$ t-catches $B$ (relative to $(U, \chi, \l, \{\eta, W'\})$), 
 \item $\sigma^{\sfa, \sfb}\rest (M^\sfa\cap \powerset(\k^\sfa))=id$,
 \item $(W')^\sfb\in M^\sfb|\lh(E)$ and $M^\sfb\models ``E$ is a nice wf-preserving for $(W')^\sfb$ extender",
 \item $M^\sfe=Ult(M^\sfa, E)$ and $\sigma^{\sfa, \sfe}=\pi_E^{\sfe}$, and
 \item $\sfe$ b-catches every member of $\vec A$ and $B$ (relative to $(U, \chi, \l, \{\eta, W'\})$).
 \end{enumerate}\vspace{0.5em}
 \end{theorem}
 \begin{proof}
     The proof is an application of Lemmas \ref{lem:keysealing} and \ref{lem:presub}. First, applying Definition \ref{def:abstractblock}, we find a countable $z'\subseteq W$ such that $(X(z\cup z'\cup a), W)$ captures $B$. Let $\sfb'$ be the rock induced by $X(z\cup z'\cup a)$. Because $X$ is $(W\mid \eta)$-full,\footnote{See Definition \ref{def:wetafull}.} we have that $\sfb'$ is $\sfa$-realizable and $\sigma^{\sfa, \sfb'}\rest (M^\sfa\cap \powerset(\k^\sfa))=id$. 

     Let $\{p, (\d_i: i\leq 3)\}\in X$ be such that the following holds:
 \vspace{0.5em}\begin{enumerate}[(a)]\itemsep0.5em
 \item for each $i\leq 3$, $\d_i$ is a Woodin cardinal,
 \item $\eta<\d_0<\d_3<\chi$,
 \item $p=(W_2, W_1, W_0,  f_1, f_0)\in V_\chi$ is a supporting system that flips through $(\d_i: i\leq 3)$,\footnote{See Definition \ref{def:supporting system}.} and
 \item $W'=W_2$ and $W'$ is infinitely tamed above $\d_3$\footnote{The existence of such a $p$ follows from Lemma \ref{lem:supporting system}.}.
 \end{enumerate}\vspace{0.5em}
     
     We now have that $B$ has a $\sfb'$-code. Indeed, fix $\bar \tau\subseteq \rng(\sigma^{\sfb'})\cap W$ such that $S_{\bar \tau}=B$. Let $v=f_2\circ f_1$ and $w=v^{-1}$. Set $\bar \mu=v(\bar \tau)$. 
     Let now $(r, s)$ be such that $r$ codes a bijection $u_r: \omega \rightarrow M^{\sfb'}$ and $s\subseteq \omega$ is such that letting $\bar \nu=(u_r(i): i\in s)$, $\bar \mu=\sigma^{\sfb'}\pwimg \bar \nu$. Clearly $(r, s)$ is a $\sfb'$-code of $B$.

     Next, using Neeman's Genericity Theorem, we find a $\sfb'$-realizable $\sfb$ such that $\cp(\sigma^{\sfb', \sfb})>\d^{\sfb'}_1$ and for some $M^{\sfb}$-generic $k\subseteq \Col(\omega, \d_2^{\sfb})$, $(r, s)\in M^{\sfb}[k]$. It now follows from Lemma \ref{lem:genbcatch} that $\sfb$ t-catches $B$. 

     At this stage, we have that $\sfb$ is $\sfa$-realizable, t-catches $B$ and $\sigma^{\sfa, \sfb}\rest (M^\sfa\cap \powerset(\k^{\sfa}))=id$. Applying Lemma \ref{lem:keysealing} we get an extender $E\in M^\sfb$ such that $\cp(E)=\k^{\sfa}$, $W_2^{\sfb}\in M^\sfb|\lh(E)$ and $M^\sfb\models ``E$ is a nice wf-preserving for $W_2^{\sfb}$ extender$"$ and an $\sfa$-realizable $\sfe$ such that $M^\sfe=Ult(M^\sfa, E)$, $\sigma^{\sfa, \sfe}=\pi_E^{M^\sfa}$ and $\sfe$ b-catches $B$. Finally it follows from Lemma \ref{lem:presub} that $\sfe$ b-catches every member of $\vec{A}$.
 \end{proof}

Notice that in the above proof, $(\sfb, E, \sfe)$ depends on the choice of $p$, which is why below we use the term ``possible".

\begin{definition}[The One Step Construction]\label{def:onestepcon} In the context of Theorem \ref{thm:summarysec5}, we say that $(\sfb, E, \sfe)$ is \emph{a possible output of the one-step-construction applied to $(\sfa, \vec{A}, B, X, W, W', U, \chi, \eta, \l)$}.
\end{definition}

\section{A Derived Model Representation for $L(\Gamma^\infty, \bR)$.} \label{sec:summarydmrep}

In this section, our goal is to produce a linear system $\sf{DM}$ of rocks with the property that $L(\Gamma^\infty, \bR)$ can be realized as the derived model of the direct limit model of $\sf{DM}$. The next definition introduces the $\sf{DM}$-systems.

\begin{definition}\label{def:predmlinearsys}
Suppose $\chi>\eta>\l$ are inaccessible cardinals, $U$ is a $\l$-ground and $U\models ``$the UB-Capturing Principle holds at $(\chi, \eta, \l)$ as witnessed by $(X, W)$". Suppose $i\in \{0, 1\}$, $k\subseteq \Col(\omega_i, \card{\Gamma^\infty})$ is $V$-generic and $W'\in X$ is such that $W\subseteq W'$ and $W'$ witnesses Clause \eqref{cl:3_def:UBCapturingPrinciple} of Definition \ref{def:abstractblock} applied to $(X, W)$. Suppose $\iota\leq \omega_1$ is a limit ordinal and $k'\in V[k]$ is a surjection $k':\iota\rightarrow \Gamma^\infty$ with the property that for all $\a<\iota$, $k'\rest \a\in V$. We say $D=((\sfa_\a, \sfb_\a, E_\a) \mid \a<\iota)$ is a \emph{$\sf{DM}$-sequence for $V$ supported by $( X, W, W', U, \chi, \eta, \l, k')$} if in $V[k]$,
   \begin{enumerate}\itemsep0.5em
   \item $\sfa_0$ is the rock induced by $X$, 
   \item for $\a<\iota$, $\sfa_{\a}$ b-catches every member of $k\rest \a$ relative to $(U, \chi, \l, \{\eta, W'\})$,
    \item for every $\a<\iota$, $(\sfb_\a, E_\a, \sfa_{\a+1})$ is a possible output of the one-step-construction applied to $(\sfa_\a, k'\rest \a, k'(\a), X, W, W', U, \chi, \eta, \l)$,\footnote{Thus, $\sfb_\a$ t-catches $A_\a$, $E_\a$ is a wf-preserving extender for some $W_2^\sfb$ as in the proof of Theorem \ref{thm:summarysec5} and $\sfa_{\a+1}$ b-catches every member of $k\rest \a+1$.}
       \item for every $\a<\iota$, $D\rest \a\in V$,
       \item for every $A\in \Gamma^\infty$\footnote{While this statement is evaluated in $V[g]$, it applies to $A\in V$.} there is an $\a<\iota$ such that if $\sfa_\a=(\sigma_{X_\a}, M_{X_\a}, \k_{X_\a})$ is the rock induced by $X_\a$ then $\sfa_\a$ b-catches $A$ relative to $(U, \chi, \l, \{\eta, W'\})$,\footnote{Since $\{x\}$ is a universally Baire set for every $x \in \bR$, we have that for each $x\in \bR$ there is an $\a<\iota$ and an $M^{\sfa_\a}$-generic $g\subseteq \Col(\omega, {<}\k^{\sfa_\a})$ such that $x\in M^{\sfa_\a}[g]$.}
       \item for every limit ordinal $\a<\iota$, $\sfa_\a$ is a direct limit of $(\sfa_\b\mid  \b<\a)$.\footnote{More precisely, $\sfa_\a$ is the $\sfa_0$-realizable rock induced by $\cup_{\b<\a}\rng(\sigma^\sfa_\a)$.}
   \end{enumerate} 
\end{definition}

The expression ``for $V$" is somewhat redundant, but omitting it may cause confusion when applying this notion in generic extensions of $V$ (see Theorem \ref{thm:exdmomega}).

The following is our main theorem on the existence of $\sf{DM}$-sequences. The reader may wish to review Definition \ref{flipping system}, Definition \ref{def:uBpreservation}, Lemma \ref{lem:kappauBpreserving}, Definition \ref{def:abstractblock} (The UB-Capturing Principle), Theorem \ref{thm:ubcapturingprinciple}, Definition \ref{def:supporting system}, Definition \ref{def:codeofA0}, Lemma \ref{lem:keysealing} and Theorem \ref{thm:summarysec5}. We state two versions of the theorem. In the first version, we collapse $\card{\Gamma^\infty}$ to be $\omega$ and in the second, we collapse it to be $\omega_1$. The proofs are straightforward applications of Theorem \ref{thm:summarysec5} and Lemma \ref{lem:presub}, and we leave them to the reader.

\begin{theorem}[The $\omega$ version]\label{thm:exdmomega} Suppose $\chi>\eta>\l$ are inaccessible cardinals, $U$ is a $\l$-ground, $U\models ``$the UB-Capturing Principle holds at $(\chi, \eta, \l)$ as witnessed by $(X, W)$", and $W'\in X$ is such that $W\subseteq W'$ and $W'$ witnesses Clause \eqref{cl:3_def:UBCapturingPrinciple} of Definition \ref{def:abstractblock} applied to $(X, W)$. Suppose  $k\subseteq \Col(\omega, \card{\Gamma^\infty})$ is $V$-generic, $\iota<\omega_1$ is a limit ordinal and $k':\iota \rightarrow \Gamma^\infty$ is a surjection in $V[k]$ with the property that for all $\a<\iota$, $k'\rest \a\in V$. Then, in $V[k]$, there is a $\sf{DM}$-sequence $D=((\sfa_\a, \sfb_\a, E_\a) \mid \a<\iota)$ for $V$ supported by $( X, W, W', U, \chi, \eta, \l, k')$.
\end{theorem}
\begin{theorem}[The $\omega_1$ version]\label{thm:exdmomega1} Suppose $\chi>\eta>\l$ are inaccessible cardinals, $U$ is a $\l$-ground, $U\models ``$the UB-Capturing Principle holds at $(\chi, \eta, \l)$ as witnessed by $(X, W)$" and $W'\in X$ is such that $W\subseteq W'$ and $W'$ witnesses Clause \eqref{cl:3_def:UBCapturingPrinciple} of Definition \ref{def:abstractblock} applied to $(X, W)$. Suppose  $k\subseteq \Col(\omega_1, \card{\Gamma^\infty})$ is $V$-generic. Then, in $V[k]$, there is a $\sf{DM}$-sequence $D=((\sfa_\a, \sfb_\a, E_\a) \mid \a<\omega_1)$ for $V$ supported by $( X, W, W', U, \chi, \eta, \l, k)$.
\end{theorem}

We end this section by presenting our derived model representation for $\Gamma^\infty$. The following terminology will be used in Theorem \ref{thm:dermodelrepomega}.

\begin{terminology}\label{term:underseq}  Assume the set up of Definition \ref{def:predmlinearsys} and suppose $$D=((\sfa_\a, \sfb_\a, E_\a) \mid \a<\iota)$$ is a ${\sf{DM}}$-sequence for $V$ supported by $( X, W, W', U, \chi, \eta, \l, k')$. We say $${\sf{dm}}=(M_\a, \sigma_{\a, \b}, \sigma_\a, \k_\a:\a<\b<\iota)$$ is the \emph{underlying sequence} of $D$ if for $\a<\b<\iota$, $M_\a=M^{\sfa_\a}$, $\sigma_{\a, \b}=\sigma^{\sfa_\a, \sfa_\b}$, $\sigma_\a=\sigma^{\sfa_\a}$, $\k_\a=\k^{\sfa_\a}$. Moreover, $M_\iota$ is the direct limit of $(M_\a, \sigma_{\a, \b}: \a<\b<\iota)$, $\sigma_{\a, \iota}:M_\a\rightarrow M_\iota$ is the direct limit embedding, and $\sigma_\iota: M_\iota\rightarrow U_\chi$ is the canonical realizability embedding given by $\sigma_\iota(x)=\sigma_\a(\bar x)$ where  $\a<\iota$ and $\bar x\in M_\a$ are such that $\sigma_{\a, \iota}(\bar x)=x$. We call $M_\iota$ the direct limit of ${\sf{dm}}$.
\end{terminology}

The reader may wish to review Definition \ref{def: smu} and Theorem \ref{woodin: der model thm}. We will present two versions of the theorem, one for $\Col(\omega, \Gamma^\infty)$ and one for $\Col(\omega_1, \Gamma^\infty)$.

\begin{theorem}[$\omega$-version]\label{thm:dermodelrepomega}
    Suppose $\chi>\eta>\l$ are inaccessible cardinals, $U$ is a $\l$-ground, $U\models ``$the UB-Capturing Principle holds at $(\chi, \eta, \l)$ as witnessed by $(X, W)$" and $W'\in X$ is such that $W\subseteq W'$ and $W'$ witnesses Clause \eqref{cl:3_def:UBCapturingPrinciple} of Definition \ref{def:abstractblock} applied to $(X, W)$. Suppose $k\subseteq \Col(\omega, \card{\Gamma^\infty})$ is $V$-generic, $\iota < \omega_1^V$ is a limit ordinal, $k'\in V[k]$ is a surjection $k':\iota\rightarrow \Gamma^\infty$ with the property that for all $\a<\iota$, $k'\rest \a\in V$ and $$D=((\sfa_\a, \sfb_\a, E_\a) \mid \a<\iota)$$ is a ${\sf{DM}}$-sequence for $V$ supported by $( X, W, W', U, \chi, \eta, \l, k')$. Let $${\sf{dm}}=(M_\a, \sigma_{\a, \b}, \sigma_\a, \k_\a:\a<\b<\iota)$$ be the \emph{underlying sequence} of $D$, $M_\iota$ be the direct limit of ${\sf{dm}}$ and $\sigma_\iota: M_\iota \rightarrow U_\chi$ be the canonical realizability embedding. Then, in $V[k]$, there is an $M_\iota$-generic $h\subseteq \Col(\omega, {<}\omega_1^V)$ such that  $$\bR^{M_\iota[h]}=\bR^V$$ and  $$(\Gamma^\infty, \bR)^\#=((Hom^*, \bR^*)^\#)^{M_\iota[h]}.$$
\end{theorem}
\begin{proof}
Notice that we have that $\k_\iota=\omega_1^V$. The existence of $h$ follows from \cite[Lemma 3.1.5]{La04}. Because $\sigma_\iota: M_\iota\rightarrow U_\chi$ and there are Woodin cardinals in $U_\chi$ (see Clause \eqref{cl:3_def:UBCapturingPrinciple} of Definition \ref{def:abstractblock}), $M_\iota[h]\models ``(Hom^*, \bR^*)^\#$ exists$"$ and $((Hom^*, \bR^*)^\#)^{M_\iota[h]}=(Hom^*, \bR^*)^\#$. It is then enough to prove that if $h\subseteq \Col(\omega, {<}\omega_1^V)$ is $M_\iota$-generic such that $h\in V[k]$ and $\bR^{M_\iota[h]}=\bR^V$ then $$\Gamma^\infty=(Hom^*)^{M_\iota[h]}.$$
For $\a<\iota$, let $h_\a=\Col(\omega, {<}\k_\a)\cap h$. We then have that for $\a<\b\leq \iota$, $\sigma_{\a, \b}:M_\a\rightarrow M_\b$ can be lifted to $\sigma_{\a, \b}^+: M[h_\a]\rightarrow M[h_\b]$.
 
 First fix $A\in \Gamma^\infty$. We want to show that $A\in (Hom^*)^{M_\iota[h]}$. Let $\a<\iota$ be such that $\sfa_\a$ b-catches $A$ relative to $(U, \chi, \l, \{\eta, W'\})$. Let $\b<\k_\a$ and $f\subseteq \Col(\omega, \b)$ be such that $f\in V$, $f$ is $M_\a$-generic and there is $\bar\mu\in M_\a[f]$ such that $\bar \nu=_{def}\sigma_\a\pwimg \bar \mu \subseteq W'$ and $S_{\bar \nu}=A$. We can then find some $\gamma\in [\a, \iota)$ such that $f\in M_\gamma[h_\gamma]$. Because we have that $\sigma_{\a, \gamma}$ lifts to $\sigma_{\a, \gamma}^f: M_\a[f]\rightarrow M_\gamma[f]$ and because $M_\gamma[f]\subseteq M[h_\gamma]$, we have that letting $\bar \tau =\sigma_{\a, \gamma}^+\pwimg \bar \mu$,\\\\ (1) $\bar \tau\in M_\gamma[h_\gamma]$, $\bar \nu=\sigma_\gamma\pwimg \bar \tau$ and $A=S_{\sigma_\gamma\pwimg \bar \tau}$.\\\\
 Let now for $\xi\in (\gamma, \iota)$, $\bar \tau_\xi=\sigma_{\gamma, \xi}^+(\bar \tau)$ and $A_\xi=S_{\sigma_\xi\pwimg \bar \tau_\xi}\cap M[h_\xi]$. We have that\\\\
 (2) $A=\bigcup_{\xi\in (\gamma, \iota)}A_\xi$ and $A_\xi=(S_{\bar \tau_\xi})^{M_\xi[h_\xi]}$ (see Lemma \ref{lem:keylemmatamed}).\\\\
 It follows from (2) that\\\\
 (3) for all $\xi<\xi'$ belonging to $(\gamma, \iota)$, $A_\xi \in M_\xi[h_\xi]$ and $\sigma_{\xi, \xi'}^+(A_\xi)=A_{\xi'}$.\\\\
 (2) and (3) imply that for every $\xi\in (\gamma, \iota)$, $A=\sigma_{\xi, \iota}(A_\xi)$. Hence, $A\in M_\iota[h]$ and $A=(S_{\sigma_{\gamma, \iota}^+ \pwimg \bar \tau})^{M_\iota[h]}$, implying that $A\in (Hom^*)^{M_{\iota}[h]}$.

 The argument showing that $(Hom^*)^{M_\iota[h]}\subseteq \Gamma^\infty$ is just the proof of Subclaim 1.1 of \cite{StstattowfreeDMT}.
\end{proof}

Next we state the $\omega_1$ version. Since its proof is very similar, we leave it to the reader.

\begin{theorem}[$\omega_1$-version]\label{thm:dermodelrepomega1}
    Suppose $\chi>\eta>\l$ are inaccessible cardinals, $U$ is a $\l$-ground, $U\models ``$the UB-Capturing Principle holds at $(\chi, \eta, \l)$ as witnessed by $(X, W)$" and $W'\in X$ is such that $W\subseteq W'$ and $W'$ witnesses Clause \eqref{cl:3_def:UBCapturingPrinciple} of Definition \ref{def:abstractblock} applied to $(X, W)$. Suppose $k\subseteq \Col(\omega_1, \card{\Gamma^\infty})$ is $V$-generic and $$D=((\sfa_\a, \sfb_\a, E_\a) \mid \a<\omega_1)$$ is a ${\sf{DM}}$-sequence for $V$ supported by $( X, W, W', U, \chi, \eta, \l, k)$. Let $${\sf{dm}}=(M_\a, \sigma_{\a, \b}, \sigma_\a, \k_\a:\a<\b<\omega_1)$$ be the underlying sequence of $D$, $M_{\omega_1}$ be the direct limit of ${\sf{dm}}$ and $\sigma_{\omega_1}: M_{\omega_1} \rightarrow U_\chi$ be the canonical realizability embedding. Then, in $V[k]$, there is an $M_{\omega_1}$-generic $h\subseteq \Col(\omega, {<}\omega_1^V)$ such that  $$\bR^{M_{\omega_1}[h]}=\bR^V$$ and  $$(\Gamma^\infty, \bR)^\#=((Hom^*, \bR^*)^\#)^{M_{\omega_1}[h]}.$$
\end{theorem}

 \begin{remark}[$\omega_1$-version]\label{length omega1 case} This remark constructs the generic $h$ used in Theorem \ref{thm:dermodelrepomega1}. Notice that because we are working in a universe that has a class of Woodin cardinals, countably closed posets do not add new universally Baire sets as they do not add new homogeneity systems (see Theorem \ref{thm:msw}). Let $\vec{x}=(x_i \mid i<\a)\in V[k]$ be an enumeration of $\bR(=\bR^V)$ and let $f:\omega_1\rightarrow \omega_1$ be any increasing continuous function in $V[k]$ such that for every $\b<\omega_1$ there is an $M_{f(\b)}$-generic $g\subseteq \Col(\omega, {<}\k_{f(\b)})$ such that $x_\b\in M_{f(\b)}[g]$. Let $\mathbb{P}$ consist of functions $p$ such that 
 \vspace{0.5em}\begin{enumerate}\itemsep0.5em
 \item $dom(p)\in \omega_1$ and $dom(p)$ has a largest element,
 \item for each $\b\in dom(p)$, $p(\b)$ is $M_{f(\b)}$-generic for $\Col(\omega, {<}\kappa_{f(\b)})$ such that $\vec{x}\restriction \b\subseteq M_{f(\b)}[p(\b)]$,
 \item if $\b_0<\b_1$ are in $dom(p)$ then $p(\b_0)=p(\b_1)\cap \Col(\omega, {<}\kappa_{f(\b_0)})$,
 \item if $(\b_i \mid i<\omega) \subseteq dom(p)$ is an increasing sequence, then $p(\b_\omega)=\bigcup_{i<\omega}p(\b_i)$\footnote{Notice that $p(\b_\omega)$ is indeed $M_{f(\b_\omega)}$-generic such that $\vec{x}\restriction \b_\omega\subseteq M_{f(\b_\omega)}[p(\b_\omega)]$. This follows from the direct limit construction. Each dense set in $M_{f(\b_\omega)}$ has a preimage in some $M_{\b_i}$.}.
 \end{enumerate}\vspace{0.5em}
 The order of $\mathbb{P}$ is by extension. It follows from the proof of Claim 3 that appears in the proof of \cite[Lemma 3.1.5]{La04} that each $p\in \mathbb{P}$ has a proper extension in $\mathbb{P}$. Notice now that $\mathbb{P}$ is countably closed\footnote{This is because if $(p_i: i<\omega)$ is a decreasing sequence of conditions then letting $\b=\sup\{\gamma: \exists i\in \omega (\gamma\in dom(p_i))\}$ and $g=\bigcup_{i<\omega}p_i(\b_i)$ where $\b_i$ is the largest element of $dom(p_i)$, $g\subseteq \Col(\omega, {<}\kappa_\b)$ is $M_\b$-generic.}. Let then $\xi$ be large and $\nu : H\rightarrow V_\xi[k]$ be an elementary embedding such that in $V[k]$, $H$ is transitive, $\card{H}=\aleph_1$, $H^\omega\subseteq H$ and $\{\Gamma^\infty, \mathbb{P}, D \}\subseteq \rng(\nu)$. We now have $h'\in V[k]$ which is $H$-generic for $\mathbb{P}$. Let then $h=\bigcup_{p\in h', \b\in dom(p)} p(\b)$. It follows that $h$ is $M_{\omega_1}$-generic for $\Col(\omega, {<}\omega_1)$. 
 \end{remark}
 
\section{$\Sealing$ from the derived model representation}\label{sec:SealingTheorem}

In this section, we use the derived model representation in Theorem \ref{thm:dermodelrepomega} to prove $\Sealing$ similarly as in \cite[Section 6]{SaTr21}.

\begin{proof}[Proof of Theorem \ref{thm:SealingWoodin}]
Let $\kappa$ be a supercompact cardinal and let $g$ be $\Col(\omega, 2^{2^\kappa})$-generic over $V$. We aim to show $\Sealing$ in $V[g]$, so let $h$ be set generic over $V[g]$ and $h'$ be set generic over $V[g*h]$. 

We now work towards the set up of Theorem \ref{thm:ubcapturingprinciple}, and so we work in $V$. Let $\l$ be an inaccessible cardinal such that $g*h*h'$ is ${<}\l$-generic over $V$ and let $\chi>\l$ be an inaccessible cardinal. Let $W\in V_\chi$ be a $(\chi, \l^+, \l)$-saturated set of measures with the property that there are infinitely many Woodin cardinals in the interval $(\l, \comp(W))$, let $X\prec V_\chi$ be such that $\card{X}=\card{V_{\k+1}}$ and $V_{\k+1}\cup \{\l, W\}\subseteq X$ and let $j: V\rightarrow N$ be a $\chi$-supercompactness embedding. We now have that the UB-Capturing Principle holds in $N$ at $(j(\chi), j(\k), \l)$ as witnessed by $(j[W], j[X\cap V_\chi])$. In what follows we apply the lemmas and theorems of the previous sections to $(V, N[g*h])$ and $(V, N[g*h*h'])$ with $V$ in the role of $U$. In particular, we apply Theorem \ref{thm:dermodelrepomega} with $U$ as $V$, $\chi$ as $j(\chi)$, $\eta$ as $j(\k)$, $\l$ as $\l$, $X$ as $j[X\cap V_\chi]$, $W$ as $W$ and $W'$ as $j(W)$. 

 Then Clause (1) of $\Sealing$ (see Definition \ref{def:Sealing}) follows immediately from Theorem \ref{thm:dermodelrepomega} applied in $N[g][h]$ with $U=V$ and $\iota=\omega$, and well-documented results about derived models (see \cite{St08dm}). More precisely, letting $(M_\iota, \k_\iota)$ be as in Theorem \ref{thm:dermodelrepomega}, $L(\Gamma^\infty_{g*h}, \bR_{g*h}) \models \AD^+$ follows from the Derived Model Theorem, see \cite[Theorem 7.1]{St09}. The statement $\powerset(\bR_{g*h}) \cap L(\Gamma^\infty_{g*h}, \bR_{g*h}) = \Gamma^\infty_{g*h}$ follows from \cite[Theorem 9.3]{St09} as $\kappa_\iota = \omega_1^{V[g*h]}$ is by construction a limit of Woodin cardinals and ${<}\kappa_\iota$-strong cardinals in $M_\iota$. 
 
 So we are left with showing Clause (2) of $\Sealing$. For this, it is enough to prove the following.\\\\
 (a) Suppose $A\in \Gamma^\infty_{g*h}$, $\phi$ is a formula and $(c_1, ..., c_n)$ are constants for indiscernibles. Then
 \begin{center}
     $(\phi, A, c_1, ..., c_n)\in (\Gamma^\infty_{g*h})^\# \iff (\phi, A_{h'}, c_1,..., c_n)\in (\Gamma^\infty_{g*h*h'})^\#$.
 \end{center}

Let $k\subseteq \Col(\omega, \card{\Gamma^\infty_{g*h}})$ be $N[g*h]$-generic and $k'\subseteq \Col(\omega, \card{\Gamma^\infty_{g*h*h'}})$ be $N[g*h*h']$-generic such that $k(0)=A$ and $k'(0)=A_h$. Applying Theorem \ref{thm:dermodelrepomega} in $N[g*h]$, we obtain a ${\sf{DM}}$-sequence 
$$D=((\sfa_\a, \sfb_\a, E_\a) \mid \a<\omega)$$
for $N[g*h]$ supported by $(j[X\cap V_\chi], j[W], j(W), V, j(\chi), j(\k), \l, k)$, and applying Theorem \ref{thm:dermodelrepomega} in $M[g*h*h']$, we obtain a ${\sf{DM}}$-sequence 
$$D'=((\sfa_\a', \sfb'_\a, E'_\a) \mid \a<\omega)$$
for $M[g*h*h']$ supported by $(j[X\cap V_\chi], j[W], j(W), V, j(\chi), j(\k), \l, k')$.
We furthermore demand that $\sfa_1=\sfa_1'$, which is possible as we can simply perform the same one step construction to obtain $\sfa_1$ and $\sfa_1'$.

Let $${\sf{dm}}=(M_\a, \sigma_{\a, \b}, \sigma_\a, \k_\a:\a<\b<\omega)$$ be the underlying sequence of $D$ and $${\sf{dm}}'=(M'_\a, \sigma'_{\a, \b}, \sigma'_\a, \k'_\a:\a<\b<\omega)$$
be the underlying sequence of $D'$. Let $(M_\omega, \sigma_{\a, \omega}, \sigma_\omega)$ and $(N'_\omega, \sigma_{\a, \omega}', \sigma'_\omega)$ be the direct limits of ${\sf{dm}}$ and ${\sf{dm}}'$. Let $l\subseteq \Col(\omega, {<}\omega_1^{N[g*h]})$ and $l'\subseteq \Col(\omega, {<}\omega_1^{N[g*h*h']})$ be $M_\omega$ and $M'_\omega$ generics such that\\\\
(1.1) $l\in N[g*h*k]$ and $l'\in N[g*h*k']$,\\
(1.2) $(\Gamma^\infty_{g*h})^{\#}=((Hom^*, \bR_{g*h})^\#)^{M_\omega[l]}$, and\\
(1.3) $(\Gamma^\infty_{g*h*h'})^{\#}=((Hom^*, \bR_{g*h*h'})^\#)^{M'_\omega[l']}$.\\\\
Recall that $\kappa_1 = \kappa^{\sfa_1}$ and $\kappa_1' = \kappa^{\sfa_1'}$. Let $l_1=l\cap \Col(\omega, {<}\k_1)$ and $l_1'=l\cap \Col(\omega, {<}\k_1')$. We then have that\\\\
(2.1) $A_1=_{def}A\cap M_1[l_1]\in (Hom^*)^{M_1[l_1]}$,\\
(2.2) $A_1'=_{def}A'\cap M_1'[l'_1]\in (Hom^*)^{M'_1[l'_1]}$,\\
(2.3) $\sigma_{1, \omega}(A_1)=A$,\\
(2.4) $\sigma_{1, \omega}'(A_1')=A_{h'}$.\\\\
(2.1) and (2.2) hold because $\sfa_1=\sfa_1'$ b-catches $A$ and $A_{h'}$. In (2.3) and (2.4) we abuse the notation and think of $\sigma_{1, \omega}$ and $\sigma_{1, \omega}'$ as acting on $M_1[l_1]$ and $M_1'[l_1']$. (2.3) and (2.4) follow from the proof of Theorem \ref{thm:dermodelrepomega}. Notice that because we started with $\sfa_1=\sfa_1'$, we have that $M_1=M_1'$ and $l_1=l_1'$. We now have that 

      \begin{eqnarray*}
        (\phi, A, c_1, ..., c_n)\in (\Gamma^\infty_{g*h})^\#  
        &\iff & M_\omega[l] \models (\phi, A, c_1, ..., c_n)\in (Hom^*)^\# \\
        &\iff & M_1[l_1] \models (\phi, A_1, c_1, ..., c_n)\in (Hom^*)^\#\\
        &\iff & M'_1[l'_1] \models (\phi, A'_1, c_1, ..., c_n)\in (Hom^*)^\#\\
        &\iff & M_\omega'[l'] \models (\phi, A_{h'}, c_1, ..., c_n)\in (Hom^*)^\# \\
        &\iff & (\phi, A_{h'}, c_1, ..., c_n)\in (\Gamma^\infty_{g*h*h'})^\#  \\
    \end{eqnarray*}
    The above equivalences finish the proof of (a). 
    
\end{proof}

\section{$\Sealing$ for derived models} \label{sec:sealing for derived models}

In this section we outline how the following two corollaries, establishing $\Sealing$ for derived models, can be obtained from $\Sealing$. The reader may wish to review Definition \ref{def: smu} (in particular, $A^*$ is defined there).

\begin{corollary}\label{cor:newmain}
    Let $\kappa$ be a supercompact cardinal and suppose there is a proper class of Woodin cardinals. Let $g \subseteq \Col(\omega, 2^{2^\kappa})$ be $V$-generic and let $\eta$ be a limit of Woodin cardinals in $V[g]$. Let $H \subseteq \Col(\omega, {<}\eta)$ be $V[g]$-generic. Then there exists an elementary embedding \[ j \colon L(\Gamma^\infty_g, \bR_g) \rightarrow (L(\Hom^*, \bR^*))^{V[g*H]} \] such that for $A \in \Gamma^\infty_g$, $j(A) = A^*$. 
\end{corollary}

We can, in fact, obtain the following commuting diagram of embeddings.

\begin{corollary}\label{cor:diagram}
    Let $\kappa$ be a supercompact cardinal and suppose there is a proper class of Woodin cardinals. Let $g \subseteq \Col(\omega, 2^{2^\kappa})$ be $V$-generic and let $\eta$ be a limit of Woodin cardinals in $V[g]$. Let $H \subseteq \Col(\omega, {<}\eta)$ be $V[g]$-generic. Then there are elementary embeddings $j^0, j^1,$ and $k$ such that the following diagram commutes:
    \vspace{0.3cm}
    \begin{center}
      \begin{tikzpicture}
        \draw[->] (-0.35,-1.4) -- node[left] {$j^1$} (-0.35,-0.3) node[above]
        {$L(\Gamma^\infty_{g*H}, \bR_{g*H})$}; 

        \draw[->] (0.5,-1.75) node[left] {$L(\Gamma^\infty_g, \bR_g)$}-- node[below] {$j^0$} (2.3,-1.75) node[right] {$(L(\Hom^*, \bR^*))^{V[g*H]}$};
       
        \draw[->] (2.9,-1.45) -- node[above] {$k$} (-0.15,-0.25);   
      \end{tikzpicture}
    \end{center}
 Moreover, these embeddings move universally Baire sets canonically, i.e.,
 \vspace{0.5em}\begin{enumerate}\itemsep0.5em
     \item For $A \in \Gamma^\infty_g$, $j^0(A) = A^*$. 
     \item For $A \in \Gamma^\infty_g$, $j^1(A) = A_H$.
     \item For $A \in (\Hom^*)^{V[g*H]}$, say $A = B^*$ for some $B \in \Hom_{{<}\eta}^{V[g*H \upharpoonright \gamma]}$ for some $\gamma<\eta$, then $k(A)$ is the canonical interpretation of $B$ in $V[g*H]$.
 \end{enumerate}\vspace{0.5em}
\end{corollary}

We only prove Corollary \ref{cor:diagram} as Corollary \ref{cor:newmain} clearly follows from it.

\begin{proof}[Proof of Corollary \ref{cor:diagram}]
    In the setting of Corollary \ref{cor:diagram}, $\Sealing$ holds in $V[g]$ by Woodin's Sealing Theorem \ref{thm:SealingWoodin}, see also Section \ref{sec:SealingTheorem}. Let $\eta$ be a limit of Woodin cardinals in $V[g]$ and let $H \subseteq \Col(\omega, {<}\eta)$ be $V[g]$-generic. Fix an increasing sequence of ordinals $(\eta_i \mid 0 < i < \cf(\eta))$ cofinal in $\eta$ and let $H_i = H_{\eta_i} = H \cap \Col(\omega, {<}\eta_i)$, where $H_0$ is trivial. Then there are elementary embeddings \[ j_i \colon L(\Gamma^\infty_{g*H_i}, \bR_{g*H_i}) \rightarrow L(\Gamma^\infty_{g*H}, \bR_{g*H}) \] for $i < \cf(\eta)$ such that for $A \in \Gamma^\infty_{g*H_i}$, $j_i(A) = A_H$ and elementary embeddings \[ j_{i,k} \colon L(\Gamma^\infty_{g*H_i}, \bR_{g*H_i}) \rightarrow L(\Gamma^\infty_{g*H_k}, \bR_{g*H_k}) \] for $i<k<\cf(\eta)$ such that for $A \in \Gamma^\infty_{g*H_i}$, $j_{i,k}(A) = A_{g*H_k}$. Moreover, for $i<k<\cf(\eta)$, \[ j_i = j_k \circ j_{i,k}. \]
    Note that 
    \begin{claim}
        $(L(\Hom^*, \bR^*))^{V(\bR^*)}$ is equal to the direct limit of $L(\Gamma^\infty_{g*H_i}, \bR_{g*H_i})$ under the elementary embeddings $j_{i,k}$ for $i<k<\cf(\eta)$. 
    \end{claim}
    \begin{proof}
         Let \[ j_{i,\infty} \colon L(\Gamma^\infty_{g*H_i}, \bR_{g*H_i}) \rightarrow M \] be the direct limit embedding. Note that $M$ is well-founded as there is an elementary embedding \[ j_\infty \colon M \rightarrow L(\Gamma^\infty_{g*H}, \bR_{g*H}) \] such that $i<\cf(\eta)$, $j_i = j_\infty \circ j_{i,\infty}.$ Moreover, if $X=j_{\infty}(\Gamma^\infty_{g*H})$ and $Y=j_{\infty}^{-1}(\bR_{g*H})$ then $M=L(X, Y)$. 

         It is enough to show that $X=Hom^*$ and $Y=\bR^*$. Suppose first that $x\in Y$. We can then find $i<\cf(\eta)$ such that $x\in \rng(j_{i, \infty})$. Then it follows that $x\in \bR_{g*H_i}$, and hence $x\in \bR^*$. Similarly, if $x\in \bR^*$ then for some $i<\cf(\eta)$, $x\in \bR_{g*H_i}$ and so, $x\in Y$ as $j_{i, \infty}(x)=x$. 

         Suppose now that $A\in Hom^*$. We want to show that $A\in X$. Fix $i<\cf(\eta)$ and some $B \in \Gamma^\infty_{g*H_i}$ so that $A = B^*$. We have for every $k\in [i, \cf(\eta))$, $j_{i, k}(B)=B_{g*H_k}$. As $A=\bigcup_{k\in [i, \cf(\eta))}B_{g*H_k}$ and $j_{i, \infty}(B)=\bigcup_{k\in [i, \cf(\eta))}j_{i, k}(B)$, we get $A=j_{i, \infty}(B)$. Hence, $A\in X$ as $A=j_{\infty}^{-1}(j_i(B))$. 
         
         Suppose next that $A\in X$. We have $i<\cf(\eta)$ and $B\in \Gamma^\infty_{g*H_i}$ such that $A=j_{i, \infty}(B)$. Arguing as above, we get that $A=B^*$. 
    \end{proof}

    We are left with the following claim (which was implicitly proven above).

    \begin{claim}
        For $A \in (L(\Hom^*, \bR^*))^{V[g*H]} \cap (\powerset(\bR^*))^{V[g*H]}$ and $B \in \Gamma^\infty_{g*H_i}$ for some $i<\cf(\eta)$ with $A = B^*$, $j_{i,\infty}(B)=A$.
    \end{claim}
    \begin{proof}
        Let $x \in j_{i,\infty}(B)$. Then, for some $k<\cf(\eta)$, $k \geq i$, $x \in \bR_{g*H_k}$ and $x \in j_{i,k}(B)$. But $j_{i,k}(B) = B_{g*H_k} = A \cap \bR_{g*H_k}$, so $x \in A$. By an analogous argument, if $x \in A$, then $x \in j_{i,\infty}(B)$.
    \end{proof}

    Therefore, $j^1 = j_0$ obtained from $\Sealing$, $j^0 = j_{0,\infty}$, and $k = j_\infty$ are as desired.
\end{proof}

\section{Generically Correct Sealing}\label{sec:GCSealing}

We start with the definition of $\GCSealing$ and show that it implies $\Sealing$. Then, we will argue that our proof of Theorem \ref{thm:newmain} in fact shows that $\GCSealing$ holds in $V[g]$ (assuming the hypothesis of Theorem \ref{thm:newmain}). The reader may wish to consult Definition \ref{def:tamed} and Definition \ref{def:saturated measures}.

\begin{definition}\label{def:gcsealing}
    Assume there is a class of Woodin cardinals. Then we say $\GCSealing$ holds if for all inaccessible cardinals $\l<\chi$ and for every $(\chi, \l^+, \l)$-saturated and infinitely tamed $W\in V_\chi$  the following holds:
\vspace{0.5em}\begin{enumerate}\itemsep0.5em
        \item Whenever $g$ is ${<}\l$-generic and $\bar \mu \subseteq W$ is a homogeneity system in $V[g]$, $S_{\bar \mu_g} \in \Gamma^\infty_g$.\label{cl:1_def:gcsealing}
        \item For a club of countable substructures $X \prec V_{\chi}$ with $\{ \l, W\}\in X$, letting $\pi_X \colon N_X \rightarrow V_\chi$ be the Mostowski collapse of $X$, whenever $G\in V$ is ${<}\pi_X^{-1}(\l)$-generic over $N_X$, there is an elementary embedding \[ j \colon L((\Gamma^\infty_G)^{N_X[G]}, (\bR_G)^{N_X[G]}) \rightarrow L(\Gamma^\infty, \bR)\] such that for every homogeneity system $\bar \mu \subseteq  \pi_X^{-1}(W)$ with $\bar \mu\in N_X[G]$, \[ j(S_{\bar \mu}^{N_X[G]}) = S_{\pi_X\shortpwimg\bar\mu}. \]\label{cl:2_def:gcsealing}
    \end{enumerate}
\end{definition}

$\GCSealing$ is a consequence of Woodin's Strong Tower Sealing. Woodin obtains Strong Tower Sealing at supercompact cardinals above an extendible cardinal $\delta$ in generic extensions collapsing $2^\delta$ to be countable using stationary tower forcing, see \cite{WoodinLongExtender}.

The following theorem is a consequence of our proof of Theorem \ref{thm:newmain}. 

\begin{theorem}\label{thm:GCSealing}
    Let $\kappa$ be a supercompact cardinal and let $g$ be $\Col(\omega, 2^{2^\kappa})$-generic over $V$. Suppose that there is a proper class of Woodin cardinals. Then $\GCSealing$ holds in $V[g]$.
\end{theorem}
\begin{proof}
Let $\l<\chi$ be inaccessible cardinals in $V[g]$ and $W_0\in V_\chi[g]$ be a $(\chi, \l^+, \l)$-saturated and infinitely tamed set of measures. We can then find $W\in V$ such that $W_0\subseteq W$ and in $V$ (and hence in $V[g]$), $W$ is a $(\chi, \l^+, \l)$-saturated and infinitely tamed set of measures. It is enough to verify Clauses \eqref{cl:1_def:gcsealing} and \eqref{cl:2_def:gcsealing} in the definition of $\GCSealing$ for $(\l, \chi, W)$.

Clause \eqref{cl:1_def:gcsealing} follows from Lemma \ref{lem:homsysteminWg} and is also proven in \cite{La04} (e.g. see \cite[Lemma 3.4.11]{La04}). We verify Clause \eqref{cl:2_def:gcsealing}. Let $X \prec V_\chi$ be a countable substructure. Let $\pi_X \colon N_X \rightarrow V_\chi[g]$ be the Mostowski collapse of $X$. In what follows we will use subscript $X$ to denote the $\pi_X$-preimages of various objects, and we will use this convention with other substructures as well. Let $G$ be ${<}\l_X$-generic over $N_X$. Then the elementary embedding \[ j \colon (L(\Gamma^\infty_G, \bR_G))^{N_X[G]} \rightarrow (L(\Gamma^\infty, \bR))^{V[g]} \] in the statement of $\GCSealing$ is obtained similarly as in the proof of Theorem \ref{thm:SealingWoodin} in Section \ref{sec:SealingTheorem}. Notice that because $$((\Gamma^\infty_G, \bR_G)^\#)^{N_X[G]}=((\Gamma^\infty)^{N_X[G]}, \bR^{N_X[G]})^\#,$$ to prove the existence of an embedding as in Clause \eqref{cl:2_def:gcsealing} of Definition \ref{def:gcsealing}, it is enough to verify the following statement in $V[g]$:\\\\
(a) For a club of $X\prec V_\chi[g]$, if $G$ is ${<}\l_X$-generic, $\bar \mu \subseteq W_X$ is a homogeneity system with $\bar \mu\in N_X[G]$, $\phi$ is a formula and  $c_0,..., c_n$ are constants for indiscernibles, then
\[ (\phi, S_{\bar \mu}^{N_X[G]}, c_0,..., c_n)\in ((\Gamma^\infty_G, \bR_G)^\#)^{N_X[G]}  \iff  (\phi, S_{\pi_X\pwimg \bar \mu}, c_0,..., c_n)\in (\Gamma^\infty, \bR)^\#. \]\vskip 0.3cm    
 Let $\theta'>\theta>\chi$ be inaccessible cardinals and $Y\prec V_{\chi+1}$ be such that $Y\in V$, $\{\chi, W\}\in Y$, $V_{\kappa+1}\subseteq Y$, and $\card{Y}^V=\card{V_{\kappa+1}}^V$. Fix some $j: V\rightarrow P'$ that witnesses the $\card{V_{\chi+1}}$-supercompactness of $\k$ in $V$. To show (a) it is enough to verify (b) in $V[g]$, where (b) is the following statement:\\\\
(b) For all $X\prec V_{\theta'}[g]$ such that $\{Y, \chi, \l, \theta, j\rest V_\theta\}\in X$, if $G$ is ${<}\l_X$-generic over $N_X[G]$, $\bar \mu \subseteq W_X$ is a homogeneity system with $\bar \mu\in N_X[G]$, $\phi$ is a formula and  $c_0,..., c_n$ are constants for indiscernibles, then
\[ (\phi, S_{\bar \mu}^{N_X[G]}, c_0,..., c_n)\in ((\Gamma^\infty_G, \bR_G)^\#)^{N_X[G]}  \iff  (\phi, S_{\pi_X\pwimg \bar \mu}, c_0,..., c_n)\in (\Gamma^\infty, \bR)^\# \] \vskip 0.3cm   

We first show that (a) follows from (b). Assume (b). Let $T\subseteq \omega\times [V_\chi[g]]^{<\omega}$ be the theory of $V_{\theta'}$ with parameters from $\{Y, \chi, \l, \theta, j\rest V_\theta\}$. More precisely, $T$ consists of pairs $(\phi, \vec{a})$ where $\phi$ is a formula and $\vec{a}$ is a finite sequence of members of $V_\chi[g]$ such that  $$V_{\theta'}[g]\models \phi[\vec{a}, Y, \chi, \l, \theta, j\rest V_\theta].$$
Suppose now that $X'\prec (V_\chi[g], T, \in)$. We claim that (a) holds for any such $X'$. Let $X$ be the Skolem hull of $X'\cup\{Y, \chi, \l, \theta, j\rest V_\theta\}$ computed in $V_{\theta'}[g]$. We claim that $X\cap V_\chi[g]=X'$. Suppose that $\vec{a}\in [X']^{<\omega}$ and $b\in V_\chi[g]$ is such that for some formula $\phi$, $b$ is the unique set such that 
$$V_{\theta'}[g]\models \phi[b, \vec{a}, Y, \chi, \l, \theta, j\rest V_\theta].$$
It then follows that $b$ is the unique member of $V_{\chi}[g]$ such that $(\phi, b^\frown \vec{a})\in T$. Since $\vec{a}\in X'$, we have that $b\in X'$. We can now apply (b) to $X$ and get (a) for $X'$.

 Fix then an $X$ as in (b), and also fix $G\in V[g]$ that is ${<}\l_X$-generic over $N_X$, a homogeneity system $\bar \mu \subseteq W_X$ with $\bar \mu \in N_X[G]$, a formula $\phi$ and $c_0, ..., c_n$ constants for indiscernibles. We set $P=j(V_\theta)$. Following the proof of Theorem \ref{thm:SealingWoodin}, we can find, in $N_X[G]$, a $(P_X, j_X(\chi), \l_X, \{j_X(\k), j_X(W_X)\})$-realizable rock\footnote{See Definition \ref{def:relrock}. We are abusing terminology here as $P_X$ is not a class, but clearly this is irrelevant.} $\sfa$ and $k\subseteq \Col(\omega, {<}\k^{\sfa})$ such that for some $\bar \nu\in M^\sfa[k]$, \\\\
(1.1) in $N_X[G]$, $\sfa$ b-catches $(S_{\bar \mu})^{N_X[G]}$ relative to $(P_X, j_X(\chi), \l_X, \{j_X(\k), j_X(W_X)\})$ (see Definition \ref{def:catchingpair}),\\
(1.2) $k$ is $M^\sfa$-generic,\\
(1.3) in $N_X[G]$, $S_{\sigma^{\sfa}\pwimg \bar \nu}=S_{\bar \mu}$,\\
(1.4) $\sigma^\sfa\pwimg \bar \nu\subseteq j_X(W)$, and\\
(1.5) $(\phi, (S_{\bar \nu})^{M^\sfa[k]}, c_0,..., c_n)\in ((Hom^*)^\#)^{M^\sfa[k]}$ if and only if $(\phi, S_{\bar \mu}, c_0, ..., c_n)\in (\Gamma^\infty_G)^\#$.\\\\
Such an $\sfa$ exists in any ${\sf{DM}}$-sequence that is supported by $$(j_X[Y_X], j_X[W_X], j_X(W_X), M_X, j_X(\chi), j_X(\k), \l_X, l)$$ where $l\subseteq \Col(\omega, \Gamma^\infty_G)$ is $N_X$-generic collapsing $\Gamma^\infty_G$ to be countable. Notice now that letting $\sfb=(\pi_X\circ \sigma^\sfa, M^\sfa, \k^\sfa)$, we have that\\\\
(2.1) $\sfb$ b-catches $S_{\pi_X\circ \sigma^\sfa\pwimg \bar \nu}$ relative to $(P, j(\chi), \l, \{j(\k), j(W_X)\})$,\\
(2.2) $k$ is $M^\sfa$-generic,\\
(2.3) $\pi_X\circ \sigma^\sfa\pwimg \bar \nu\subseteq j(W)$, and\\
(2.4) $(\phi, (S_{\bar \nu})^{M^\sfb[k]}, c_0,..., c_n)\in ((Hom^*)^\#)^{M^\sfb[k]}$ if and only if $(\phi, S_{\pi_X\circ \sigma^{\sfa}\pwimg\bar \nu}, c_0, ..., c_n)\in (\Gamma^\infty)^\#$.\\\\
(2.1) is a consequence of Lemma \ref{lem:keylemmatamed}. To finish it is then enough to show that the following statement holds:\\\\
(c) $S_{\pi_X\pwimg\bar \mu}= S_{\pi_X\circ \sigma^{\sfa}\pwimg\bar \nu}$.\\\\
We have that in $N_X[G]$\\\\
(3) $S_{j_X\pwimg \bar \mu}=S_{\bar \mu}$ (see Lemma \ref{lem:homsysteminWg}).\\\\
Clearly (3) then implies that in $N_X[G]$,\\\\
(4) $S_{j_X\pwimg \bar \mu}=S_{\sigma^\sfa \pwimg \bar \nu}$.\\\\
We also have that in $V[g]$,\\\\
(5) $S_{\pi_X\pwimg \bar \mu}=S_{j\circ \pi_X \pwimg \bar \mu}=S_{\pi_X\circ j_X \pwimg \bar \mu}$. \\\\
It is then enough to show that \\\\
(d) $S_{\pi_X\circ j_X \pwimg\bar \mu}= S_{\pi_X\circ \sigma^{\sfa}\pwimg\bar \nu}$.\\\\
(d), however, easily follows from Lemma \ref{lem:equivalence}.
\end{proof}

Finally, we observe that $\Sealing$ is a consequence of $\GCSealing$.

\begin{lemma}\label{lem:gcsealingsealing}
    Assume there is a proper class of inaccessible limits of Woodin cardinals. Then $\GCSealing$ implies $\Sealing$.
\end{lemma}
\begin{proof}
    Suppose there is a proper class of Woodin cardinals and $\GCSealing$ holds in $V$. Let $\bP_1 \in V$ and $\bP_2 \in V^{\bP_1}$ be partial orders and let $h_1$ be $\bP_1$-generic over $V$ and $h_2$ be $(\bP_2)_{h_1}$-generic over $V[h_1]$. Let $\l$ be an inaccessible cardinal such that $\mathbb{P}_1*\mathbb{P}_2\in V_\l$. Let $\chi$ and $W$ be such that $\chi$ is an inaccessible cardinal, $W\in V_\chi$ and $W$ is an infinitely tamed $(\chi, \l^+, \l)$-saturated set of measures. 
    
    We want to show that Clause~\eqref{eq:Sealing2} in Definition \ref{def:Sealing} holds, i.e., that there is an elementary embedding \[ j \colon L(\Gamma^\infty_{h_1}, \bR_{h_1}) \rightarrow L(\Gamma^\infty_{h_1 * h_2}, \bR_{h_1 * h_2}) \] such that for every $A \in \Gamma^\infty_{h_1}$, $j(A) = A_{h_2}$. The derived model theorem (see Theorem \ref{woodin: der model thm}) implies that assuming $\Sealing$ and a class of inaccessible limits of Woodin cardinals, Clause \eqref{eq:Sealing2} implies Clause \eqref{eq:Sealing1} in Definition \ref{def:Sealing}.

    Fix an increasing sequence $(\xi_l \mid l < \omega)$ of $V$-cardinals that are in the interval $(\l, \chi)$. Recall that these are uniform indiscernibles for $L(\Gamma^\infty_{h_1}, \bR_{h_1})$ and $L(\Gamma^\infty_{h_1*h_2}, \bR_{h_1*h_2})$.
    It suffices to show that for any natural number $m$, real $x \in \bR_{h_1}$, and set $A \in \Gamma^\infty_{h_1}$, 
     if $\varphi$ is an arbitrary formula, then \[ L(\Gamma^\infty_{h_1}, \bR_{h_1}) \models \varphi[A, x, u_m] \Longrightarrow L(\Gamma^\infty_{h_1*h_2}, \bR_{h_1*h_2}) \models \varphi[A_{h_2}, x, u_m], \] where $u_m = (\xi_l \mid l < m)$.

     Fix some $A \in \Gamma^\infty_{h_1}$ and let $\dot{\bar\mu}$ be a $\bP_1$-name such that $(\dot{\bar\mu})_{h_1}$ is a homogeneity system for $A$ in $V[h_1]$. That means, we aim to show that \[ \Vdash_{\bP_1} L(\Gamma^\infty_{h_1}, \bR_{h_1}) \models \varphi[S_{\dot{\bar\mu}}, \dot x, u_m] \, \Longrightarrow \,\, \Vdash_{\bP_1*\bP_2} L(\Gamma^\infty_{h_1*h_2}, \bR_{h_1*h_2}) \models \varphi[S_{\dot{\bar\mu}}, \dot x, u_m], \] where $\dot x$ is a $\bP_1$-name such that $(\dot x)_{h_1} = x$ and we confuse $h_1$ and $h_2$ with their standard names as well as $\dot x$ with the standard $\bP_1*\bP_2$-name such that $(\dot x)_{h_1*h_2} = x$.

     Let $X \prec V_\chi$ be a countable substructure such that $\{ \bP_1 * \bP_2, \dot{\bar\mu}, \dot x, \l, W\} \cup \{ \xi_l \mid l < \omega \} \subset X$. Let \[ \pi_X \colon N_X \rightarrow V_\chi \] be the Mostowski collapse of $X$ and let $(\bar\bP_1 * \bar\bP_2, \dot{\bar\nu}, \dot{\bar x}), \bar \l, \bar W, (\bar\xi_l \mid l < \omega)$ be the respective collapses such that $\pi_X(\bar\bP_1 * \bar\bP_2, \dot{\bar\nu}, \dot{\bar x}, \bar \l, \bar W) = (\bP_1 * \bP_2, \dot{\bar\mu}, \dot x, \l, W)$ and $\pi_X(\bar\xi_l) = \xi_l$ for each $l < \omega$. Let $\bar h_1$ be $\bar\bP_1$-generic over $N_X$, let $\bar\nu = (\dot{\bar\nu})_{h_1}$ and let $\bar x = (\dot{\bar x})_{h_1}$. Then, by $\GCSealing$, there is an elementary embedding \[ j_1 \colon (L(\Gamma^\infty_{\bar h_1}, \bR_{\bar h_1}))^{N_X[\bar h_1]} \rightarrow L(\Gamma^\infty, \bR) \] such that \[ j_1 \colon ((\Gamma^\infty_{\bar h_1}, \bR_{\bar h_1})^\#)^{N_X[\bar h_1]} \rightarrow (\Gamma^\infty, \bR)^\# \] and \[ j_1((S_{\bar\nu})^{N_X[\bar h_1]}) = S_{\pi_X\shortpwimg\bar\nu}. \]
     Let $\bar h_2 \in V$ be $\bP_2$-generic over $N_X[\bar h_1]$. Then, again by $\GCSealing$, there is an elementary embedding \[ j_2 \colon (L(\Gamma^\infty_{\bar h_1 * \bar h_2}, \bR_{\bar h_1 * \bar h_2}))^{N_X[\bar h_1 * \bar h_2]} \rightarrow L(\Gamma^\infty, \bR) \] such that \[ j_2 \colon ((\Gamma^\infty_{\bar h_1 * \bar h_2}, \bR_{\bar h_1 * \bar h_2})^\#)^{N_X[\bar h_1 * \bar h_2]} \rightarrow (\Gamma^\infty, \bR)^\# \] and \[ j_2((S_{\bar\nu})^{N_X[\bar h_1 * \bar h_2]}) = S_{\pi_X\shortpwimg\bar\nu}. \]

     Note that we are left with showing that 
     \begin{eqnarray*}
     &(L(\Gamma^\infty_{\bar h_1}, \bR_{\bar h_1}))^{N_X[\bar h_1]} \models \varphi[(S_{\bar\nu})^{N_X[\bar h_1]}, \bar x, \bar u_m] \Longrightarrow \\ &(L(\Gamma^\infty_{\bar h_1 * \bar h_2}, \bR_{\bar h_1 * \bar h_2}))^{N_X[\bar h_1 * \bar h_2]} \models \varphi[(S_{\bar\nu})^{N_X[\bar h_1 * \bar h_2]}, \bar x, \bar u_m],
     \end{eqnarray*}
     where $\bar u_m$ is any sequence of uniform indiscernibles for the models in the equation above. 

     But now using $j_1((S_{\bar\nu})^{N_X[\bar h_1]}) = S_{\pi_X\shortpwimg\bar\nu}$ and $j_2((S_{\bar\nu})^{N_X[\bar h_1 * \bar h_2]}) = S_{\pi_X\shortpwimg\bar\nu}$, we have
     \begin{eqnarray*}
         && (L(\Gamma^\infty_{\bar h_1}, \bR_{\bar h_1}))^{N_X[\bar h_1]} \models \varphi[(S_{\bar\nu})^{N_X[\bar h_1]}, \bar x, \bar u_m]\\
         && \Longrightarrow  L(\Gamma^\infty, \bR) \models \varphi[S_{\pi_X\shortpwimg\bar\nu}, x, u_m] \\
         && \Longrightarrow  (L(\Gamma^\infty_{\bar h_1 * \bar h_2}, \bR_{\bar h_1 * \bar h_2}))^{N_X[\bar h_1 * \bar h_2]} \models \varphi[(S_{\bar\nu})^{N_X[\bar h_1 * \bar h_2]}, \bar x, \bar u_m],
     \end{eqnarray*}
     as desired, where $\bar u_m$ and $u_m$ are arbitrary sequences of $m$ uniform indiscernibles for the relevant models.
\end{proof}

\section{A proof of  Theorem \ref{sealing for the ub powerset}}\label{sec: sealing for ub powerset}

In this section we prove Theorem \ref{sealing for the ub powerset}. We start by some preliminary lemmas that we will use in the proof of Theorem \ref{sealing for the ub powerset}.

\subsection{Resurrection of derived models}

Our goal here is to show that the derived models can be resurrected in the following sense. We disregard the notation used in the previous section. Recall that if $h\subseteq \Col(\omega, {<}\l)$ is generic and $\xi<\l$ then $h_\xi=h\cap \Col(\omega, {<}\xi)$.

\begin{lemma}\label{lem: resurrection} Suppose $\k$ is a supercompact cardinal and there is a proper class of inaccessible limits of Woodin cardinals. Suppose $\l$ is an inaccessible limit of Woodin cardinals  above $\kappa$, $\d>\l$ is a Woodin cardinal and $\chi>\d$ is an inaccessible cardinal. Let $j: V\rightarrow M$ be a $\chi$-supercompactness embedding, i.e., $\cp(j)=\k$, $j(\k)>\chi$ and $M^{\chi}\subseteq M$. Suppose $g\subseteq \Col(\omega, {<}\l)$ is $V$-generic, $\mathbb{P}\in V_{\chi}[g]$ is such that $V[g]\models ``\mathbb{P}$ is  countably closed", $k\subseteq \mathbb{P}$ is $V[g]$-generic and $(\pi, N, \sigma, h)\in V[g*k]$ are such that
\begin{enumerate}\itemsep0.5em
\item $N$ is a transitive model of $\sf{ZFC}$,
\item $\pi: V_{\chi}\rightarrow N$ and $N\subseteq V_\chi[g*k]$,
\item $\pi(\l)=\l$,
\item every $x\in \bR_g$ is in some ${<}\l$-generic extension of $N$,
\item $\sigma: N\rightarrow j(V_{\chi})$, $j\restriction V_{\chi}=\sigma\circ \pi$, and
\item $h\subseteq \Col(\omega, {<}\l)$ is $N$-generic such that  $h\in V[g*k]$ and $\bR^{N[h]}=\bR_g$.
  \end{enumerate}
   Then $\Gamma^\infty_g=(\Gamma^{\infty})^{N[h]}$.
\end{lemma}
\begin{proof} Let $p=(f, W_1, W_0)\in V_{\chi}$ be a flipping system (see Definition \ref{flipping system}) that flips through $\d$ such that $\comp(p)>\l^+$ and $W_1$ is $(\chi, \l^+, \l)$-saturated (see Definition \ref{def:saturated measures}). We have that if $r$ is ${<}\l$-generic and $\bar \mu\subseteq W_1$ is a homogeneity system in $V[r]$ then $S_{\bar \mu_r}\in \Gamma^\infty_r$. This follows from Lemma \ref{lem:homsysteminWg} and is also proven in \cite{La04} (e.g. see \cite[Lemma 3.4.11]{La04}). Notice that because $M[g*k]$ is $\chi$-closed in $V[g*k]$, we have that $(\pi, N, \sigma, h)\in M[g*k]$.

We first show that   $(\Gamma^{\infty})^{N[h]}\subseteq \Gamma^\infty_g$. To see this, fix a set $D\in (\Gamma^\infty)^{N[h]}$. There is then some $\xi<\l$ such that for some homogeneity system $\bar{\mu}\subseteq  \pi(W_1)$, $\bar{\mu}\in N[h_\xi]$ and $D=S_{\bar{\mu}}^{N[h]}$. Using the flipping function $f$ and our choice of $j$, we now get that $D=S_{\sigma\pwimg\bar{\mu}}^{V[g]}$, and hence $D\in \Gamma^\infty_g$ (see Lemma \ref{lem:keylemmatamed}).

Conversely, let $D\in \Gamma^\infty_{g}$. Again, let $\xi<\l$ be such that there is a homogeneity system $\bar{\mu}\subseteq W_1$ with the property that $\bar{\mu}\in V[g_\xi]$ and $D=S_{\bar{\mu}}^{V[g]}$. It is enough to prove that\\\\
(a) $D\in N[h]$ and for some homogeneity system $\bar \nu\subseteq \pi(W_1)$ with $\bar \nu\in N[h]$, $N[h]\models D=S_{\bar \nu}$.\\\\
Let $\zeta\in (\xi, \l)$ be a Woodin cardinal and let $W^*_1\subseteq W_1$ be $(\chi, \zeta, \xi)$-saturated and such that $\bar \mu \subseteq W^*_1$. Let $q=(f^*, W^*_1, W^*_0)$ be  a flipping system that flips through $\zeta$ with $\comp(q)>\xi$. 

Let now $\xi^*<\l$ be large enough such that $\pi\rest V_\zeta\in N[h_{\xi^*}]$. Because $f^*$ is a bijection, we have that $(\pi(f^*))^{-1}(\pi\pwimg f^*(\bar \mu))\in N[h_{\xi^*}]$. Set then $\bar \nu=(\pi(f^*))^{-1}(\pi\pwimg f^*(\bar \mu))$. We thus have that\\\\
(1) $\bar \nu\in N[h]$ and $(S_{\bar \nu_h})^{N[h]}=S_{\bar \mu_g}$ (because $\sigma\pwimg \bar \nu= j\pwimg \bar \mu$).\\\\
(1) easily implies (a).
\end{proof}

Next we introduce genericity iterations, which are the ways models such as $N$ in Lemma \ref{lem: resurrection} are produced. We will use Neeman's genericity iterations (see Theorem \ref{thm:NeemanGenIt}) as in previous sections of the paper. 

\begin{definition}\label{def: genericity iteration1} Suppose $\k$ is a supercompact cardinal, $\l>\k$ is an inaccessible limit of Woodin cardinals and $\d>\l$ is an inaccessible cardinal. Let $g\subseteq \Col(\omega, {<}\l)$ be $V$-generic and let $j: V\rightarrow M$ be such that $\cp(j)=\k$, $j(\k)>\d$ and $M^\d\subseteq M$. Suppose that $(\pi, N, \sigma)$ is such that
\begin{enumerate}
\item $N$ is a transitive model of $\sf{ZFC}$,
\item $\pi: V_{\d}\rightarrow N$,
\item $\pi(\l)=\l$,
\item $\sigma: N\rightarrow j(V_{\d})$ and $j\restriction V_{\d}=\sigma\circ \pi$.
  \end{enumerate}
Fix an $\eta\in (\kappa, \l)$, let $(a_\xi \mid \xi<\l=\omega_1^{V[g]})\in V[g]$ be an enumeration of $\bR_{g}$ and let $(\eta_\xi \mid \xi< \l)$ 
            be a strictly increasing sequence of Woodin cardinals of $N$ with supremum $\l$ such that $\eta_0>\eta$. We then obtain a sequence $(\X_\xi, \M_\xi, \sigma_\xi \mid \xi<\l)\in V[g]$ such that

            \vspace{0.5em}
            \begin{enumerate}\itemsep0.5em
                \item[(GI.1)] $\M_0=N$, $\sigma_0=\sigma$ and $\X_0$ is a nice iteration tree of length $\omega+1$ on $\M_0$ (see Definition \ref{:def:niceext}) that is below $\eta_0$, above $\eta$ and for some $\M^{\X_0}$-generic 
             $h\subseteq \Col(\omega, \pi^{\X_0}(\eta_0))$ with $h\in V[g]$, $a_0\in \M^{\X_0}[h]$,
                \item[(GI.2)] for each $\xi<\l$, $\M_{\xi+1}$ is the last model of $\X_\xi$,
                \item[(GI.3)] for each limit ordinal $\xi<\l$, $\M_\xi$ is the direct limit of $(\M_\b, \pi_{\b, \b'} \mid \b<\b'<\xi)$ where $\pi_{\b, \b'}:\M_\b\rightarrow \M_{\b'}$ is the composition of the iteration embeddings given by the iteration trees $(\X_\b \mid \b<\xi)$,\footnote{We have that $\pi_{\b, \b'}(\l)=\l$.} 
                \item[(GI.4)]for each $\xi<\l$, $\sigma_\xi: \M_\xi\rightarrow j(V_{\d})$ is an elementary embedding such that $j\restriction V_{\d}=\sigma_\xi\circ \pi_{0, \xi}$ and for all $\xi<\xi'<\l$, $\sigma_\xi=\sigma_{\xi'}\circ \pi_{\xi, \xi'}$,
                \item[(GI.5)] for each $\xi<\l$, $\X_\xi$ is a nice iteration tree of length $\omega+1$ on $\M_\xi$ that is below $\pi_{0, \xi}(\eta_\xi)$, above $\sup_{\b<\xi}\pi_{0, \xi}(\eta_\b)$ and for some $\M^{\X_\xi}$-generic $h\subseteq \Col(\omega, \pi_{0, \xi+1}(\eta_\xi))$ with $h\in V[g]$, $a_\xi\in \M^{\X_\xi}[h]$.
            \end{enumerate}
            \vspace{0.5em}
            
             Let now $N'$ be the direct limit of $(\M_\b, \pi_{\b, \b'} \mid \b<\b'<\l)$, $\pi_{\b, \l}: \M_\b\rightarrow N'$ be the direct limit embedding, $\pi'=\pi_{0, \l}$ and $\sigma': N'\rightarrow j(V_{\d})$ be the canonical realizability map, i.e., $\sigma'(x)=\sigma_\xi(\bar{x})$ where $\bar{x}$ and $\xi<\l$ are such that $\pi_{\xi, \l}(\bar{x})=x$. We then say that $(\X_\xi, \M_\xi, \sigma_\xi \mid \xi<\l)$ is an $\bR_g$-genericity iteration of $(\pi, N, \sigma)$, and $(\pi', N', \sigma')$ is an $\bR_g$-genericity iterate of $(\pi, N, \sigma)$. 
             \end{definition}
Our next lemma shows that the uB-powerset can be correctly computed over the iterates of $V$. 

\begin{lemma}\label{lem: resurrection1} Suppose $\k$ is a supercompact cardinal and there is a proper class of inaccessible limits of Woodin cardinals. Suppose $\l$ is an inaccessible limit of Woodin cardinals above $\kappa$ and suppose $\d_0$ is a Woodin cardinal and $\d_0<\d_1$ is an inaccessible cardinal such that $\d_0>\l$. Let $j: V\rightarrow M$ be a $\d_1$-supercompactness embedding, i.e., $\cp(j)=\k$, $j(\k)>\d_1$ and $M^{\d_1}\subseteq M$. Suppose $g\subseteq \Col(\omega, {<}\l)$ is $V$-generic, $\l_0<\l$ is such that $\l_0=\omega_1^{V[g_{\l_0}]}$ and $(\pi, N, \sigma)\in V[g]$ is such that
\begin{enumerate}
\item $N$ is a transitive model of $\sf{ZFC}$,
\item $\pi: V_{\d_1}\rightarrow N$ and $N\subseteq V_{\d_1}[g]$,
\item $\pi(\l)=\l$,
\item every $x\in \bR_{g_{\l_0}}$ is in some ${<}\l_0$-generic extension of $N$,
\item for every $\a<\l_0$, $N|\a\in V[g_{\l_0}]$ and $N|\a$ is countable in $V[g_{\l_0}]$,
\item $\sigma: N\rightarrow j(V_{\d_1})$ and $j\restriction V_{\d_1}=\sigma\circ \pi$.
  \end{enumerate}
    Let $h\subseteq \Col(\omega, {<}\l_0)$ be $N$-generic such that 
  \begin{enumerate}
  \item $\bR^{N[h]}=\bR_{g_{\l_0}}$,
  \item $h\in V[g]$,
  \item $h$ appears in a forcing extension of $V[g_{\l_0}]$ by some poset $\mathbb{P}$ such that $\card{\mathbb{P}}^{V[g_{\l_0}]}=\lambda_0$\footnote{See \cite[Lemma 3.1.5]{La04}. As was commented before, the generic $h$ is added by the poset $\mathbb{P}$ appearing in the proof of \cite[Lemma 3.1.5]{La04}.}, and
  \item for all $\xi<\l_0$, $h_\xi\in V[g_{\l_0}]$.
  \end{enumerate}
 Suppose $\Gamma^\infty_{g_{\l_0}}=(\Gamma^\infty)^{N[h]}$ and $X\in V_{\l}[g_{\l_0}]\cap V_{\l}^{N[h]}$. Then \[ \powerset_{uB}(X)^{V[g_{\l_0}]}=\powerset_{uB}(X)^{N[h]}. \]
\end{lemma}
\begin{proof} 
Let $\iota=\max(\iota_X^{N[h]}, \iota_X)$. Because $\l$ is an inaccessible cardinal, we have that $\iota<\l$. We intend to compute $\powerset_{uB}(X)^{V[g_{\l_0}]}$ in $V[g_{\l_0}]$ after forcing with $\Col(\omega, \iota)$, and since $\powerset_{uB}(X)^{V[g_{\l_0}]}$ is independent of the generic we pick for $\Col(\omega, \iota)$, we may just as well assume that our generic is in $V[g]$. Let then $k\in \powerset(\Col(\omega, \iota))\cap V[g]$ be both $V[g_{\l_0}]$-generic and $N[h]$-generic. 

Let $r\subseteq \Col(\omega_1^{V[g]}, \bR_g)$ be $V[g]$-generic. Notice that because $M[g]$ is $\d_1$-closed in $V[g]$, we have that $(\pi, N, \sigma)\in M[g]$. We now perform an $\bR_g$-genericity iteration of $N$ (in $M[g*r]$) and obtain $(\pi', N', \sigma')\in M[g*r]$ such that 

\vspace{0.5em}             \begin{enumerate}\itemsep0.5em
    \item[(1.1)] $N'$ is a transitive model of ${\sf{ZFC}}$,
    \item[(1.2)] $\pi': N\rightarrow N'$ is an elementary embedding and $N'\subseteq V_{\d_1}[g*r]$,
    \item[(1.3)] $\pi'(\l)=\l$ and $\cp(\pi')>\xi$ where $\xi$ is such that $(\powerset_{uB}(X))^{N[h]}\in V_{\xi}^{N[h]}$,
    \item[(1.4)] every $x\in \bR_g$ is in some ${<}\l$-generic extension of $N'[h]$,
    \item[(1.5)] $\sigma': N'\rightarrow j(V_{\d_1})$ is such that $\sigma=\sigma'\circ \pi'$.
\end{enumerate}\vspace{0.5em}

Let now $h'\subseteq \Col(\omega, {<}\l)$ be $N'[h]$-generic with $k \in N'[h*h']$ and $\bR^{N'[h*h']}=\bR_g$. We can find such an $h'$ in a forcing extension of $V[g*r]$ by some poset $\mathbb{P}$ such that $\card{\mathbb{P}}^{V[g]}=\lambda$ and $\mathbb{P}$ is countably closed in $V[g*r]$ (see Remark \ref{length omega1 case}). It then follows from Lemma \ref{lem: resurrection} that $(\Gamma^\infty)^{N'[h*h']}=\Gamma^\infty_g$. Let $j_{h, h'}^{N'[h*h']} \colon L(\Gamma_h^\infty, \bR_h)^{N'[h]} \rightarrow L(\Gamma^\infty_{h*h'}, \bR_{h*h'})^{N'[h*h']}$ be the canonical embedding\footnote{There are many such embeddings. By ``the canonical" embedding, we mean the embedding that moves indiscernibles to indiscernibles in the order preserving way.} in $N'[h*h']$ with the property that $j_{h,h'}(A) = A_{h'}$ for $A \in \Gamma_h^\infty$. Moreover, let $j_{g_{\l_0}, g} \colon L(\Gamma_{g_{\l_0}}^\infty, \bR_{g_{\l_0}}) \rightarrow L(\Gamma^\infty_{g_{\l_0}*g}, \bR_{g_{\l_0}*g})$ be the canonical embedding with the property that $j_{g_{\l_0},g}(A) = A_{g}$ for $A \in \Gamma_{g_{\l_0}}^\infty$. Notice that $$j_{h, h'}^{N'[h*h']}=j_{g_{\l_0}, g}\rest L(\Gamma_h^\infty, \bR_h)^{N'[h]}.$$ 
The above equality is a consequence of the fact that $A_{h'}=A_g$, and this follows from calculations presented in Lemma \ref{lem:keylemmatamed}. It then follows from Lemma \ref{easy consequence} that the parameters used to define members of $\powerset_{uB}(X)^{V[g_{\l_0}]}$ and $\powerset_{uB}(X)^{N'[h]}$ are the same, and hence $\powerset_{uB}(X)^{V[g_{\l_0}]}=\powerset_{uB}(X)^{N'[h]}$. Since $\cp(\pi')>\xi$, we have that $\powerset_{uB}(X)^{V[g_{\l_0}]}=\powerset_{uB}(X)^{N[h]}$.
\end{proof}

\subsection{A Derived Model Theorem.} Here we prove a derived model theorem for the uB-powerset. The proof is very much like the proof of the $\sf{Derived\ Model\ Theorem}$ (see \cite{StstattowfreeDMT}). Again, we do not use the terminology introduced in the previous section. Recall that $\mathcal{A}^\infty_h = (\powerset_{uB}(\Gamma^\infty))^{V[h]}$.

 \begin{theorem}\label{lem: over der model} Suppose $\k$ is a supercompact cardinal, there is a proper class of inaccessible limits of Woodin cardinals and $\l$ is an inaccessible limit of Woodin cardinals above $\kappa$. Suppose $h\subseteq \Col(\omega, {<}\lambda)$ is $V$-generic. Then $L(\mathcal{A}^\infty_{h})\models {\sf{AD^+}}$.
    \end{theorem}
    \begin{proof}
        It is enough to show that in $V[h]$, $\Gamma^\infty_{h}=\powerset(\bR)\cap L(\mathcal{A}_{h}^\infty)$. Towards a contradiction, suppose that there is an $\alpha'$ such that
        \[ \Gamma^\infty_{h}\not =\powerset(\bR)\cap L_{\alpha'+1}(\mathcal{A}_{h}^\infty). \]
        We work in $V[h]$.

        \medskip
        
        \noindent\textbf{The definition of $(\a_0, \phi_0, s_0)$.}\\
          Let $(\alpha_0, \phi_0, s_0)$ be the lex-least tuple $(\a, \phi, s)$ such that there is a pair $(A, B)$ with the following properties:

          \vspace{0.5em}\begin{enumerate}\itemsep0.5em
              \item[(1.1)] $\Gamma^\infty_{h}\not =\powerset(\bR)\cap L_{\alpha+1}(\mathcal{A}_{h}^\infty)$,
              \item[(1.2)] $s\in \a^{<\omega}$,
              \item[(1.3)] $A\in \mathcal{A}^\infty_{h}$,
              \item[(1.4)] $B\in \Gamma^\infty_{h}$,
              \item[(1.5)]\footnote{We assume that the formula $\phi$ is in the language that has a constant symbol for $\mathcal{A}_{h}^\infty$.}   letting $C$ be the set of $y\in \bR_h$ such that \[ L_{\alpha}(\mathcal{A}_{h}^\infty) \models \phi[A, B, s, y],\]
             then $C\not \in \Gamma^\infty_{h}$.
          \end{enumerate}\vspace{0.5em}

          \medskip
          
             \noindent\textbf{The definition of $(\a_1, \phi_1, s_1)$.}\\
             Let now $\lambda_1>\lambda$ be an inaccessible limit of Woodin cardinals, and let $(\alpha_1, \phi_1, s_1)$ be the lex-least tuple $(\b, \psi, t)$ such that there is a pair $(E, F)$ with the property that whenever $h_1\subseteq \Col(\omega, {<}\lambda_1)$ is $V[h]$-generic then

             \vspace{0.5em}\begin{enumerate}\itemsep0.5em
                 \item[(2.1)] $E, F\in \Gamma^\infty_{h}$,
                 \item[(2.2)] $t\in \beta^{<\omega}$,
                 \item[(2.3)] \footnote{The formula $\psi$ is in the language with a constant symbol for $\Gamma^\infty_{h*h_1}$.} letting $A$ be the set of $a$ such that 
             \[ L_{\beta}(\Gamma^\infty_{h*h_1}, \bR_{h*h_1})\models \psi[F_{h_1}, j_{h, h_1} \shortpwimg \Gamma^\infty_{h}, t, a],\footnote{Recall that $j_{h, h_1}\colon L(\Gamma_h^\infty,\bR_h) \rightarrow L(\Gamma_{h*h_1}^\infty,\bR_{h*h_1})$ denotes the canonical embedding with the property that $j_{h,h_1}(A) = A_{h_1}$ for $A \in \Gamma^\infty_h$. Notice that $j_{h, h_1} \shortpwimg \Gamma^\infty_{h}$ is independent of $h_1$.}\]
             and letting $C\in V[h]$ be defined by $y\in C$ if and only if 
              \[L_{\a_0}(\mathcal{A}_{h}^\infty)\models \phi_0[A, E, s_0, y],\]
              then $C\not \in \Gamma^\infty_{h}$.
             \end{enumerate}\vspace{0.5em}

            \medskip
              Thus, $(\a_1, \phi_1, s_1)$ defines the least possible member of $\mathcal{A}^\infty_h$ which can be used as a parameter to define the set of reals $C\not \in \Gamma^\infty_h$.\\\\
        \noindent\textbf{The definition of $(B_0, B_1)$.}\\
        Let $B_0, B_1\in \Gamma^\infty_{h}$ be such that

        \vspace{0.5em}\begin{enumerate}\itemsep0.5em
            \item[(3.1)] letting  $A$ be the set of $a$ such that for all $V[h]$-generic $h_1\subseteq \Col(\omega, {<}\l_1)$,
              \[ L_{\a_1}(\Gamma^\infty_{h*h_1}, \bR_{h_1})\models \phi_1[(B_1)_{h_1},  j_{h, h_1} \shortpwimg \Gamma^\infty_{h}, s_1, a], \]
             and letting $C\in V[h]$ be defined by $y\in C$ if and only if 
              \[ L_{\a_0}(\mathcal{A}_{h}^\infty)\models \phi_0[A, B_0, s_0, y], \]
              then $C\not \in \Gamma^\infty_{h}$.
        \end{enumerate}\vspace{0.5em}

        \medskip
              
        \noindent\textbf{The definition of $(\d_0, \d_1, W_0, W_1)$.}\\
        Let $\d_0<\d_1$ be two Woodin cardinals such that $\d_0>\l_1$. Let $p=(f, W_1, W_0)\in V_{\d_1}$ be a flipping system that flips through $\d_0$ such that $W_1$ is a $(\d_1, \l_1^+, \l_1)$-saturated set of measures (see Definition \ref{flipping system} and Definition \ref{def:saturated measures}).
  \medskip
        
        \noindent\textbf{The definition of $(\nu, \eta, \eta', \bar{\mu}^0, \bar{\mu}^1)$.}\\
      Let $\nu<\lambda$ be such that there are homogeneity systems $\bar{\mu}^0, \bar{\mu}^1\in V[h_\nu]$ such that

      \vspace{0.5em}\begin{enumerate}\itemsep0.5em
          \item[(4.1)] $\bar{\mu}^0\subseteq W_1$ and $\bar{\mu}^1\subseteq W_1$,
          \item[(4.2)] $(S_{\bar{\mu}^0_h})^{V[h]}=B_0$ and  $(S_{\bar{\mu}^1_h})^{V[h]}=B_1$.
      \end{enumerate}\vspace{0.5em}     
      Let $\eta$ be the least Woodin cardinal that is strictly greater than $\nu$ and let $\eta'$ be the least inaccessible above $\eta$.

      \medskip
      
       \noindent\textbf{The definitions of the term relations $\tau'$ and $\tau$.}\\
    In $V[h_\nu]$, we now have a term relation $\tau'$ which is always realized as $C$. First notice that in $V[h_\nu]$ we have a name $\dot{A}$ such that for any $V[h_\nu]$-generic $m\subseteq \Col(\omega, {<}\lambda)$, $\dot{A}_m$ is the set of $a$ such that whenever $h_1\subseteq \Col(\omega, {<}\l_1)$ is $V[h_\nu*m]$-generic, 
     \[ L_{\a_1}(\Gamma^\infty_{h_\nu*m*h_1}, \bR_{h_\nu*m*h_1})\models \phi_1[S_{\bar{\mu}^1}^{V[h_\nu*m*h_1]},  j_{h_\nu*m, h_1}\shortpwimg \Gamma^\infty_{h_\nu*m}, s_1, a].\footnote{Recall that $j_{h_\nu*m, h_1}\colon L(\Gamma_{h_\nu*m}^\infty,\bR_{h_\nu*m}) \rightarrow L(\Gamma_{h_\nu*m*h_1}^\infty,\bR_{h_\nu*m*h_1})$ denotes the canonical embedding with the property that $j_{h_\nu*m,h_1}(A) = A_{h_1}$ for $A \in \Gamma^\infty_{h_\nu*m}$.} \]
   Let now $\tau'$ be a name for a set of reals in $V[h_\nu]^{\Col(\omega, {<}\lambda)}$ such that $(p, \sigma)\in \tau'$ if and only if the following conditions hold:

   \vspace{0.5em}\begin{enumerate}\itemsep0.5em
       \item[(6.1)] $\sigma$ is a standard name for a real,
       \item[(6.2)] $p\in \Col(\omega, {<}\lambda)$,
       \item[(6.3)] whenever $m\subseteq \Col(\omega, {<}\lambda)$ is $V[h_\nu]$-generic with the property that $p\in m$,  
             \[ L_{\a_0}(\mathcal{A}_{h_\nu*m}^\infty)\models \phi_0[\dot{A}_m, S_{\bar{\mu}^0}^{V[h_\nu*m]}, s_0, \sigma_m]. \]
   \end{enumerate}\vspace{0.5em}
        We thus have that

        \vspace{0.5em}\begin{enumerate}\itemsep0.5em
            \item[(7.1)] for any $V[h_\nu]$-generic $m\subseteq \Col(\omega, {<}\lambda)$, $\tau_m'\not \in \Gamma^\infty_{h_\nu*m}$.
        \end{enumerate}\vspace{0.5em}
         We then let $\tau$ be the term relation for a set of reals in $V[h_\nu]^{\Col(\omega, \eta)}$ such that $(p, \sigma)\in \tau$ if and only if

         \vspace{0.5em}\begin{enumerate}\itemsep0.5em
             \item[(8.1)] $\sigma$ is a standard name for a real,
             \item[(8.2)] $p\in \Col(\omega, \eta)$,
             \item[(8.3)] $p$ forces that it is forced by $\Col(\omega, {<}\l)$ that $\sigma\in \tau'$.
         \end{enumerate}\vspace{0.5em}
         Let now $j: V\rightarrow M$ be an embedding (in $V$) witnessing that $\kappa$ is $\d_1$-supercompact. More precisely, $\cp(j)=\kappa$, $j(\kappa)>\d_1$ and $M^{\d_1}\subseteq M$. Let $\Sigma\in V[h_\nu]$ be the 
         set of $\T$ such that

         \vspace{0.5em}\begin{enumerate}\itemsep0.5em
             \item[(9.1)] $\T$ is a nice  iteration tree on $V_{\eta'}$ that is above $\nu$,
             \item[(9.2)] $\lh(\T)=\omega+1$,
             \item[(9.3)] all models of $j\T$, the copy of $\T$ on $M$, are well-founded.
         \end{enumerate}\vspace{0.5em}
             It follows from Windszus' Lemma (see Lemma \ref{lem:lemma 1.1} or \cite[Lemma 1.1]{StstattowfreeDMT}) that in $M[h_{\eta'}]$, $\Sigma$ is $j(\nu)$-homogeneously Suslin, and since $j(\nu)>\lambda_0$, we have that $\Sigma_{h}\in \Gamma^\infty_{h}$. The following is our main claim. It implies $C \in \Gamma^\infty_{h}$ as $C$ is defined by real quantification from $\Sigma_h \in \Gamma^\infty_h$ and will therefore finish the proof of Theorem \ref{lem: over der model}.
             
           \begin{claim}\label{defining c} In $V[h]$, for $x\in \bR_{h}$, $x\in C$ if and only if for every $\T\in \Sigma_h$ and for every $\M^\T[h_\nu]$-generic $m\subseteq \Col(\omega, \pi^\T(\eta))$,\footnote{Here, $\M^\T$ is the last model of $\T$ and $\pi^\T$ is the iteration embedding.} if $x\in \M^\T[h_\nu*m]$ then $x\in (\pi^\T(\tau))_m$.
           \end{claim}
           
           \begin{proof} Fix a real $x\in \bR_{h}$ and suppose $\T\in \Sigma_h$ and $m\subseteq \Col(\omega, \pi^\T(\eta))$ are such that
           
           \vspace{0.5em}\begin{enumerate}\itemsep0.5em
               \item[(10.1)] $m\in V[h]$ and $m$ is $\M^\T[h_\nu]$-generic,
               \item[(10.2)] $x\in \M^\T[h_\nu*m]$.
           \end{enumerate}\vspace{0.5em}
           We want to see that\\\\
           (*) $x\in C$ if and only if $x\in (\pi^\T(\tau))_m$.\\\\
           Because $\T\in \Sigma_h$, we have (in $V[h]$) an embedding $\sigma_0': \M^\T\rightarrow j(V_{\eta'})$ such that $j\restriction V_{\eta'}=\sigma_0'\circ \pi^\T$. Let now $\U$ be the result of applying $\T$ to $V_{\d_1}$\footnote{This means that we apply the extenders of $\T$ to $V_{\d_1}$ rather than $V_{\eta'}$.}. We then have $\sigma_0: \M^\U\rightarrow j(V_{\d_1})$ such that $\sigma_0'=\sigma_0\restriction \M^\T$ and $j\restriction V_{\d_1}=\sigma_0\circ \pi^\U$. Notice that $\M^\T=(V_{\pi^\U(\eta')})^{\M^\U}$. Set $N_0=\M^\U$ and $\pi_0=\pi^\U$. We then have that $(\pi_0, N_0, \sigma_0)$ is as in the hypothesis of Definition \ref{def: genericity iteration1}, and so we let $(\pi, N, \sigma)$ be an $\bR_{h}$-genericity iteration of $(\pi_0, N_0, \sigma_0)$ with the property that $\cp(\pi)>\pi_0(\eta')$. Let $m_1\subseteq \Col(\omega, {<}\l)$ be $N$-generic given by \cite[Lemma 3.1.5]{La04} such that

           \vspace{0.5em}\begin{enumerate}\itemsep0.5em
               \item[(11.1)] $\bR^{N[h_\nu*m*m_1]}=\bR_{h}$.
           \end{enumerate}\vspace{0.5em}
            Notice next that

            \vspace{0.5em}\begin{enumerate}\itemsep0.5em
                \item[(12.1)] $B_0=(S_{\pi\circ \pi_0(\bar{\mu}^0)})^{N[h_\nu*m*m_1]}$ and $B_1=(S_{\pi\circ \pi_0(\bar{\mu}^1)})^{N[h_\nu*m*m_1]}$\footnote{Notice that $\pi_0\circ \pi$ can be extended to act on $V[h_\nu]$ which is the model that contains $\bar{\mu}^0$ and $\bar{\mu}^1$.}.
            \end{enumerate}\vspace{0.5em}
            Recall that our goal is to show that (*) holds. Notice first that $\pi(\pi_0(\tau))=\pi_0(\tau)$. This is because $\cp(\pi)>\pi_0(\eta')$.  We now work in $N[h_\nu*m*m_1]$. In this model, $(\pi(\pi_0(\tau')))_{m*m_1}$ is realized as the set of all reals $y$ such that 
             \[ L_{\pi(\pi_0(\a_0))}(\mathcal{A}_{h_\nu*m*m_1}^\infty)\models \phi_0[(\pi(\pi_0(\dot{A})))_{m*m_1}, B_0, \pi(\pi_0(s_0)), y]. \]
         It follows from Lemmas \ref{lem: resurrection} and \ref{lem: resurrection1} that $\mathcal{A}_{h_\nu*m*m_1}^\infty=\mathcal{A}_{h}$ and $\pi(\pi_0(\dot{A})))_{m*m_1}=\dot{A}_h$. Given this, it follows from the minimality of $(\a_0, \phi_0, s_0)$ that

         \vspace{0.5em}\begin{enumerate}\itemsep0.5em
             \item[(13.1)] $\pi(\pi_0((\a_0, s_0)))=(\a_0, s_0)$,
             \item[(13.2)] $(\pi(\pi_0(\tau')))_{m*m_1}=C$.
         \end{enumerate}\vspace{0.5em}
         (*) now easily follows from (13.2). 
        \end{proof}
    \end{proof}

\subsection{The proof of Theorem \ref{sealing for the ub powerset}}\label{subsec:sealing for the ub powerset}

\begin{proof} We use the notation of Theorem \ref{sealing for the ub powerset}.
 We have that $\sf{Sealing}$ holds (see Theorem \ref{thm:SealingWoodin}). Suppose $\l_0<\l_1<\chi$ are inaccessible cardinals above $\k$ that are limits of Woodin cardinals. Let $W\in V_\chi$ be an infinitely tamed $(\chi, \l_1^+, \l_1)$-saturated set of measures, and let $X\prec V_{\chi+1}$ be such that $V_{\k+1}\cup \{\l_0, \l_1, W\}\subseteq X$ and $V\models \card{X}=\card{V_{\k+1}}$. Let $j: V\rightarrow N$ be a $\chi$-supercompactness embedding and let $h\subseteq \Col(\omega, {<}\l_0)$ be $V[g]$-generic. Then it follows from Theorem \ref{lem: over der model} that $L(\mathcal{A}^\infty_{g*h})\models {\sf{AD^+}}$. Let $k\subseteq \Col(\omega, \Gamma^\infty_{g*h})$ be $V[g*h]$-generic.
We first claim that the following holds:
\vspace{0.5em}\begin{enumerate}\itemsep0.5em
    \item[(a)] Whenever $H\subseteq \Col(\omega, {<}\k)$ is $V$-generic, $L(\mathcal{A}^\infty_H)\models {\sf{AD^+}}$.
\end{enumerate}\vspace{0.5em}
 The proof of (a) is very much like the proof of Theorem \ref{thm:dermodelrepomega}. Because of this, we will give a quick outline of the proof.
 
 It follows from Theorem \ref{thm:ubcapturingprinciple} and Theorem \ref{thm:exdmomega} that in $N[g*h*k]$, there is a $\sf{DM}$-sequence $D=((\sfa_\a, \sfb_\a, E_\a) \mid \a<\omega)$ for $V$ that is supported by $$(j[X\cap V_\chi], j[W], j(W), N, j(\chi), j(\k), \l, k).$$
 Let $${\sf{dm}}=(M_\a, \sigma_{\a, \b}, \sigma_\a, \k_\a:\a<\b<\omega)$$ be the underlying sequence of $D$, $M_\omega$ be the direct limit of ${\sf{dm}}$ and $\sigma_\omega: M_\omega \rightarrow N_{j(\chi)}$ be the canonical realizability embedding. As in the proof of Theorem \ref{thm:dermodelrepomega}, we can now find an $M_\omega$-generic $m\subseteq \Col(\omega, {<}\l_0)$\footnote{We have $\l_0=\k_\omega=\omega_1^{V[g*h]}$.} such that $(\Gamma^\infty)^{M_\omega[m]}=\Gamma^\infty_{g*h}$. We then claim that
\vspace{0.5em}\begin{enumerate}\itemsep0.5em
    \item[(b)] $(\mathcal{A}^\infty)^{M_\omega[m]}=\mathcal{A}^\infty_{g*h}$.
\end{enumerate}\vspace{0.5em}
Clearly, given Theorem \ref{lem: over der model}, (b) implies (a). To show (b) we apply the extenders $E_n$ of the $\sf{DM}$-sequence $D$ to $V$. More precisely, let $(K_i: i\leq \omega)$ be obtained by setting $K_0=V_{\chi}$ and letting $K_{i+1}=Ult(K_i, E_i)$ and $K_\omega$ be the direct limit of $(K_i, n_{i, i'}: i<i'<\omega)$ where $n_{i, i'}: K_i\rightarrow K_{i'}$ is the composition of ultrapower embeddings. Let $n_{i, \omega}$ be the direct limit embedding. The construction of the sequence $(K_i: i\leq \omega)$ is possible because $V_{\k+1}\subseteq M_0$, which by induction implies that for each $i\leq \omega$,
\vspace{0.5em}\begin{enumerate}\itemsep0.5em
    \item[(1.1)] $n_{0, i}(\k)=\k_i$,
    \item[(1.2)] $K_i|\k_i+1=M_i|\k_i+1$.
\end{enumerate}\vspace{0.5em}
We also get embeddings $q_i: K_i\rightarrow N|j(\chi)$ such that for $i<i'\leq \omega$, $q_i=q_{i'}\circ n_{i, i'}$. The embeddings exist because each $E_i$ is derivable from $q_i$. We start by setting $q_0=j\rest K_0$ and then apply Lemma \ref{lem:tworealizability}. Notice that $E_0$ is derivable from $q_0$ as witnessed by $(\sigma_1(\lh(E_0)), \sigma_1\rest \lh(E_0))$. We then inductively get that $E_i$ is derivable from $q_i$ as witnessed by $(\sigma_{i+1}(\lh(E_i)), \sigma_{i+1}\rest \lh(E_i))$, and this allows us to define $q_{i+1}$ by setting $q_{i+1}(\pi_{E_i}^{K_i}(f)(a))=q_i(f)(\sigma_{i+1}(a))$ where $a\in \lh(E)^{<\omega}$ and $f\in K_i$ is a function $f:[\k_i]^{\card{a}}\rightarrow K_i$. 

Notice that since $E_i$ is countable in $V[g*h]$, we have that $n_{0, \omega}(\l_1)=(\l_1)$. Notice next that Lemma \ref{lem: resurrection1} applies to $(\chi, \l_0, \l_1, n_{0, \omega}, K_\omega)$ with $\chi=\d_1$, $N=K_\omega$, $\pi=n_{0, \omega}$ and $\sigma=q_\omega$. This is because $(\Gamma^\infty)^{K_\omega[m]}=(\Gamma^\infty)^{M_\omega[m]}=\Gamma^\infty$, which is a consequence of the fact that $K_\omega|\k_\omega+1=M_\omega|\k_\omega+1$. It follows that $(\mathcal{A}^\infty)^{K_\omega[m]}=\mathcal{A}^\infty_{g*h}$. But since $K_\omega|\k_\omega+1=M_\omega|\k_\omega+1$, (b) follows, and this finishes the proof of (a).

Clause \eqref{eq: sealing for the ub powerset clause 2} of Definition \ref{def: sealing for the ub powerset} now follows from (a) because working in $V[g*r]$ where $r\subseteq \Col(\omega_1, \bR_g)$-generic and using Theorem \ref{thm:dermodelrepomega1} we can find $M_{\omega_1}$, which is a direct limit of some ${\sf{DM}}$-sequence for $V$, and an $M_{\omega_1}$-generic $s\subseteq \Col(\omega, {<}\k_{\omega_1})$ such that $(\mathcal{A}^\infty)^{M_{\omega_1}[h]}=\mathcal{A}^\infty_g$. (a) then implies that $L(\mathcal{A}^\infty_g)\models {\sf{AD^+}}$.

Clause \eqref{eq: sealing for the ub powerset clause 3} of Definition \ref{def: sealing for the ub powerset} uses both Theorem \ref{thm:dermodelrepomega} and Lemma \ref{lem: resurrection1}. Let $h$ and $h'$ be two consecutive $V[g]$-generics such that $V[g*h*h']\models \card{(2^{2^\omega})^{V[g*h]}}=\aleph_0$. Let $k\subseteq \Col(\omega, (2^{2^\omega})^{V[g*h*h']})$ be $V[g*h*h']$-generic. Working as we did while proving (a) we can find $\sf{DM}$-sequence $D=((\sfa_\a, \sfb_\a, E_\a) \mid \a<\omega \cdot 2)$ for $V[g*h*h']$ that is supported by $$(j[X\cap V_\chi], j[W], j(W), N, j(\chi), j(\k), \l, k),$$ and is such that $D\rest \omega=((\sfa_\a, \sfb_\a, E_\a) \mid \a<\omega)$ is a 
 $\sf{DM}$-sequence for $V[g*h]$\footnote{We can achieve this by applying Theorem \ref{thm:dermodelrepomega} to a surjection $k':\omega\cdot 2\rightarrow \Gamma^\infty_{g*h*h'}$ such that $k'\rest \omega$ enumerates $\Gamma^\infty_{g*h}$.}. Let $${\sf{dm}}=(M_\a, \sigma_{\a, \b}, \sigma_\a, \k_\a:\a<\b<\omega\cdot 2)$$ be the underlying sequence of $D$, $M_{\omega\cdot 2}$ be the direct limit of ${\sf{dm}}$ and $\sigma_{\omega\cdot 2}: M_{\omega\cdot 2} \rightarrow N_{j(\chi)}$ be the canonical realizability embedding.

Let now $G\subseteq \Col(\omega,{<}\omega_1^{V[g*h*h']})$ be $M_{\omega \cdot 2}$-generic such that 

\vspace{0.5em}\begin{enumerate}\itemsep0.5em
    \item[(2.1)] $G\in V[g*h*h'*k]$ and $G_{\kappa_\omega} = G\cap \Col(\omega, {<}\kappa_\omega)\in V[g*h*h']$\footnote{To get such a $G$, we factor $\Col(\omega, {<}\kappa_{\omega \cdot 2})$ as  $\Col(\omega, {<}\kappa_{\omega})\times \Col(\omega, {<}(\kappa_{\omega \cdot 2}, \kappa_\omega])$. Then we can find $G_0\in V[g*h*h']$ which is $M_{\omega \cdot 2}$-generic for $\Col(\omega, {<}\kappa_{\omega})$ and $G_1 \in V[g*h*h'*k]$ which is $M_{\omega \cdot 2}[G_0]$-generic for $\Col(\omega, {<}(\kappa_{\omega \cdot 2}, \kappa_{\omega}])$.},
    \item[(2.2)] $(\Gamma^\infty)^{M_\omega[G_{\kappa_\omega}]}=\Gamma^\infty_{g*h}$,
    \item[(2.3)] $(\Gamma^\infty)^{M_{\omega \cdot 2}[G]}=\Gamma^\infty_{g*h*h'}$.
\end{enumerate}\vspace{0.5em}
It follows from Lemma \ref{lem: resurrection1} that

\vspace{0.5em}\begin{enumerate}\itemsep0.5em
    \item[(3.1)] $(\mathcal{A}^\infty)^{M_\omega[G_{\kappa_\omega}]}=\mathcal{A}^\infty_{g*h}$,
    \item[(3.2)] $(\mathcal{A}^\infty)^{M_{\omega \cdot 2}[G]}=\mathcal{A}^\infty_{g*h*h'}$.
\end{enumerate}\vspace{0.5em}
Because we can lift $\pi_{\omega, \omega \cdot 2}$ to an embedding $\pi_{\omega, \omega \cdot 2}^+: M_\omega[G_{\kappa_\omega}]\rightarrow M_{\omega \cdot 2}[G]$, Clause \eqref{eq: sealing for the ub powerset clause 3} of Definition \ref{def: sealing for the ub powerset} follows from (2.1), (2.2), (2.3), (3.1), and (3.2).
 \end{proof}
 
 \section{A proof of Theorem \ref{thm: sealing for clubs}}\label{sec: sealing for clubs}

 \begin{proof} Let $j: V\rightarrow V^*$ be an embedding witnessing that $\k$ is an elementarily huge cardinal and such that ${\sf{ehst}}(\k)=\sup\{ j(f)(\k) \mid f\in V\}$. We will use the notation introduced in the statement of Theorem \ref{thm: sealing for clubs} (so $g\subseteq \Col(\omega, 2^{{\sf{ehst}}(\k)})$ is $V$-generic).
 
 The proof of Theorem \ref{thm: sealing for clubs} is an application of Theorem \ref{thm:dermodelrepomega1} and \cite[Theorem 0.1]{TrDMSpctM}. We start by proving the first part of the {\sf{Sealing Theorem}} in the sense of Theorem \ref{thm: sealing for clubs}. This means that we start by establishing Clauses \eqref{cl:1_def:st} and \eqref{cl:2_def:st} of Definition \ref{def: wst} for the model $L(\Gamma^\infty, \bR)[\mathcal{C}^\infty]$. The proof of Clause \eqref{cl:3_def:st} of Definition \ref{def: wst} follows the very same path as the proof of the similar clause in Theorem \ref{thm:SealingWoodin} (see Section \ref{sec:SealingTheorem}) and in Theorem \ref{sealing for the ub powerset} (see Subsection \ref{subsec:sealing for the ub powerset}). Therefore, we will only indicate how to prove it.
 
 To start, notice that $j(\k)$ satisfies the hypothesis of \cite[Theorem 0.1]{TrDMSpctM} (in $V^*$, with $j(\k)=\d_0$). Therefore, the following holds by the proof of \cite[Theorem 0.1]{TrDMSpctM}. 

 \vspace{0.5em}\begin{enumerate}\itemsep0.5em
     \item[(1.1)] Suppose for some $\l$, $F\in V^*$ is a $(j(\k), \l)$-extender and $i: V^*\rightarrow M=\Ult(V^*, F)$ is the ultrapower embedding given by $F$. Let $G\subseteq \Col(\omega, {<}j(\k))$ be $V^*$-generic and let, in $V^*[G]$, $\mathcal{F}\subseteq \powerset_{\omega_1}(\Gamma^\infty_G)$ be the filter given by $A\in \mathcal{F}$ if and only if for any $M[G]$-generic $H\subseteq \Col(\omega, {<}[j(\k), i(j(\k))))$\footnote{The conditions in this poset are finite functions $q:[j(\k), i(j(\k)))\times \omega\rightarrow i(j(\k))$ such that $q(\alpha, n)<\alpha$ for every $(\alpha, n)\in dom(q)$.}, letting $i^+ : V^*[G]\rightarrow M[G*H]$ be the lift-up of $i$, $i^+ \pwimg \Gamma^\infty_G \in i^+(A)$. Then, letting \[\mathcal{F}'=\mathcal{F}\cap L(\Gamma^\infty_G, \bR_G)[\mathcal{F}],\] 
     \begin{enumerate}\itemsep0.5em
         \item[(1.1.1)] $L(\Gamma^\infty_G, \bR_G)[\mathcal{F}]\models ``\mathcal{F}'$ is a normal, fine ultrafilter on $\powerset_{\omega_1}(\Gamma^\infty_G)"$, and
         \item[(1.1.2)] $\powerset(\bR)\cap L(\Gamma^\infty_G, \bR_G)[\mathcal{F}]=\Gamma^\infty_G$. 
     \end{enumerate}\vspace{0.5em}
    \end{enumerate}\vspace{0.5em}
 By elementarity, we have that the following holds in $V$. 

 \vspace{0.5em}\begin{enumerate}\itemsep0.5em
     \item[(2.1)] Suppose for some $\l$, $F\in V$ is a $(\k, \l)$-extender and $i: V\rightarrow M=\Ult(V, F)$ is the ultrapower embedding given by $F$. Let $G\subseteq \Col(\omega, {<}\k)$ be $V$-generic and let, in $V[G]$, $\mathcal{F}\subseteq \powerset_{\omega_1}(\Gamma^\infty_G)$ be the filter given by $A\in \mathcal{F}$ if and only if for any $M[G]$-generic $H\subseteq \Col(\omega, {<}[\k, i(\k)))$, letting $i^+ : V[G]\rightarrow M[G*H]$ be the lift-up of $i$, $i^+ \pwimg \Gamma^\infty_G \in i^+(A)$. Then, letting \[\mathcal{F}'=\mathcal{F}\cap L(\Gamma^\infty_G, \bR_G)[\mathcal{F}],\] 
     \begin{enumerate}\itemsep0.5em
         \item[(2.1.1)] $L(\Gamma^\infty_G, \bR_G)[\mathcal{F}]\models ``\mathcal{F}'$ is a normal, fine ultrafilter on $\powerset_{\omega_1}(\Gamma^\infty_G)"$, and
         \item[(2.1.2)] $\powerset(\bR)\cap L(\Gamma^\infty_G, \bR_G)[\mathcal{F}]=\Gamma^\infty_G$.
     \end{enumerate}\vspace{0.5em}
    \end{enumerate}\vspace{0.5em}
Notice now that because $V_{j(\k)}\prec V$, it is enough to prove that the conclusion of Theorem \ref{thm: sealing for clubs} holds in $V_{j(\k)}$, which is what we now start doing. 

 Let $F$ be the $(\k, j(\k))$-extender derived from $j$, and suppose $\d\in (\k, j(\k))$ is an inaccessible cardinal. Let $F_\d=F\restriction \d$. Notice now that 
\vspace{0.5em}\begin{enumerate}\itemsep0.5em
         \item[(3.1)] 
          if $\d$ is an inaccessible cardinal and $W\in V_\d$ is any set of measures that is tamed above $\k$, $F_\d$ is a wf-preserving extender for $W$.
     \end{enumerate}\vspace{0.5em}
    Also, notice that letting $\d={\sf{ehst}}(\k)$ and letting $k: Ult(V, F_\d)\rightarrow Ult(V, F)$ be the canonical factor embedding given by $k([a, f]_{F_\d})=[a, f]_F$, $\cp(k)=\d$.
 
Suppose now that $h$ is $V[g]$-generic for a poset $\mathbb{P}\in V_{j(\k)}[g]$. Let $\l<\chi_0<\chi_1<j(\k)$ be such that for $i\in 2$, $\chi_i$ is an inaccessible limit of Woodin cardinals of $V[g]$, and $\l>\d$ is an inaccessible cardinal such that $\mathbb{P}\in V_\l[g]$. Let $W\in V_{\chi_0}$ be a $(\chi_1, \l^+, \l)$-saturated set of measures that is infinitely tamed above $\l$, and let $X\prec V_{\chi_1+1}$ be such that $V_{\d+1}\cup\{\l, F_{\chi_0}, W\}\in X$ and $\card{X}=\card{V_{\d+1}}$.

We now use Theorem \ref{thm:dermodelrepomega1} in $V^*[g*h]$\footnote{Also, see Remark \ref{length omega1 case}.}. Let $\iota=\omega_1^{V[g*h]}$ and let $k\subseteq \Col(\iota, \card{\Gamma^\infty_{g*h}})$ be $V[g*h]$-generic. We have that since $k$ is generic for a countably closed poset, $\Gamma^\infty_{g*h*k}=\Gamma^\infty_{g*h}$. Let $$D=((\sfa_\a, \sfb_\a, E_\a) \mid \a<\iota)$$ be a ${\sf{DM}}$-sequence for $V$ supported by $(j[X\cap V_\chi], j[W], j(W), V^*, j(\chi), j(\k), \l, k)$ (see also Definition \ref{def:abstractblock} and Theorem \ref{thm:ubcapturingprinciple}). Let $${\sf{dm}}=(M_\a, \sigma_{\a, \b}, \sigma_\a, \k_\a:\a<\b<\iota)$$ be the underlying sequence of $D$, $M_{\iota}$ be the direct limit of ${\sf{dm}}$ and $\sigma_{\iota}: M_{\iota} \rightarrow U_\chi$ be the canonical realizability embedding. 

We introduce the following modification to the construction of the ${\sf{DM}}$ sequence we have introduced. Notice that in the proofs of Theorems \ref{thm:dermodelrepomega} and \ref{thm:dermodelrepomega1}, it is required that the extenders $E_\a$ witness wf-preservation for some set of measures $W$. (3.1) implies that $j(F_{\chi_0})$ witnesses wf-preservation for any set of measures $W$ in $j(V_{\chi_0})$, and so we require that $\sigma_\a(E_\a)=j(F_{\chi_0})$.

It follows from Theorem \ref{thm:dermodelrepomega1} that there is an $M_\iota$-generic $m\subseteq \Col(\omega, {<}\iota)$ such that $m\in V[g*h*k]$ and $\Gamma^\infty_{g*h}=(Hom^*)^{M_\iota[m]}$. 
Let 
\vspace{0.5em}\begin{enumerate}[(a)]\itemsep0.5em
\item $\l_0=\pi_0^{-1}({\sf{ehst}}(\k))$,
\item for $i\leq \iota$, $\l_i=\sigma_{0, i}(\l_0)$, 
\item for $i\leq \iota$, $U_i=E_i\restriction \l_i$, 
\item for $i\leq \iota$, $m_i=m\cap \Col(\omega, {<}\k_i)$, 
\item for $i\leq \iota$, $hom_i=(Hom^*)^{M_i[m_i]}$,\footnote{We thus have that $hom_i=(\Gamma^\infty)^{M_i[m_i]}$.}
\item for $i\leq \iota$, $\tau_i=_{def}\pi_{U_i}^{M_i}: M_i\rightarrow \Ult(M_i, U_i)$,
\item for $i\leq \iota$, let $\mathcal{F}_i\subseteq (\powerset_{\omega_1}(Hom^*))^{M_i[m_i]}$ be the filter derived from $\sigma_i$. More precisely, $A\in \mathcal{F}_i$ if and only if for every $\Ult(M_i, U_i)[m_i]$-generic $$n\subseteq \Col(\omega, {<}[\kappa_i, \tau_i(\kappa_i))),$$ $\tau_i^+ \pwimg hom_i \in \tau_i^+(A)$ where $\tau^+_i: M_i[m_i]\rightarrow \Ult(M_i, U_i)[m_i*n]$ is the lift-up of $\tau_i$,
\item  for $i\leq \iota$, in $M_i[m_i]$, $\mathcal{F}'_i=\mathcal{F}_i\cap L(\Gamma^\infty, \bR)[\mathcal{F}_i]$.
\end{enumerate}\vspace{0.5em}
It follows from (2.1) that

\vspace{0.5em}\begin{enumerate}\itemsep0.5em
    \item[(4.1)] for $i\leq \iota$, in $M_i[m_i]$, $L(\Gamma^\infty, \bR)[\mathcal{F}_i]\models ``\mathcal{F}'_i$ is a normal, fine ultrafilter on $\powerset_{\omega_1}(\powerset(\bR))"$,
    \item[(4.2)] for $i\leq \iota$, in $M_i[m_i]$, $\powerset_{\omega_1}(\bR)\cap L(\Gamma^\infty, \bR)[\mathcal{F}_i]=\Gamma^\infty$.
\end{enumerate}\vspace{0.5em}
Because $U_i$ is a sub-extender of $E_i$ for $i\leq \iota$, there is $\tau'_i: \Ult(M_i, U_i)\rightarrow M_{i+1}$ with the property that $\tau_i'\circ \tau_i=\sigma_{i, i+1}$. Moreover, $\cp(\tau_i')=\l_i$.\footnote{This fact is key in (5.3).} Let now for $i<i'\leq \iota$, $\sigma_{i, i'}^+: M_i[m_i]\rightarrow M_{i'}[m_{i'}]$ be the lift-up of $\sigma_{i,i'}$. We now have that

\vspace{0.5em}\begin{enumerate}\itemsep0.5em
    \item[(5.1)] $\sigma_{i, i'}^+(\mathcal{F}_i)=\mathcal{F}_{i'}$,
    \item[(5.2)] if $A\in \mathcal{F}_i$ then $\sigma^+_{i, i+1} \pwimg hom_i \in \sigma^+_{i, i+1}(A)$,\footnote{This is because we can lift $\tau_i'$ to $\tau_i'^+: \Ult(M_i, U_i)[m\cap \Col(\omega, {<}\l_i)]\rightarrow M_{i+1}[m_{i+1}]$.}
    \item[(5.3)] if $A\in \mathcal{F}_i$ then for every $i'\in (i, \iota]$, $\sigma^+_{i, i'} \pwimg hom_i \in \sigma^+_{i, i'}(A)$,
    \item[(5.4)] if $A\in \mathcal{F}_\iota$ then for some $i<\iota$, $\{ \sigma^+_{i, \iota} \pwimg hom_i \mid i<\iota\}\subseteq A$.\footnote{This follows from (5.1), (5.3) and the fact that $\iota$ is a limit ordinal.}
\end{enumerate}\vspace{0.5em}
It then follows from (4.1) and (5.4) that

\vspace{0.5em}\begin{enumerate}\itemsep0.5em
    \item[(6.1)] in $V[g*h*H]$, $\mathcal{F}_\iota'=\mathcal{C}^\infty\cap M_\iota[m]$.
\end{enumerate}\vspace{0.5em}
Therefore,

\vspace{0.5em}\begin{enumerate}\itemsep0.5em
    \item[(7.1)] in $V[g*h*H]$, letting $\mathcal{C}'=\mathcal{C}^\infty\cap L(\Gamma^\infty, \bR)[\mathcal{C}^\infty]$, we have that
    \vspace{0.5em}\begin{enumerate}\itemsep0.5em
        \item[(7.1.1)] $L(\Gamma^\infty, \bR)[\mathcal{C}^\infty]\models ``\mathcal{C}'$ is a normal, fine ultrafilter on $\powerset_{\omega_1}(\powerset(\bR))$" and
        \item[(7.1.2)] $\powerset(\bR)\cap L(\Gamma^\infty, \bR)[\mathcal{C}^\infty]=\Gamma^\infty$.
    \end{enumerate}\vspace{0.5em}
    \end{enumerate}\vspace{0.5em}
Notice next that $L(\Gamma^\infty_{g*h}, \bR_{g*h})[\mathcal{C}^\infty_{g*h}]=L(\Gamma^\infty_{g*h}, \bR_{g*h})[\mathcal{C}^\infty_{g*h*k}]$. This is because if $A\in V[g*h]$, $A\subseteq \powerset_{\iota}(\Gamma^\infty_{g*h})$ and $A$ is stationary in $V[g*h]$ then $A$ is stationary in $V[g*h*k]$.\footnote{This follows from the fact that countably closed posets are proper.} Therefore, we have that
\vspace{0.5em}\begin{enumerate}\itemsep0.5em
    \item[(8.1)] in $V[g*h]$, letting $\mathcal{C}'=\mathcal{C}^\infty\cap L(\Gamma^\infty, \bR)[\mathcal{C}^\infty]$, 
    \vspace{0.5em}\begin{enumerate}\itemsep0.5em
        \item[(8.1.1)] $L(\Gamma^\infty, \bR)[\mathcal{C}^\infty]\models ``\mathcal{C}'$ is a normal, fine ultrafilter on $\powerset_{\omega_1}(\powerset(\bR))$" and
        \item[(8.1.2)] $\powerset(\bR)\cap L(\Gamma^\infty, \bR)[\mathcal{C}^\infty]=\Gamma^\infty$.
    \end{enumerate}\vspace{0.5em}
    \end{enumerate}\vspace{0.5em}
This finishes the proof of the first part of the $\sf{Sealing\ Theorem}$ in the sense of Theorem \ref{thm: sealing for clubs}. It remains to show that if $h'$ is $V[g*h]$-generic then there is an elementary embedding $$j: L(\Gamma^\infty_{g*h}, \bR_{g*h})[\mathcal{C}^\infty_{g*h}]\rightarrow L(\Gamma^\infty_{g*h*h'}, \bR_{g*h*h'})[\mathcal{C}^\infty_{g*h*h'}]$$
such that $j(\mathcal{C}^\infty_{g*h})=\mathcal{C}^\infty_{g*h*h'}$. What we have shown is the following (we use the objects introduced above). Using the above notation, given $d\in L(\Gamma^\infty_{g*h}, \bR_{g*h})[\mathcal{C}^\infty_{g*h}]$ and $i<\iota$ such that $d\in \rng(\sigma_{i, \iota}^+)$, we let $d_i$ be the $\sigma_{i, \iota}$-preimage of $d$. Given sets $a, b, c$, we let $((a, b)[c])^\#$ be the sharp of $L(a, b)[c]$.\footnote{The language of such a sharp contains constants for the members of $a$ and $b$ and another constant symbol for $c\cap L(a, b)[c]$.} Let $\mathcal{C}'_{g*h}=\mathcal{C}^\infty_{g*h}\cap L(\Gamma^\infty_{g*h}, \bR_{g*h})[\mathcal{C}^\infty_{g*h}]$.
\vspace{0.5em}\begin{enumerate}\itemsep0.5em
        \item[(9.1)] For $i<\iota$, $\sigma_{i, \iota}^+(hom_i)=\Gamma^\infty_{g*h}$ and $\sigma_{i, \iota}^+(\mathcal{F}'_{i})=\mathcal{C}'_{g*h}$.
        \item[(9.2)] Suppose $A\in \Gamma^\infty_{g*h}$, $\phi$ is a formula, and $(c_0,..., c_n)$ is a sequence of constants for indiscernibles. Then $$(\phi, A, \mathcal{C}'_{g*h}, c_0,..., c_n)\in ((\Gamma^\infty_{g*h}, \bR_{g*h})[\mathcal{C}^\infty_{g*h}])^\#$$  if and only if whenever $i<\iota$ is such that $A\in \rng(\sigma_{i, \iota}^+)$, $$(\phi, A_i, \mathcal{F}_i', c_0, ..., c_n)\in (((hom_i, \bR^{M_i[m_i]})[\mathcal{F}_{i}'])^\#)^{M_i[m_i]}.$$
    \end{enumerate}\vspace{0.5em}
Using (9.1) and (9.2) it is easy to establish\footnote{As we have done in Theorem \ref{thm:SealingWoodin} (see Section \ref{sec:SealingTheorem}) and in Theorem \ref{sealing for the ub powerset} (see Subsection \ref{subsec:sealing for the ub powerset})} that the following holds:
\vspace{0.5em}\begin{enumerate}\itemsep0.5em
        \item[(10)] Suppose $A\in \Gamma^\infty_{g*h}$, $\phi$ is a formula, and $(c_0,..., c_n)$ is a sequence of constants for indiscernibles. Then $$(\phi, A, \mathcal{C}'_{g*h}, c_0,..., c_n)\in ((\Gamma^\infty_{g*h}, \bR_{g*h})[\mathcal{C}^\infty_{g*h}])^\#$$  if and only if $$(\phi, A_{h'}, \mathcal{C}'_{g*h*h'}, c_0,..., c_n)\in ((\Gamma^\infty_{g*h*h'}, \bR_{g*h*h'})[\mathcal{C}^\infty_{g*h*h'}])^\#$$ 
    \end{enumerate}\vspace{0.5em}
Proving (10) contains no new insight and follows the same set up of the two theorems mentioned above, and because of this we leave the details to the reader. 
\end{proof}

\bibliographystyle{plain}
\bibliography{References}

\end{document}